\documentclass[a4paper,10pt]{article}

\usepackage[a4paper,textwidth=17cm,textheight=23.5cm,centering]{geometry}

\usepackage{enumitem}

\usepackage[utf8]{inputenc}
\usepackage[T1]{fontenc}
\usepackage{xcolor,graphicx}
\usepackage{eepic}
\usepackage[english]{babel}
\usepackage{amsfonts, amssymb, amsmath, amsthm}

\usepackage{mathrsfs}
\usepackage{cancel}

\usepackage{mathtools}
\mathtoolsset{showonlyrefs}

\usepackage{amsmath}
\numberwithin{equation}{section}

\usepackage{graphicx}
\usepackage{color}
\usepackage{graphicx}
\usepackage[colorinlistoftodos]{todonotes}
\usepackage[colorlinks=true, allcolors=blue]{hyperref}

\usepackage{array} 
\usepackage{amsfonts}
\usepackage{amsmath}
\usepackage{amssymb}
\usepackage{graphicx}
\usepackage{float} 
\usepackage{color}
\usepackage{esint}    

\usepackage{enumitem} 

\usepackage{multirow}

\newtheorem{thm}{Theorem}[section]
\newtheorem{cor}[thm]{Corollary}
\newtheorem{lem}[thm]{Lemma}

\newtheorem{defn}[thm]{Definition}
\newtheorem{rem}[thm]{Remark}

\numberwithin{equation}{section}

\newcommand{\dx}{\,{\rm d}x}
\newcommand{\dy}{\,{\rm d}y}

\def\LL{\mathrm{L}} 
\def\supp{\mathrm{supp}} 

\newcommand{\ue}{u_\epsilon}

\newcommand{\plap}{\Delta_p}
\newcommand{\plapN}{\Delta_p^{N}}

\newcommand{\gr}{\nabla}
\newcommand{\ph}{\varphi}

\newcommand{\intn}{\fint}

\DeclareMathOperator{\sgn}{sgn}

\newcommand{\A}{\mathcal{A}}
\newcommand{\ua}{\overline{u}}

\newcommand{\ub}{\underline{u}}

\newcommand{\wb}{\underline{w}}
\newcommand{\wa}{\overline{w}}

\newcommand{\RR}{\mathbb{R}}

\newcommand{\Rd}{\mathbb{R}^d}

\newcommand{\NN}{\mathbb{N}}

\newcommand{\beps}{\overline{\epsilon}}

\def\dist{\mathrm{dist}} 
\def\diam{\mathrm{diam}} 

\def\qed{\,\unskip\kern 6pt \penalty 500
	\raise -2pt\hbox{\vrule \vbox to8pt{\hrule width 6pt
			\vfill\hrule}\vrule}\par}

\def\quotient#1#2{\raise1ex\hbox{$#1$}\Big/\lower1ex\hbox{$#2$}}

\definecolor{darkblue}{rgb}{0.05, .05, .65}
\definecolor{darkgreen}{rgb}{0.1, .65, .1}
\definecolor{darkred}{rgb}{0.8,0,0}

\usepackage{tocloft}

\begin{document}

	\title{\bf A game-theoretical interpretation \\ for a doubly nonlinear parabolic equation}
	
	\author{\Large F\'{e}lix del Teso, 
		Carlos Fuertes Mor\'{a}n 
		~and~ Julio D. Rossi 
	}
	\date{} 
	
	\maketitle
	
	
	\begin{abstract}
		We introduce a game--theoretical framework for the doubly nonlinear parabolic equation
		\[
		|\partial_t u|^{p-2} \partial_t u - \Delta_p u = 0.
		\]
		where $\Delta_p u = \nabla \cdot ( |\nabla u |^{p-2} \nabla u)$ with $p>2$ is the standard $p-$Laplacian.
		A key feature to our approach is a new asymptotic mean value formula (AMVF) for the $p-$Laplacian that is robust even when the gradient vanishes and is independent of the sign of the $p-$Laplacian. This new AMVF leads naturally to a dynamic programming principle (DPP) whose solutions converge to the viscosity solution of the boundary value problem for the differential equation. In addition, solutions to the DPP coincide with value functions for a stochastic, two-players, zero-sum game that we introduce and analyze here. 
	\end{abstract}
	
	

	
	
	\smallskip\noindent {\sl\small\copyright~2026 by the authors. This paper may be reproduced, in its entirety, for non-commercial purposes.}
	
	
	\tableofcontents

	
	
	\section{Introduction}\label{sec1}

	The main goal of this paper is to provide a game--theoretical interpretation for the doubly nonlinear parabolic equation
	\begin{equation}\label{equation.intro}
		\left| \partial_t u \right|^{p-2}\partial_t u (x,t)- \plap u (x,t) = 0,
	\end{equation}
	where $p>2$ and $\plap$ denotes the $p$-Laplacian operator. For a smooth function $\ph$, the $p$-Laplacian is given by
	\[
	\plap \ph = \nabla \cdot \big( |\nabla \ph |^{p-2} \nabla \ph \big).
	\]
	
	Model \eqref{equation.intro} is doubly nonlinear: the diffusion is driven by the $p$-Laplacian, while the time derivative appears in the equation through the nonlinear term $|\partial_t u|^{p-2}\partial_t u$. This places the problem within the class of doubly nonlinear parabolic equations, which has been extensively studied. The works \cite{Akagi_2023,Mielke_Rossi_Savare_2013} address the existence of solutions to the Cauchy problem in a general framework, while \cite{Akagi_Schimperna_2024} establishes existence in bounded domains. Equation~\eqref{equation.intro} is studied in detail in bounded domains in \cite{Lindgren_Hynd_2016}, where the large-time behavior of solutions is analyzed. In particular, it is shown that the Rayleigh quotient computed along solutions converges to the optimal constant in Poincar\'e's inequality as $t$ goes to infinity (see also \cite{Hynd_Lindgren_2017}, where these quotients are used to approximate eigenvalues and eigenvectors). Moreover, in the previously mentioned reference, a comparison principle for viscosity solutions is established.
	This fact is quite useful for our arguments.

	We study the equation \eqref{equation.intro} in two different settings. On the one hand, we consider the Dirichlet problem in bounded domains. On the other hand, we analyze the Cauchy problem in the whole space $\RR^d$. In both cases, our goal is to show that viscosity solutions to \eqref{equation.intro} can be characterized
	as limits of value functions for an appropriate stochastic game, thus extending known connections between nonlinear partial differential equations and game theory to this doubly nonlinear parabolic framework.

	Now, let us summarize the main contributions of this paper.
	\begin{enumerate}[label=\noindent(\roman*)] 
		\item We introduce a new asymptotic mean value formula (AMVF) for the $p$-Laplacian tailored to a game--theoretical interpretation. This new formula, partially inspired by \cite{delTeso_Rossi_2024}, offers several advantages over previous ones:
		
		\begin{enumerate}[label= $\bullet$]
			\item The AMVF in \cite{delTeso_Rossi_2024} requires knowing a priori the sign of $\plap \ph$. While sufficient for elliptic problems of the form $-\plap u = f$, in the parabolic problem \eqref{equation.intro} the sign of the $p$-Laplacian depends on the sign of the time derivative, which is unknown beforehand. The new AMVF is more robust, providing a single algebraic expression independently of this sign.
			
			\item The formula in \cite{delTeso_Rossi_2024} fails at points where $\nabla \ph = 0$. Our new AMVF remains valid in this case, which is particularly useful when analyzing the associated Dynamic Programming Principle (DPP) and proving convergence to viscosity solutions of the parabolic problem.
		\end{enumerate}
		
		\item Using the new AMVF for the $p$-Laplacian, we derive a corresponding AMVF for the parabolic problem \eqref{equation.intro}. Here let us remark that we 
		deal with a $p-1$ homogeneous doubly nonlinear problem.
		
		\item We characterize viscosity solutions of problem \eqref{equation.intro} via the viscosity formulation of the AMVF.
		
		\item Based on the parabolic AMVF, we formulate and analyze the associated DPP for problem \eqref{equation.intro}, both in bounded domains and in $\RR^d$.
		We show existence and uniqueness for the boundary value problem 
		for the DPP. We also obtain some qualitative properties (like a comparison principle) for solutions to the DPP.
		
		\item We also prove that solutions of the DPP converge to the solution of the corresponding PDE problem.

		\item Finally, we describe the game induced by asymptotic mean value operator. We prove, using probabilistic arguments, that the game has a value that coincides with the solution to the DPP. Here we extend results contained in \cite{MPRparabolico} (see also the book \cite{libroBR}).
	\end{enumerate}

	\subsection{Heuristic derivation of the AMVF}\label{sec:heuristic}
	For the reader's convenience, we describe the underlying ideas leading to the AMVF in a simplified setting. Recall that if $x \in \mathbb{R}^d$ and $\nabla \varphi(x) \neq 0$, then
	\begin{align}\label{eq:plap-nplap}
		\plap \varphi(x)
		=
		|\nabla \varphi(x)|^{p-2}
		\Big(
		\Delta \varphi(x)
		+
		(p-2)
		\Big\langle
		D^2 \varphi(x)\frac{\nabla \varphi(x)}{|\nabla \varphi(x)|},
		\frac{\nabla \varphi(x)}{|\nabla \varphi(x)|}
		\Big\rangle
		\Big)
		\eqqcolon
		|\nabla \varphi(x)|^{p-2} \plapN \varphi(x),
	\end{align}
	where $\plapN \varphi (x)$ denotes the so-called normalized $p$-Laplacian, see
	\cite{MPRparabolico} and the book \cite{libroBR}.  To present our main ideas avoiding technical issues at critical points we consider, 
	as a model case,
	\[
	\mathcal{L}[\varphi](x) = |\nabla \varphi(x)|\,\Delta \varphi(x),
	\]
	which corresponds to \eqref{eq:plap-nplap} with $p=3$, but without the term
	$
	\langle
	D^2 \varphi \frac{\nabla \varphi}{|\nabla \varphi|},
	\frac{\nabla \varphi}{|\nabla \varphi|}
	\rangle.
	$ 
	For this simplified version we will use the classical asymptotic expansions
	\begin{align}\label{eq:ASMFintro}
		|\nabla \varphi(x)|
		&\sim
		\frac{1}{r}\Big( \sup_{B_r(x)} \varphi - \varphi(x) \Big),
		\qquad
		\Delta \varphi(x)
		\sim
		\frac{1}{\rho^{2}}\Big( \fint_{B_{\gamma\rho}(x)} \varphi(y)\,dy - \varphi(x) \Big),
	\end{align}
	where $\gamma=\sqrt{4+2d}$. We will also use the notations 
	$$a_+ = \max\{a,0\}\quad \mbox{and} \quad a_- = \max\{-a,0\}
	\quad \mbox{so that} \quad a=a_+-a_-,$$ and for $\beta>0$ we denote the signed power as $$(a)^\beta = \operatorname{sgn}(a)\,|a|^\beta.$$
	
	To continue with the analysis of the model equation, assume first that $\Delta \varphi(x)\ge 0$. Using the identity 
	$a^{1/2}b^{1/2}
	=
	\inf_{c>0}
	\Big\{
	\frac{1}{2}c^{1/2}a
	+
	\frac{1}{2}c^{-1/2}b
	\Big\}$ for $a,b\ge0$, 
	we obtain
	\begin{equation}\label{heuristic-1}
		\big(\mathcal{L}[\varphi](x)\big)^{1/2}
		=
		|\nabla \varphi(x)|^{1/2}
		\big(\Delta \varphi(x)\big)^{1/2}
		=
		\inf_{c>0}
		\Big\{
		\frac{c^{1/2}}{2}|\nabla \varphi(x)|
		+
		\frac{c^{-1/2}}{2}\Delta \varphi(x)
		\Big\}.
	\end{equation}
	Although this identity holds only when $\mathcal{L}[\varphi](x)\ge0$, we always have
	\[
	\big(\mathcal{L}[\varphi](x)^{1/2}\big)_+
	=
	\Big(
	\inf_{c>0}
	\Big\{
	\frac{c^{1/2}}{2}|\nabla \varphi(x)|
	+
	\frac{c^{-1/2}}{2}\Delta \varphi(x)
	\Big\}
	\Big)_+ .
	\]
	Indeed, if $\Delta \varphi(x)\ge0$, this reproduces \eqref{heuristic-1}; if $\Delta \varphi(x)<0$, the infimum equals $-\infty$, so both sides vanish. Similarly,
	\[
	\big(\mathcal{L}[\varphi](x)^{1/2}\big)_-
	=
	\Big(
	\inf_{c>0}
	\Big\{
	\frac{c^{1/2}}{2}|\nabla \varphi(x)|
	+
	\frac{c^{-1/2}}{2}\big(-\Delta \varphi(x)\big)
	\Big\}
	\Big)_- .
	\]
	Fix $\epsilon>0$ and choose $r=\epsilon^2 c^{1/2}$ and $\rho=\epsilon c^{-1/2}$ in \eqref{eq:ASMFintro}, then 
	\begin{align*}
		\left(\mathcal{L}[\ph](x)\right)_+  &\sim \Big(\inf_{c>0} \Big\{  \frac{\sup_{B_{\epsilon^2 c^{1/2}}(x)}\ph-\ph(x)}{2\epsilon^2} + \frac{\fint_{\gamma\epsilon c^{-1/4}}\ph(y)\dy-\ph(x)}{2\epsilon^2} \Big\}\Big)_+\\
		&=\frac{1}{\epsilon^2} \Big( \max \Big[ \inf_{c>0} \Big\{ \frac{1}{2} \sup_{B_{\epsilon^2 c^{1/2}}(x)}\ph + \frac{1}{2} \fint_{\gamma\epsilon c^{-1/4}}\ph(y)\dy \Big\},\, \ph(x)\Big]-\ph(x)\Big)
	\end{align*}
	and
	\begin{align*}
		\left(\mathcal{L}[\ph](x)\right)_- &\sim \Big(\inf_{c>0} \Big\{ \frac{\ph(x)-\inf_{B_{\epsilon^2 c^{1/2}}(x)}\ph}{2\epsilon^2} +  \frac{\ph(x)-\fint_{\gamma\epsilon c^{-1/4}}\ph(y)\dy}{2\epsilon^2} \Big\}\Big)_-\\
		&=- \frac{1}{\epsilon^2}\Big(\min \Big[\sup_{c>0} \Big\{ \frac{1}{2} \sup_{B_{\epsilon^2 c^{1/2}}(x)}\ph + \frac{1}{2} \fint_{\gamma\epsilon c^{-1/4}}\ph(y)\dy \Big\},\, \ph(x)\Big]-\ph(x)\Big).
	\end{align*}
	Here we used $\inf\{-a\}=-\sup\{a\}$, $(-a)_-=a_+$, and $\max\{a-c,0\}=\max\{a,c\}-c$. Therefore,\normalcolor
	\begin{align*}
		\big(\mathcal{L}[\varphi](x)\big)^{1/2}
		&=
		\big(\mathcal{L}[\varphi](x)^{1/2}\big)_+
		-
		\big(\mathcal{L}[\varphi](x)^{1/2}\big)_- \sim
		\frac{\A_\epsilon[\varphi](x)-\varphi(x)}{\epsilon^2/2},
	\end{align*}
	where
	\begin{equation}\label{operator-A-1/2}
		\begin{split}
			\A_\epsilon[\varphi](x)
			&=\frac{1}{2}\max\Big[
			\inf_{c\in[m,M]}
			\Big\{
			\frac{1}{2}\sup_{B_{\epsilon^2 c^{1/2}}(x)}\varphi
			+
			\frac{1}{2}\fint_{B_{\gamma \epsilon c^{-1/2}}(x)}\varphi(y)\,dy
			\Big\},
			\varphi(x)
			\Big] \\
			&\quad+
			\frac{1}{2}\min\Big[
			\sup_{c\in[m,M]}
			\Big\{
			\frac{1}{2}\inf_{B_{\epsilon^2 c^{1/2}}(x)}\varphi
			+
			\frac{1}{2}\fint_{B_{\gamma \epsilon c^{-1/2}}(x)}\varphi(y)\,dy
			\Big\},
			\varphi(x)
			\Big].
		\end{split}
	\end{equation}
	The restriction $c\in [m,M]$ is required for a rigorous justification of the asymptotics.
	We choose $m$ and $M$ according to \eqref{m-M} and observe that $m(\epsilon)\to0^+$ and $M(\epsilon)\to\infty$ as $\epsilon\to0^+$.
	
	Finally, we observe that the previous AMVP also extends to the parabolic setting
	$$
	\big(\partial_t \phi(x,t)\big)^2
	=
	|\nabla \phi(x,t)|\,\Delta \phi(x,t).
	$$
	Indeed, using the expansion 
	$$\partial_t \phi(x,t)
	\sim
	\frac{\phi(x,t+\tau)-\phi(x,t)}{\tau},$$
	with $\tau=\epsilon^2/2$, together with \eqref{operator-A-1/2} for any $t>0$ fixed, we obtain 
	$$
	\partial_t \phi(x,t)-\big(\mathcal{L}[\phi](x,t)\big)^{1/2}
	\sim
	\frac{\phi(x,t+\epsilon^2/2)-\A_\epsilon[\phi](x,t)}{\epsilon^2/2}.
	$$
	
	For the general case of \eqref{eq:plap-nplap}, we use 
	$$
	\frac12 \sup_{B_{\gamma \rho}(x)}\ph + \frac12 \inf_{B_{\gamma \rho}(x)}\ph - \ph (x) \sim 
	\langle
	D^2 \varphi(x)\frac{\nabla \varphi(x)}{|\nabla \varphi(x)|},
	\frac{\nabla \varphi(x)}{|\nabla \varphi(x)|} \rangle,
	$$
	that comes from the analysis of Tug-of-War games in \cite{PSSW}, to obtain approximations to
	$$(
	|\nabla \varphi(x)|^{p-2} \plapN \varphi(x))^{\frac{1}{(p-1)}}  = \left(
	|\nabla \varphi(x)|^{p-2}
	\Big(
	\Delta \varphi(x)
	+
	(p-2)
	\Big\langle
	D^2 \varphi(x)\frac{\nabla \varphi(x)}{|\nabla \varphi(x)|},
	\frac{\nabla \varphi(x)}{|\nabla \varphi(x)|}
	\Big\rangle
	\Big) \right)^{\frac{1}{(p-1)}}.
	$$
	
	The obtained AMVF can also be used to handle the parabolic equation,
	$$
	\partial_t \phi(x,t)- (\plap \phi(x,t))^{\frac{1}{(p-1)}} = 
	\partial_t \phi(x,t)- (
	|\nabla \varphi(x)|^{p-2} \plapN \phi(x,t))^{\frac{1}{(p-1)}}
	$$
	that is our main goal in this paper.

	\subsection{Comments on related literature}
	
	Let us now refer briefly to some related previous results. 
	Asymptotic mean value formulas for the classical Laplacian can be found in
	\cite{Blaschke,Ku,Privaloff}, for other linear elliptic operators in \cite{Littman.et.al1963}, and
	for degenerate elliptic equations in \cite{BL12}.
	In addition, asymptotic mean value formulas were found for nonlinear operators
	such as the normalized $p$-Laplacian.
	These mean value formulas come from the connection 
	between probability (via the
	dynamic programming principle for Tug-of-War games) 
	and the normalized infinity Laplacian, see \cite{LeGruyerArcher,LeGruyer,PSSW}. A 
	nonlinear mean value property for $p$-harmonic functions first appeared in \cite{Manfredi_Parviainen_Rossi_2010},
	while for the non-homogenous equation we refer to \cite{delTeso_Rossi_2024}. 
	Concerning mean value formulas for
	the heat equation we also refer to \cite{Watson} and \cite{Aimar}. 
	For a version with variable coefficients, see \cite{Fabes}.
	For asymptotic nonlinear mean value formulas for two diﬀerent parabolic versions of the Monge-Ampere equation
	we quote \cite{BChMR} and for parabolic equations given in terms of eigenvalues of the Hessian to \cite{BER}.
	
	See also the books \cite{libroBR,Lewicka_book} for further related references and other equations.

	\section{Preliminaries and main results}\label{sec:mainres}
	The aim of this section is to state the main results contained in this paper. 
	
	We rely on several results existing in the literature, which we list below.
	We will need the following classic asymptotic formulas for the modulus of the gradient.

	\begin{lem}\label{exp-gradient}
		Let $d\geq 1$, $x\in\Rd$, $R>0$, and $\ph \in C^2(B_R(x))$. Then, for all $r\in (0,R/2)$, we have
		\begin{equation}\label{gr-1}
			\ph(x)-\inf_{B_r(x)}\ph = r \, | \gr \ph(x)| + r \, E_1(r,\ph,x), \quad \textup{and} \quad
			\sup_{B_r(x)}\ph - \ph(x) = r \, | \gr \ph(x)| + r \, E_1(r,\ph,x),
		\end{equation}
		with $E_1(r,\ph,x) = o_r(1)$. Moreover, if $\ph \in C^3(B_R(x))$, there exists a constant $k_1=k_1(d)>0$ such that
		\begin{equation}\label{error-1}
			\big| E_1(r,\ph,x) \big| \le k_1 \, r \, \| D^2 \ph \|_{\LL^\infty(B_R(x))}.
		\end{equation}
	\end{lem}

	One of the key results in the development of the theory extending the classical mean value property for the Laplacian to a nonlinear and non-homogeneous problem related to the $p$-Laplacian, was established in \cite{Manfredi_Parviainen_Rossi_2010}, see also \cite{libroBR,Lewicka_book}. 
	We refer to the following theorem.
	
	\begin{thm}\label{exp-plapN}
		Let $d\geq 1$, $p\geq 2$, $x\in\Rd$, $R>0$ and $\ph\in C^2(B_R(x))$ with $\gr\ph(x)\not=0$. Let also $\beta=\frac{p-2}{p+d}$ and $\gamma=\sqrt{2(p+d)}$ and  define
		$$M_\rho[\ph](x)\coloneqq\frac{\beta}{2}\Big( \sup_{B_{\gamma \rho}(x)}\ph + \inf_{B_{\gamma \rho}(x)}\ph \Big) + (1-\beta) \intn_{B_{\gamma \rho}(x)} \ph(y)\dy.$$
		Then, for any $\rho$ sufficiently small, we have 
		\begin{equation*}
			M_\rho[\ph](x)-\ph(x)=\rho^2 \plapN \ph(x) + \rho^2 E_2(\rho,\ph,x),
		\end{equation*}
		with $E_2(\rho,\ph,x)=o_\rho(1)$.
		Moreover, if $\ph \in C^3(B_R(x))$, there exists a constant $k_1=k_1(d)>0$ such that
		\begin{equation}\label{error-2}
			\left| E_2(\rho,\ph,x)\right|\leq k_2 \rho \left(\| D^3 \ph\|_{\LL^\infty(B_R(x))} + \frac{|D^2\ph(x)|^2}{|\gr \ph(x)|}\right).
		\end{equation}
	\end{thm}
	
	A key tool for the derivation of all the asymptotic expansions contained in this paper is the following numerical lemma. It states that any geometric mean can be expressed as the infimum of certain arithmetic means. A proof can be found in \cite{delTeso_Rossi_2024}.
	
	\begin{lem}\label{infimum}
		Let $a,b\geq 0$ and $0<\alpha<1$, then
		\begin{equation}\label{infimum-1}
			a^\alpha b^{1-\alpha} = \inf_{c>0} \left\{ \alpha c^{1-\alpha} a + (1-\alpha)c^{-\alpha} b \right\}.
		\end{equation}
		Moreover, for any $0\leq m < M \leq \infty$ we have
		\begin{equation}\label{error-3}
			\left|a^\alpha b^{1-\alpha} - \inf_{c\in [m,M]} \left\{ \alpha c^{1-\alpha} a + (1-\alpha)c^{-\alpha} b \right\} \right|\leq \alpha \, a \, m^{1-\alpha} + (1-\alpha)b M^{-\alpha}.
		\end{equation}
	\end{lem}

	\paragraph*{Notation.} Finally, let us fix some notation that we will use trough the paper.
	\begin{itemize}
		\item  Remember that  we use the next notation for the signed power, $(a)^\beta = \sgn(a) |a|^\beta$. Moreover, we write the positive and the negative part as follows, $a^+=\max\left[a,0 \right]$ and $a^-=\max\left[ -a,0 \right]$.
		
		\item We write $u\in B(\RR^d)$ if $\sup_{\RR^d} u<\infty$. Notice that this differs from $\LL^\infty(\RR^d)$, where the essential supremum is considered instead of the supremum.
		
		\item We denote by $C_b(\RR^d)$ the space of bounded continuous functions on $\RR^d$, and by $UC_b(\RR^d)$ the space of bounded uniformly continuous functions on $\RR^d$.
		
		\item Finally, for $\lambda\in\RR_+$, the Hölder space $C^\lambda(\RR^d)$ denotes $C^{k,\alpha}(\RR^d)$, where $\lambda=k+\alpha$ with $k\in\NN$ and $\alpha\in(0,1]$.
		
		\item Given $f\in UC_b(\RR^d)$ we denote by $\Lambda_f$ to its modulus of continuity, that is, $\Lambda_f:[0,\infty)\rightarrow [0,\infty)$ is a nondecreasing continuous function with $\Lambda_f(0)=0$ and $\left|f(x)-f(y)\right|\leq \Lambda_f(|x-y|)$ for all $x,y\in \RR^d$. The modulus of continuity can be defined by
		\[
		\Lambda_f(\delta)\coloneqq \sup \{\,|f(x)-f(y)| : |x-y|\le \delta,\; x,y\in \RR^d \,\}.
		\]
	\end{itemize}
	
	
	\subsection{AMVF for the $p$-Laplacian and the doubly nonlinear parabolic equation}
	We will present here asymptotic expansions for both the $p$-Laplacian $\Delta_p$ and the doubly nonlinear parabolic equation $ (\partial_t u )^{p-1} - \plap u = 0$. To do so, let us observe first that, from \eqref{eq:plap-nplap} we have
	\[
	(\Delta_p \varphi)^{\frac{1}{p-1}}=|\gr \varphi|^\alpha \left(\plapN u\right)^{1-\alpha}, \quad \mbox{with} \quad  \alpha=\frac{p-2}{p-1}.
	\]
	
	The corresponding AMVF follows from the same heuristic ideas presented in Section  \ref{sec:heuristic}. More precisely, we use $M_\rho$ be the AMVF formula for the normalized $p$-Laplacian given in Theorem \ref{exp-plapN} and the numerical identity $a^\alpha b^{1-\alpha} = \inf_{c>0} \left\{ \alpha c^{1-\alpha} a + (1-\alpha)c^{-\alpha} b \right\}$ valid for positive $a,b\in \RR$ that appears in Lemma \ref{infimum}. With these tools at hand, we define the averaging operator
	\begin{equation}\label{operator-A}
		\begin{split}
			\A_\epsilon[\ph](x)=&\frac{1}{2}\max\left[ \inf_{c\in[m(\epsilon),M(\epsilon)]}\left\{ \alpha \sup_{B_{\epsilon^2 c^{1-\alpha}}(x)}\ph +(1-\alpha) M_{\epsilon c^{-\alpha/2}}[\ph](x) \right\}, \ph(x) \right]\\
			&+\frac{1}{2} \min\left[ \sup_{c\in[m(\epsilon),M(\epsilon)]}\left\{ \alpha \inf_{B_{\epsilon^2 c^{1-\alpha}}(x)}\ph +(1-\alpha) M_{\epsilon c^{-\alpha/2}}[\ph](x) \right\}, \ph(x) \right]
		\end{split}
	\end{equation} 
	with the choice
	\begin{equation}\label{m-M}\tag{m-M}
		m(\epsilon)=\epsilon^{\frac{2(p-1)}{3p-4}} \hspace{3mm}\mbox{and}\hspace{3mm}M(\epsilon)=\epsilon^{-2+\frac{2}{p}}.
	\end{equation}
	
	We have the following AMVF for the $p$-Laplace operator.
	\begin{thm}\label{thm.asexp}
		Let $d\geq 1$, $p>2$, $x\in\Rd$, $R\in(0,1)$, and assume \eqref{m-M}. Then, there exists $\beps\in(0,R)$ such that for all $\epsilon\in(0,\beps)$ the quantity \eqref{operator-A} is well defined for any bounded measurable function $\varphi$ in $B_{R}(x)$. Moreover, if $\ph\in C^2(B_{R}(x))$, then the following asymptotic expansions holds
		\begin{equation*}
			\A_\epsilon[\ph](x)=\ph(x)+\frac{\epsilon^2}{2}\left(\plap \ph(x)\right)^{\frac{1}{p-1}}+\epsilon^2 E(\epsilon, \ph, x).
		\end{equation*}
		We further have the following estimates for the error term $E(\epsilon, \ph, x)$:
		\begin{enumerate}[label=(\alph*), leftmargin=*]
			\item\label{thm.asexp-item-a} If $\ph\in C^2(B_{R}(x))$, then $E(\epsilon, \ph, x)=o_\epsilon(1)$ as $\epsilon\to0^+$. 
			\item\label{thm.asexp-item-b} If $\ph\in C^3(B_{R}(x))$ and  $\nabla\varphi(x)\not=0$,  there exists constants $k_1=k_1(d,p)>0$, $k_2=k_2(d,p)>0$ such that 
			\begin{equation*}\left|E(\epsilon,\ph,x)\right|\leq  k_1 \epsilon^{2-\frac{4}{p}}\|D^2\ph\|_{\LL^\infty(B_{\beps}(x))}
				+ k_2 \epsilon^{\frac{2}{3p-4}}\left(\|\gr\ph\|_{\LL^\infty(B_{\beps}(x))}+\|D^3\ph\|_{\LL^\infty(B_{\beps}(x))}+\frac{|D^2\ph(x)|^2}{|\gr\ph(x)|}\right).
			\end{equation*}
			\item\label{thm.asexp-item-c} If $\ph\in C^2(B_{R}(x))$, there exists a constant $k_3=k_3(d,p)>0$ such that
			\begin{equation*}\left|E(\epsilon,\ph,x)\right|\leq k_3 \left( \|D^2\ph\|_{\LL^\infty(B_{\beps}(x))}^{\frac{1}{p-1}}|\gr \ph(x)|^{\frac{p-2}{p-1}}+ \epsilon^{2-\frac{4}{p}} \|D^2\ph\|_{\LL^\infty(B_{\beps}(x))}+\epsilon^{\frac{2}{3p-4}}|\gr\ph(x)|\right).
			\end{equation*}
		\end{enumerate}
	\end{thm}

	\begin{rem}\rm
		First, we observe that the error estimate in Theorem \ref{thm.asexp}\ref{thm.asexp-item-a} holds for $C^2$ functions regardless of whether the gradient vanishes or not. However, the estimate obtained there is not locally uniform in $x$. In contrast, the error estimate in Theorem \ref{thm.asexp}\ref{thm.asexp-item-b}, which requires $C^3$ regularity, is locally uniform away from points where the gradient vanishes. Nevertheless, this estimate degenerates as one approaches points where the gradient is zero. This is precisely why Theorem \ref{thm.asexp}\ref{thm.asexp-item-c} is needed, since it allows us to avoid such degeneracies.  These features are essential for the characterization of viscosity solutions and for the convergence of the DPP established later. See Remark \ref{rem-error} for more details.
	\end{rem}

	From the AMVF for the $p-$Laplacian we derive the following result regarding the doubly nonlinear parabolic equation.
	
	\begin{thm}\label{thm:AMVFparab-intro} Let $d\geq 1$, $p>2$, $x\in\Rd$, $t>0$, $R\in(0,1)$, $\tau\in(0,t)$, and assume \eqref{m-M}. Then, there exists $\beps\in(0,\min[R,\tau])$ such that for all $\epsilon\in(0,\beps)$ the quantity \eqref{operator-A} is well defined for any bounded measurable function $\varphi$ in $B_{R}(x)\times[t-\tau,t+\tau]$. Moreover, if $\ph\in C^2(B_{R}(x)\times[t-\tau,t+\tau])$, then the following asymptotic expansion holds
		\begin{equation*}
			\A_\epsilon[\ph](x,t)=\ph \Big(x,t+\frac{\epsilon^2}{2}\Big)+\frac{\epsilon^2}{2}\Big(\left(\plap \ph(x,t)\right)^{\frac{1}{p-1}}-\partial_t \ph(x,t)\Big)+\epsilon^2 \widehat{E}(\epsilon, \ph, x,t).
		\end{equation*}
		\begin{enumerate}[label=(\alph*), leftmargin=*]
			\item If $\ph\in C^2(B_{R}(x)\times[t-\tau,t+\tau])$, then $\widehat{E}(\epsilon, \ph, x,t)=o_\epsilon(1)$ as $\epsilon\to0^+$.
			\item If $\ph\in C^3(B_{R}(x)\times[t-\tau,t+\tau])$ and  $\nabla\varphi(x,t)\not=0$,  there exists constants $k_1=k_1(d,p)>0$, $k_2=k_2(d,p)>0$ such that 
			\begin{align*}\left|\widehat{E}(\epsilon,\ph,x,t)\right|\leq& \epsilon^2 \|\partial_{tt}\ph\|_{\LL^\infty([t-\beps^2/2,t+\beps^2/2])}+ k_1 \epsilon^{2-\frac{4}{p}}\|D^2\ph\|_{\LL^\infty(B_{\beps}(x))}\\
				&+ k_2 \epsilon^{\frac{2}{3p-4}}\left(\|\gr\ph\|_{\LL^\infty(B_{\beps}(x))}+\|D^3\ph\|_{\LL^\infty(B_{\beps}(x))}+\frac{|D^2\ph(x,t)|^2}{|\gr\ph(x,t)|}\right).
			\end{align*}
			\item If $\ph\in C^2(B_{R}(x)\times[t-\tau,t+\tau])$, there exists a constant $k_3=k_3(d,p)>0$ such that
			\begin{align*}&\left|\widehat{E}(\epsilon,\ph,x,t)\right|\leq \epsilon^2 \|\partial_{tt}\ph\|_{\LL^\infty([t-\beps^2/2,t+\beps^2/2])}\\
				&+ k_3 \left( \|D^2\ph\|_{\LL^\infty(B_{\beps}(x))}^{\frac{1}{p-1}}|\gr \ph(x,t)|^{\frac{p-2}{p-1}}+ \epsilon^{2-\frac{4}{p}} \|D^2\ph\|_{\LL^\infty(B_{\beps}(x))}+\epsilon^{\frac{2}{3p-4}}|\gr\ph(x,t)|\right).
			\end{align*}
		\end{enumerate}
	\end{thm}
	
	
	\subsection{Characterization of viscosity solutions}

	The results in the previous section allow to provide asymptotic mean value characterizations of viscosity solutions to the doubly nonlinear parabolic equation in the range $p\in(2,\infty)$.
	
	\begin{thm}\label{thm:ascharac}
		Let $d\geq 1$, $p>2$, and $\Omega\subseteq \Rd$ be an open set and assume \eqref{m-M}. Then, a continuous function $u$ is a viscosity solution of
		\begin{equation*}
			(\partial_t u(x,t))^{p-1}=\plap u(x,t)\quad \mbox{for all} \quad (x,t)\in \Omega\times (0,T),
		\end{equation*}
		if and only if $u$ is a viscosity solution of 
		\[
		\A_\epsilon[\ph](x,t)=\ph(x,t+\epsilon^2/2) + o(\epsilon^2) \quad \textup{as} \quad \epsilon\to0^+\quad \mbox{for all} \quad (x,t)\in \Omega\times (0,T).
		\]
	\end{thm}
	
	The precise definitions of viscosity solution in the above result will be given in Appendix \ref{ap:viscosity}.
	
	Similar results to the ones in Theorem \ref{thm:ascharac} can be obtained for elliptic equations of the form $-\Delta_p u=f$. The precise statement and its proof follow similarly to the ones in \cite{delTeso_Rossi_2024} and hence we omit them here.
	
	
	\subsection{DPP for the doubly nonlinear parabolic equation in bounded domains}

	We present here a dynamic programming principle (DPP) (that is related to a game--theoretical interpretation) for the following doubly nonlinear Dirichlet problem, extensively studied in \cite{Lindgren_Hynd_2016}:
	\begin{equation}\label{problem-intro-1}
		\left\{\begin{array}{rlll}
			\left( \partial_t u(x,t) \right)^{p-1} &=& \plap u(x,t), & \mbox{in } \Omega\times (0,\infty),\\ [4pt]
			u(x,0)&=&u_0(x), & \mbox{in } \Omega,\\ [4pt]
			u(x,t)&=&g(x), & \mbox{on } \partial \Omega\times (0,\infty).
		\end{array}\right.
	\end{equation}
	
	We assume that $\Omega\subset \Rd$ is a bounded domain, $p\in(2,\infty)$, and that the data belong to spaces of bounded continuous functions, namely $u_0\in C_b(\overline{\Omega})$ and $g\in C_b(\Rd\setminus\Omega)$
	(we will assume that the boundary datum is extended to a function defined in $\Rd\setminus\Omega$. With an abuse of notation we also call $g$ to this extension).
	
	The AMVF given in Theorem \ref{thm:AMVFparab-intro} allows us to derive the following DPP problem associated with \eqref{problem-intro-1}:
	\begin{equation}\label{DPP-main-1}
		\left\{\begin{array}{clll}
			\ue(x,t)&=&\A_\epsilon[\ue](x,t-\epsilon^2/2), & \mbox{in } \Omega\times (0,\infty),\\ [4pt]
			\ue(x,t)&=&u_0(x), & \mbox{in } \Omega\times (-\epsilon^2/2,0],\\ [4pt]
			\ue(x,t)&=&g(x), & \mbox{in } \Rd\setminus \Omega\times (-\epsilon^2/2,\infty).
		\end{array}\right.
	\end{equation}
	
	As will be see later, \eqref{DPP-main-1} can be reinterpreted as a explicit Euler scheme associated with problem \eqref{problem-intro-1}. In fact, the following result establishes a precise connection between the two problems.
	
	\begin{thm}\label{thm-DPPI}
		Let $d\ge 1$, $p>2$, $\epsilon\in(0,1)$, and let $\Omega\subset\Rd$ be a bounded domain satisfying the uniform exterior ball condition. Let $u_0\in C_b(\Omega)$ and $g\in C_b(\Rd\setminus\Omega)$, and assume that \eqref{m-M} holds. Then the following hold:
		\begin{enumerate}[label=(\alph*), leftmargin=*]
			\item There exists a unique bounded measurable solution $\ue$ to \eqref{DPP-main-1}.
			
			\item There exists a function $u\in C(\overline{\Omega}\times[0,\infty))$ such that
			\[
			u_\epsilon \to u 
			\quad \textup{locally uniformly in } \overline{\Omega}\times[0,\infty)
			\quad \textup{as } \epsilon \to 0^+ .
			\]
			The limit $u$ is the unique viscosity solution of \eqref{problem-intro-1}.
		\end{enumerate}
	\end{thm}

	
	\subsection{Game theoretical interpretation for the doubly nonlinear parabolic equation}\label{Sec2.5}
	
	The asymptotic expansion for $p$-harmonic functions introduced in \cite{Manfredi_Parviainen_Rossi_2010} (see Theorem~\ref{exp-plapN}) yields a game-theoretical interpretation of $p$-harmonic functions via the so-called \emph{tug-of-war with noise}, a two-player zero-sum game whose rules are described below. 
	
	\textbf{Description of the tug-of-war with noise for $p$-harmonic functions.}
	Fix a parameter $\rho>0$, an open bounded domain $\Omega$, a starting point $x_0\in\Omega$, and a  final payoff function $g\in C_b(\Rd\setminus \Omega)$.
	\begin{enumerate}[label=(\roman*)]
		\item At each round, a biased coin is tossed with probabilities $\beta=(p-2)/(p+d)$ for heads and $1-\beta$ for tails.
		\begin{itemize}[leftmargin=1em]
			\item If the outcome of the coin toss is tails, the next position of the game is chosen randomly inside the ball of radius $\gamma\rho$ (with uniform distribution).
			\item If the outcome is heads, a round of tug-of-war is played: a fair coin
			(probabilities $1/2-1/2$) determines which of the two players chooses the next position of the game inside the ball of radius $\gamma\rho$.
		\end{itemize}
		
		\item The procedure described in (i) is repeated starting at the new position until the position of the game exits $\Omega$ for the first time; this exit point is denoted by $x_\tau$. At that moment, the second player pays the amount $g(x_\tau)$ to the first player.
	\end{enumerate}

	This previous  construction also allows us to provide a game-theoretical interpretation of the doubly nonlinear parabolic equation under study through the DPP \eqref{problem-intro-1}. The game again involves two players: Player~$I$, who aims to maximize the final payoff, and Player~$II$, who aims to minimize it.
	
	\textbf{Description of the new game for the doubly nonlinear parabolic equation.}
	Fix $\epsilon>0$, a bounded open domain $\Omega$, a starting space-time position $(x_0,t_0)\in\Omega\times (0,\infty)$, and two final payoff functions $u_0\in C_b(\overline{\Omega})$ and $g\in C_b(\Rd\setminus \Omega)$.
	
	\begin{enumerate}[label=(\roman*)]
		
		\item[(I)] At each round, a fair coin is tossed with probability $1/2$ for heads and $1/2$ for tails.
		
		\item[(II)(a)] If the outcome of (I) is heads, Player $I$ may either keep the spatial position $x_{k+1}=x_k$ and update time as $t_{k+1}=t_k-\epsilon^2/2$, or initiate the following procedure: Player $II$ selects a parameter $c\in[m(\epsilon),M(\epsilon)]$, after this choice is made, a biased coin is tossed with probabilities $\alpha=(p-2)/(p-1)$ for heads and $1-\alpha$ for tails:
		\begin{itemize}[leftmargin=1em]
			\item If heads occurs, Player $I$ chooses $x_{k+1}\in B_{\epsilon^2 c^{1-\alpha}}(x_k)$ and $t_{k+1}=t_k-\epsilon^2/2$.
			\item If tails occurs, a round of tug-of-war with noise (as described above) is played in $B_{\gamma\epsilon c^{-\alpha/2}}(x_k)$ (with probabilities $\beta$ and $1-\beta$ as described above), which determines $x_{k+1}$, and time is also changed to $t_{k+1}=t_k-\epsilon^2/2$.
		\end{itemize}
		
		\item[(II)(b)] If the outcome of (I) is tails, the same procedure as in (II)(a) is followed, but with the roles of Player $I$ and Player $II$ interchanged
		(this time it is Player $I$ who chooses $c$,...).
		
		\item[(III)] The procedure in (I) and (II) is repeated until the end of the game
		that happens when the position $x_\tau$ leaves $\Omega$ (with $t_\tau>0$) or
		when time is exhausted, $t_\tau \leq 0$. The final payoff for a particular occurrence of the game is given by: if $t_{k+1}\le 0$ with $x_{k+1}\in\Omega$, Player $II$ pays $u_0(x_{k+1})$ to Player $I$; if $x_{k+1}\notin\Omega$ with $t_{k+1}>0$, Player $II$ pays $g(x_{k+1})$ to Player $I$.
	\end{enumerate}
	
	For this game we have the following result that gives the connection 
	with the solution to the DPP and the parabolic problem \eqref{problem-intro-1}.
	
	\begin{thm} \label{teo.converge.bounded}
		Let the assumptions of Theorem \ref{thm-DPPI} hold, and consider the game described above by $(I)$, $(II)(a)$, $(II)(b)$ and $(III)$. Then, the game has a value, and it is given by the unique solution $\ue$ of the dynamic programming principle \eqref{DPP-main-1}.
		
		In particular, the value of the game converges to the viscosity solution of the parabolic problem \eqref{problem-intro-1}.
	\end{thm}

	
	\subsection{DPP for the doubly nonlinear parabolic equation in $\RR^d$}
	
	Analogous results can also be obtained for the parabolic problem posed in the whole space
	\begin{equation}\label{problemII.intro}
		\left\{
		\begin{array}{rlll}
			(\partial_t u(x,t))^{p-1} &=&\plap u(x,t),
			& \mbox{in } \Rd \times (0,\infty),\\ [4pt]
			u(x,0)&=&u_0(x), & \mbox{in } \Rd.
		\end{array}
		\right.
	\end{equation}
	The corresponding DPP associated with this problem, derived from the operator~\eqref{operator-A}, is given by
	\begin{equation}\label{DPPII.intro}
		\left\{
		\begin{array}{clll}
			\ue(x,t)&=&\A_\epsilon[\ue](x,t-\epsilon^2/2), & \mbox{in } \Rd \times (0,\infty),\\ [4pt]
			\ue(x,t)&=&u_0(x), & \mbox{in } \Rd\times (-\epsilon^2/2,0].
		\end{array}
		\right.
	\end{equation}
	
	The convergence of the solutions of \eqref{DPPII.intro} to a viscosity solution of \eqref{problemII.intro} is established in Section~\ref{sec5}. 
	Since, to the best of our knowledge, uniqueness for this problem in the whole space remains open, the limit obtained through the DPP provides only a particular viscosity solution. Moreover, this limit solution inherits several regularity properties from the solutions to the DPP. Finally, we point out that we have convergence 
	of solutions to the DPP only via compactness arguments and hence we can only assert convergence along subsequences $u_{\epsilon_j} \to u$
	as $\epsilon_j \to 0$.
However,
	if we assume an exponential decay at infinity of the initial datum, we can show
	convergence of the whole family $u_\epsilon$. 
	
	\begin{thm} \label{teo.conver.Rd}
		Let $d\in\NN$, $p>2$ and $\epsilon\in(0,1)$. Let $u_0\in C_b^\lambda(\Rd)$ for some $\lambda\in(0,2]$, and assume that \eqref{m-M} holds. Then the following hold:
		\begin{enumerate}[label=(\alph*), leftmargin=*]
			
			\item There exists a unique bounded measurable solution $\ue$ to \eqref{DPPII.intro}.
			
			\item There exists a convergent subsequence $\{u_{\epsilon_j}\}_{j=1}^\infty$ with $\epsilon_j\to0$ as $j\to+\infty$ such that
			\[
			u_{\epsilon_j} \to u 
			\quad \textup{locally uniformly in } \Rd\times[0,\infty)
			\quad \textup{as } j\to+\infty,
			\]
			for some function $u \in UC(\RR^d\times [0,\infty))$.
			
			\item Every uniform limit along subsequences of $u_{\epsilon}$, $u$ is a viscosity solution to \eqref{problemII.intro} and satisfies the following estimates:
			\begin{itemize}[leftmargin=0.05em]
				\item[] \textup{(Boundedness)} $
				\|u(\cdot,t)\|_{\LL^{\infty}(\RR^d)}
				\leq
				\|u_{0}\|_{\LL^{\infty}(\RR^d)}$,
				for all  $t\in [0,\infty)$.

				\item[] \textup{(Regularity in space)} $\|u(\cdot+y,t)-u(\cdot,t)\|_{\LL^{\infty}(\Rd)}
				\leq
				\|u_{0}(\cdot+y)-u_{0}\|_{\LL^{\infty}(\Rd)}$, for all $(y,t)\in \RR^d\times[0,\infty)$.
				
				\item[] \textup{(Regularity in time)} $\|u(\cdot,t)-u(\cdot,s)\|_{\LL^\infty(\Rd)}
				\leq
				C \|u_0\|_{C^\lambda(\Rd)} |t-s|^{\lambda/2}$,
				for all  $t,s\in [0,\infty)$.
			\end{itemize}
			
			\item In addition, if the initial condition satisfies
			\begin{equation}\label{exp.decay.u0}
				|u_0(x)| \leq C e^{-|x|},
			\end{equation}
			then, the whole family $\{u_{\epsilon}\}_{\epsilon>0}$ converges as $\epsilon \to 0^+$. More precisely, for every $T>0$,
			\[
			u_{\epsilon} \to u 
			\quad \text{uniformly in } \Rd\times[0,T]
			\quad \text{as } \epsilon \to 0^+,
			\]
			for some function $u \in UC(\Rd\times [0,\infty))$ which is a viscosity solution of \eqref{problemII.intro}.
		\end{enumerate}
	\end{thm}
	
	Notice that \eqref{exp.decay.u0} holds for initial conditions that are compactly supported and continuous. 
	In general, let us point out that several difficulties arise even for the heat
equation (the linear case, $p=2$) if the initial data have a complex behavior at infinity, see, for instance, \cite{22}.

	Concerning the game theoretical interpretation in the whole space, the results are similar to the ones
	that hold for the problem in a bounded domain. The rules of the game are the same, the only main 
	difference is that the game ends only when time is exhausted, $t_\tau \leq 0$, and in this case the final payoff is given by $u_0(x_{\tau})$. 
	The convergence of the values of the game as $\epsilon \to 0$ relies on Theorem \ref{teo.conver.Rd}.
	We leave the details of this case to the reader.


	\section{Asymptotic expansion for $C^2$ and $C^3$ functions}\label{sec3}
	In this section, we prove the asymptotic expansions and mean value formulas for the $p$-Laplacian and doubly nonlinear parabolic equation \eqref{equation.intro} in the range $p\in(2,\infty)$. Throughout this work, it is necessary to restrict the set appearing in the infimum \eqref{infimum-1}
	to $c\in [m,M]$, in order to make the computations well defined. More precisely, the hypotheses about the extremes of the interval
	$[m,M]$ are:
	\begin{equation}\label{HT}\tag{HT}
		\begin{split}
			&m:\RR_+\longrightarrow \RR_+ \mbox{ is nondecreasing and such that } m(\epsilon)\to 0 \mbox{ and }\epsilon \, m(\epsilon)^{-\alpha/2}\to 0 \mbox{ as }\epsilon\to 0^+.\\
			&M:\RR_+\longrightarrow \RR_+ \mbox{ is nonincreasing and such that } M(\epsilon)\to +\infty \mbox{ and }\epsilon^2 M(\epsilon)^{1-\alpha}\to 0 \mbox{ as }\epsilon\to 0^+.
		\end{split}
	\end{equation} 
	Note that in Section \ref{sec:mainres} we had assumed that $m$ and $M$ are given by explicit powers of $\epsilon$, see \eqref{m-M}. 
	These choices of $m(\epsilon)$ and $M(\epsilon)$ satisfy \eqref{HT} and, as we will see later, are the precise ones for estimating the error 
	in the asymptotic expansions.
	
	We now show that the operator $\A_\epsilon$ (see formula \eqref{operator-A}) is well defined under suitable assumptions.
	\begin{thm}\label{well-def}
		Let $d\geq1$, $p>2$, $x\in \Rd$, and $\ph:B_R(x)\longrightarrow \RR$ a measurable bounded function, for some $R>0$. Assume \eqref{HT}. Then, there exists an $\beps>0$ such that for all $\epsilon\in (0,\beps)$ the operator $\A_\epsilon[\ph](x)$ is well defined.
	\end{thm}
	\begin{proof}
		Thanks to \eqref{HT}, there exists an $\beps>0$ small such that for any $c\in [m(\epsilon),M(\epsilon)]$ and $\epsilon\in (0,\beps)$ we have
		$$B_{\epsilon^2 c^{1-\alpha}}(x)\subset B_{\epsilon^2M(\epsilon)^{1-\alpha}}(x)\subset B_R(x)\hspace{3mm}\mbox{and}\hspace{3mm}B_{\gamma \epsilon c^{-\alpha/2}}(x)\subset B_{\gamma \epsilon m(\epsilon)^{-\alpha/2}}(x)\subset B_R(x).$$
		Thus, $\sup_{B_{\epsilon^2 c^{1-\alpha}}(x)}\ph$ and $\inf_{B_{\epsilon^2 c^{1-\alpha}}(x)}\ph$ are well defined by the first inclusion. Moreover, $\sup_{B_{\gamma \epsilon c^{\alpha/2}}(x)}\ph$, $\inf_{B_{\gamma \epsilon c^{\alpha/2}}(x)}\ph$ and $\intn_{B_{\gamma \epsilon c^{\alpha/2}}(x)}\ph(y,t)\dy$ are
		also well defined by the second inclusion. In conclusion, the operator $\A_\epsilon[\ph](x,t)$ is well defined. 
	\end{proof}
	
	Now, we are ready to prove the asymptotic expansion for $C^2$ functions at points where $\nabla \varphi\not=0$.
	\begin{thm}\label{exp-c2}
		Let $d\geq 1$, $p>2$, $x\in\Rd$, $t>0$, $R>0$ and assume \eqref{HT} is satisfied. 
		Then, the following statements hold:
		\begin{enumerate}[label=\alph*), leftmargin=*]
			\item Let $\ph\in C^2(B_R(x))$ with $\gr \ph(x)\not=0$. Then, there exists $0<\beps<R$ such that for all $\epsilon\in(0,\beps)$ we have
			\begin{equation}\label{elliptic-formulac2}
				\A_\epsilon[\ph](x)=\ph(x)+\frac{\epsilon^2}{2}\left(\plap \ph(x)\right)^{\frac{1}{p-1}}+o(\epsilon^2).
			\end{equation}
			\item Let $\tau>0$ and $\ph\in C^2(B_R(x)\times [t-\tau,t+\tau])$ with $\gr\ph(x,t)\not=0$. Then, there exists $0<\beps<\min[R,\tau]$ such that for all $\epsilon\in(0,\beps)$ we have
			\begin{equation}\label{formula-c2}
				\A_\epsilon[\ph](x,t)=\ph(x,t+\epsilon^2/2)+\frac{\epsilon^2}{2}\left(\left(\plap \ph(x,t)\right)^{\frac{1}{p-1}}-\partial_t \ph(x,t)\right)+o(\epsilon^2).
			\end{equation}
		\end{enumerate}
	\end{thm}
	\begin{proof}
		Recall that the operator $\A_\epsilon$ is defined in \eqref{operator-A} and denote
		$$Q_\epsilon[\ph](x):= \frac{\A_\epsilon[\ph](x)-\ph(x)}{\epsilon^2/2}.$$

		We split the proof of $a)$ in several steps:
		
		\textit{Step 1: We first show that for all $\epsilon>0$ small enough the following identity holds:}
		\begin{equation}\label{exp-c2-step1}
			\begin{aligned}
				Q_\epsilon&[\ph](x) 
				= \max\Big[ 
				\inf_{c\in[m,M]} \Big\{ 
				\alpha\, c^{1-\alpha} \Big(|\nabla \ph(x)| + \frac{o(\epsilon^2 c^{1-\alpha})}{\epsilon^2 c^{1-\alpha}}\Big) 
				+ (1-\alpha)\, c^{-\alpha} \Big(\plapN \ph(x) + \frac{o(\epsilon^2 c^{-\alpha})}{\epsilon^2 c^{-\alpha}}\Big) 
				\Big\}, \, 0 
				\Big] \\
				&\qquad+ \min\Big[ 
				-\inf_{c\in[m,M]} \Big\{ 
				\alpha\, c^{1-\alpha} \Big(|\nabla \ph(x)| + \frac{o(\epsilon^2 c^{1-\alpha})}{\epsilon^2 c^{1-\alpha}}\Big)
				+ (1-\alpha)\, c^{-\alpha} \Big(-\plapN \ph(x) + \frac{o(\epsilon^2 c^{-\alpha})}{\epsilon^2 c^{-\alpha}}\Big) 
				\Big\}, \, 0 
				\Big].
			\end{aligned}
		\end{equation}
		To prove this assertion we start with the following direct computation
		\begin{align*}
			Q_\epsilon[\ph](x)
			&=
			\frac{1}{\epsilon^2}
			\max\Big[
			\inf_{c\in[m,M]} \Big\{
			\alpha \Big(
			\sup_{B_{\epsilon^2 c^{1-\alpha}}(x)} \ph
			- \ph(x)
			\Big)
			+ (1-\alpha)\Big(
			M_{\epsilon c^{-\alpha/2}}[\ph](x)
			- \ph(x)
			\Big)
			\Big\},\,0
			\Big]\\
			&\quad
			+ \frac{1}{\epsilon^2}
			\min\Big[
			\sup_{c\in[m,M]} \Big\{
			\alpha \Bigl(
			\inf_{B_{\epsilon^2 c^{1-\alpha}}(x)} \ph
			- \ph(x)
			\Bigr)
			+ (1-\alpha)\Big(
			M_{\epsilon c^{-\alpha/2}}[\ph](x)
			- \ph(x)
			\Big)
			\Big\},\,0
			\Big].
			\\
			&=
			\max\Bigg[
			\inf_{c\in[m,M]} \Bigg\{
			\alpha c^{1-\alpha}\Bigl(
			\frac{\sup_{B_{\epsilon^2 c^{1-\alpha}}(x)} \ph
				- \ph(x)}{\epsilon^2 c^{1-\alpha}}
			\Bigr)
			+ (1-\alpha)c^{-\alpha}\Bigl(
			\frac{M_{\epsilon c^{-\alpha/2}}[\ph](x)
				- \ph(x)}{\epsilon^2 c^{-\alpha}}
			\Bigr)
			\Bigg\},\,0
			\Bigg]
			\\
			&\quad+ \min\Bigg[
			-\inf_{c\in[m,M]} \Bigg\{
			\alpha c^{1-\alpha}\Bigl(
			\frac{\ph(x)-\inf_{B_{\epsilon^2 c^{1-\alpha}}(x)} \ph}{\epsilon^2 c^{1-\alpha}}
			\Bigr)
			+ (1-\alpha)c^{-\alpha}\Bigl(
			-\frac{M_{\epsilon c^{-\alpha/2}}[\ph](x)
				- \ph(x)}{\epsilon^2 c^{-\alpha}}
			\Bigr)
			\Bigg\},\,0
			\Bigg].
		\end{align*}
		From here \eqref{exp-c2-step1} follows  using the asymptotic expansions for $|\nabla\varphi|$ and $\plapN \varphi$ in Lemma \ref{exp-gradient} and Theorem \ref{exp-plapN}. 
		
		We aim to bound expression \eqref{exp-c2-step1} from above and from below. Let us fix $\delta>0$ small. Since $\gr\ph(x)\not=0$ and condition \eqref{HT} holds, there exists $\beps>0$ such that for all $\epsilon\in(0,\beps)$ we have
		$$-\delta<\frac{o(\epsilon^2c^{1-\alpha})}{\epsilon^2c^{1-\alpha}}<\delta\hspace{3mm}\mbox{and}\hspace{3mm}-\delta<\frac{o(\epsilon^2c^{-\alpha})}{\epsilon^2c^{-\alpha}}<\delta.$$
		
		\textit{Step 2: Lower bound. } We will show that $\displaystyle Q_\epsilon[\ph](x)\geq \left(\plap \ph(x)\right)^{\frac{1}{p-1}}+o_\delta(1)+o_\epsilon(1)$.
		
		Starting from \eqref{exp-c2-step1} we have,
		\begin{align*}
			Q_\epsilon[\ph](x)&\geq \max \Big[ 
			\inf_{c\in[m,M]} \Big\{ 
			\alpha\, c^{1-\alpha} \big( |\gr \ph(x)| - \delta \big) 
			+ (1-\alpha)\, c^{-\alpha} \big( \plapN \ph(x) - \delta \big) 
			\Big\}, 0 
			\Big] \\
			&\quad - \max \Big[ 
			\inf_{c\in[m,M]} \Big\{ 
			\alpha\, c^{1-\alpha} \big( |\gr \ph(x)| + \delta \big) 
			+ (1-\alpha)\, c^{-\alpha} \big( -\plapN \ph(x) + \delta \big) 
			\Big\}, 0 
			\Big]\\
			&\eqqcolon Q_{\epsilon,\delta}^+[\ph](x)-Q_{\epsilon,\delta}^-[\ph](x).
		\end{align*}
		
		We now aim to apply Lemma~\ref{infimum} to the two infima that appear above. To this end, we distinguish 
		different cases. If the normalized $p$-Laplacian is positive we deduce that the first maximum attains the value of the infimum, while the second one vanishes. If the normalized $p$-Laplacian is nonpositive, then the first maximum vanishes, whereas the second one attains the value of the infimum.  
		In conclusion, if we bound each term separately by the same quantity, we can deduce that, independently of the sign of the $p$-Laplacian at each point, the quotient $Q_\epsilon[\ph]$ is bounded from below. 
		
		Assume that $\plapN\ph(x)> 0$. Then we can choose $\delta>0$ small enough such that $\plapN\ph(x)-\delta>0$. Thus:
		\begin{align*}
			Q_{\epsilon,\delta}^+[\ph](x)&\geq  
			\inf_{c>0} \Big\{ 
			\alpha\, c^{1-\alpha} \big( |\gr \ph(x)| - \delta \big) 
			+ (1-\alpha)\, c^{-\alpha} \big( \plapN \ph(x) - \delta \big) 
			\Big\} \\
			&=  |\gr \ph(x)|^\alpha \big( \plapN \ph(x) \big)^{1-\alpha} + o_\delta(1).
		\end{align*}
		On the other hand, since $-\plapN\ph(x)+\delta<0$, the infimum term in $Q_{\epsilon,\delta}^-[\ph](x)$ is negative for $\epsilon$ small enough. Thus $Q_{\epsilon,\delta}^-[\ph](x)=0$.
		
		Now assume that $\plapN\ph(x)\leq 0$. By similar reasoning we obtain that $Q_{\epsilon,\delta}^+[\ph](x)=0$ and 
		\begin{align*}
			- Q_{\epsilon,\delta}^-[\ph](x)
			&\geq -  
			\big(|\gr \ph(x)| + \delta \big)^\alpha \big( -\plapN \ph(x) + \delta \big)^{1-\alpha} 
			- \alpha\, m(\epsilon)^{1-\alpha} \big( |\gr \ph(x)| + \delta \big)^\alpha \\
			&\quad - (1-\alpha)\, M(\epsilon)^{-\alpha} \big( -\plapN \ph(x) + \delta \big)^{1-\alpha}  \\
			&= -  |\gr \ph(x)|^\alpha \big( -\plapN \ph(x) \big)^{1-\alpha} + o_\delta(1) + o_\epsilon(1).
		\end{align*}
		
		Therefore, regardless of the sign of the normalized $p$-Laplacian, we conclude that
		\[
		\frac{\A_\epsilon[\ph](x)-\ph(x)}{\epsilon^2/2}
		\geq 
		|\gr\ph(x)|^\alpha\left(\plapN\ph(x)\right)^{1-\alpha}
		+o_\delta(1)+o_\epsilon(1).
		\]
		
		\textit{Step 3: upper bound. } We want to show that $\displaystyle Q_\epsilon[\ph](x) < \left(\plap\ph(x,t)\right)^{\frac{1}{p-1}} + o_\delta(1)+o_\epsilon(1).$
		
		It is standard to check that $Q_\epsilon[-\ph](x)=-Q_\epsilon[\ph](x)$ and $\left(\plap [-\ph](x)\right)^{\frac{1}{p-1}}=-\left(\plap \ph(x)\right)^{\frac{1}{p-1}}$. Thus, we apply Step 2 to the function $-\ph$ to get
		\[
		-Q_\epsilon[\ph](x)\geq -\left(\plap \ph(x)\right)^{\frac{1}{p-1}}+o_\delta(1)+o_\epsilon(1),
		\]
		which is precisely what we wanted to show.

		\textit{Step 4: Proof of $a)$.} By the arbitrariness of $\delta$, we conclude, from the estimates in Step 2 and Step 3,  that $Q_\epsilon[\ph](x)= \left(\plap\ph(x)\right)^{\frac{1}{p-1}} + o_\epsilon(1),$
		which completes the proof of~$a)$.
		
		\textit{Step 5: Proof of $b)$.} On the one hand, for every $t>0$ fixed, we use the previous result to obtain
		$$\frac{\A_\epsilon[\ph](x,t)-\ph(x,t)}{\epsilon^2/2}=\left(\plap \ph(x,t)\right)^{\frac{1}{p-1}}+o_\epsilon(1).$$
		On the other hand, by Taylor expansions
		$$ \frac{\ph(x,t+\epsilon^2/2)-\ph(x,t)}{\epsilon^2/2}=\partial_t \ph(x,t) + o_\epsilon(1).$$
		Therefore,  subtracting both expressions we conclude that 
		$$\frac{\A_\epsilon[\ph](x,t)-\ph(x,t+\epsilon^2/2)}{\epsilon^2/2}=\left(\left(\plap \ph(x,t)\right)^{\frac{1}{p-1}}-\partial_t \ph(x,t)\right)+o_\epsilon(1).$$
	\end{proof}
	
	One of the main novelties of this paper, compared with \cite{delTeso_Rossi_2024}, is that the previous results, besides being independent of the sign of the $p$-Laplacian, can also be extended to points where the gradient vanishes (and so does the $p$-Laplacian). Points at which the gradient is zero have traditionally posed difficulties when working with the $p$-Laplacian. For this reason, the proof is slightly different, and these points are treated separately.
	
	We will use the following crucial estimate
	$$\Big|\frac{M_{\rho}[\ph](x)-\ph(x)}{\rho^2}\Big|\leq \|D^2 \ph\|_{\LL^\infty(B_{\gamma\rho}(x))}.$$

	\begin{thm}\label{exp-c2-grad}
		Let the hypotheses of Theorem \ref{exp-c2} hold.  Then, the following statements hold:
		\begin{enumerate}[label=\alph*), leftmargin=*]
			\item Let $\ph\in C^2(B_R(x))$ with $\gr \ph(x)=0$. Then, there exists $0<\beps<R$ such that for all $\epsilon\in(0,\beps)$ we have
			\begin{equation}\label{elliptic-formulac2-grad}
				\A_\epsilon[\ph](x)=\ph(x)+o(\epsilon^2).
			\end{equation}
			\item  Let $\tau>0$ and $\ph\in C^2(B_R(x)\times [t-\tau,t+\tau])$ with $\gr\ph(x,t)=0$. Then, there exists $0<\beps<\min[R,\tau]$ such that for all $\epsilon\in(0,\beps)$ we have
			\begin{equation}\label{formula-c2-grad}
				\A_\epsilon[\ph](x,t)=\ph(x,t+\epsilon^2/2)-\frac{\epsilon^2}{2}\partial_t \ph(x,t)+o(\epsilon^2).
			\end{equation}
		\end{enumerate}
	\end{thm}
	\begin{proof}
		Before starting, recall that, as in Step 1 of the proof of Theorem \ref{exp-c2}, we obtain
		\begin{align*}\label{ec*}
			Q_\epsilon[\ph](x)
			&=
			\max\Big[
			\inf_{c\in[m,M]} \Big\{
			\alpha c^{1-\alpha}
			\frac{\sup_{B_{\epsilon^2 c^{1-\alpha}}(x)} \ph - \ph(x)}
			{\epsilon^2 c^{1-\alpha}}
			+ (1-\alpha) c^{-\alpha}
			\frac{M_{\epsilon c^{-\alpha/2}}[\ph](x) - \ph(x)}
			{\epsilon^2 c^{-\alpha}}
			\Big\},\,0
			\Big]\\
			&\quad
			+ \min\Big[
			\sup_{c\in[m,M]} \Big\{
			\alpha c^{1-\alpha}
			\frac{\inf_{B_{\epsilon^2 c^{1-\alpha}}(x)} \ph - \ph(x)}
			{\epsilon^2 c^{1-\alpha}}
			+ (1-\alpha) c^{-\alpha}
			\frac{M_{\epsilon c^{-\alpha/2}}[\ph](x) - \ph(x)}
			{\epsilon^2 c^{-\alpha}}
			\Big\},\,0
			\Big].
		\end{align*}
		Where we have written $Q_\epsilon[\ph](x)=\frac{\A_\epsilon[\ph](x)-\ph(x)}{\epsilon^2/2}$ as in the proof of Theorem \ref{exp-c2}.
		
		Assume $\ph\in C^2(B_R(x))$ has $\gr\ph(x)=0$ and fix $\delta>0$. Then, by \eqref{HT}, and the fact that the quotient of $M_{\epsilon c^{-\alpha/2}}$ is bounded, the identity above (for every $\epsilon\in (0,\beps)$) leads us to
		\begin{align*}
			Q_\epsilon[\ph](x)&\leq \inf_{c\in[m,M]}\left\{\alpha c^{1-\alpha}\delta + (1-\alpha)c^{-\alpha}\overline{\kappa}\right\}\leq \delta^\alpha\overline{\kappa}^{1-\alpha}+\alpha \delta m(\epsilon)^{1-\alpha}+(1-\alpha)\overline{\kappa}M(\epsilon)^{-\alpha}.
		\end{align*}
		Where we have used Lemma \ref{infimum}. Again, with a similar computation we can find a lower bound
		\begin{align*}
			Q_\epsilon[\ph](x)&\geq -\inf_{c\in[m,M]}\left\{\alpha c^{1-\alpha}\delta+(1-\alpha)c^{-\alpha}\overline{\kappa}\right\}\geq -\delta^\alpha \overline{\kappa}^{1-\alpha}-\alpha\delta m(\epsilon)^{1-\alpha}-(1-\alpha)\overline{\kappa}M(\epsilon)^{-\alpha}.
		\end{align*}
		Hence, taking limits as $\delta\to 0^+$, the asymptotic expansion {\it a)} is proven. 
		
		If $\ph\in C^2(B_R(x)\times [t-\tau,t+\tau])$ with $\gr\ph(x,t)=0$ the asymptotic expansion
		{\it b)} follows by Taylor expansion for $\ph(x,t+\epsilon^2/2)$ as in Theorem \ref{exp-c2}.
	\end{proof}
	
	Now we show that, under stronger regularity assumptions, namely $\ph\in C^3$, one can obtain a more refined version of the asymptotic expansions provided by Theorems \ref{exp-c2} and \ref{exp-c2-grad}. Specifically, it becomes possible to estimate the error term appearing in those expressions in terms of powers of $\epsilon$.

	\begin{thm}\label{exp-c3}
		Let $d\geq 1$, $p>2$, $x\in\Rd$, $t>0$, $R>0$ and assume \eqref{m-M} is satisfied. Then, the following statements holds:
		\begin{enumerate}[label=\alph*), leftmargin=*]
			\item Let $\ph\in C^3(B_{R}(x))$ with $\gr \ph(x)\not=0$. Then, there exists $0<\beps<R$ and two positive constants $k_1$ and $k_2$ depending on $d$ and $p$ such that for all $\epsilon\in(0,\beps)$ we have
			\begin{equation}\label{elliptic-formulac3}
				\A_\epsilon[\ph](x)=\ph(x)+\frac{\epsilon^2}{2}\left(\plap \ph(x)\right)^{\frac{1}{p-1}}+\epsilon^2 E(\epsilon,\ph,x),
			\end{equation}
			with the bound
			\begin{equation}\label{error}
				\left|E(\epsilon,\ph,x)\right|\leq k_1 \epsilon^{2-\frac{4}{p}}\|D^2\ph\|_{\LL^\infty(B_{\beps}(x))}\\
				+ k_2 \epsilon^{\frac{2}{3p-4}}\left(\|\gr\ph\|_{\LL^\infty(B_{\beps}(x))}+\|D^3\ph\|_{\LL^\infty(B_{\beps}(x))}+\frac{|D^2\ph(x)|^2}{|\gr\ph(x)|}\right).
			\end{equation}
			\item Let $\tau>0$ and $\ph\in C^3(B_{R}(x)\times [t-\tau,t+\tau])$ with $\gr\ph(x,t)\not=0$. Then, there exists $0<\beps<\min[R,\tau]$ such that for all $\epsilon\in(0,\beps)$ we have
			\begin{equation}\label{formula-c3}
				\A_\epsilon[\ph](x,t)=\ph(x,t+\epsilon^2/2)+\frac{\epsilon^2}{2}\left(\left(\plap \ph(x,t)\right)^{\frac{1}{p-1}}-\partial_t \ph(x,t)\right)+\epsilon^2 \widehat{E}(\epsilon,\ph,x,t),
			\end{equation}
			where
			\begin{equation}\label{error-para}
				\begin{split}
					\left|\widehat{E}(\epsilon,\ph,x,t)\right|\leq& \epsilon^2 \|\partial_{tt}\ph\|_{\LL^\infty([t-\beps^2/2,t+\beps^2/2])}+ k_1 \epsilon^{2-\frac{4}{p}}\|D^2\ph\|_{\LL^\infty(B_{\beps}(x))}\\
					&+ k_2 \epsilon^{\frac{2}{3p-4}}\left(\|\gr\ph\|_{\LL^\infty(B_{\beps}(x))}+\|D^3\ph\|_{\LL^\infty(B_{\beps}(x))}+\frac{|D^2\ph(x,t)|^2}{|\gr\ph(x,t)|}\right).
				\end{split}
			\end{equation}
		\end{enumerate}
	\end{thm}
	\begin{proof}
		We split the proof in several steps as in Theorem \ref{exp-c2}.
		
		\textit{Step 1:} We will show the following asymptotic expansion for $\A_\epsilon[\ph](x)$,
		\begin{equation}\label{exp-c3-step2}
			\A_\epsilon[\ph](x)-\ph(x)+\epsilon^2 E(\epsilon,\ph,x)=\frac{\epsilon^2}{2}\left(\plap\ph(x)\right)^{\frac{1}{p-1}}.
		\end{equation}
		On the one hand, we write the $p$-Laplacian raised to the (signed) power $1-\alpha$ as its positive part minus its negative part, and apply Lemma~\ref{infimum} to each of them,
		\begin{align*}
			|\gr\ph(x)|^\alpha \left(\plapN\ph(x)\right)^{1-\alpha}=& \max\Big[\inf_{c\in[m,M]}\left\{\alpha c^{1-\alpha}|\gr\ph(x)|+(1-\alpha)c^{-\alpha}\plapN\ph(x)\right\},0\Big]\\
			&-\max\Big[\inf_{c\in[m,M]}\left\{\alpha c^{1-\alpha}|\gr\ph(x)|+(1-\alpha)c^{-\alpha}\left(-\plapN\ph(x)\right)\right\},0\Big]+ E_3(\epsilon,\ph,x).
		\end{align*} 
		Here, $E_3(\epsilon,\ph,x)$ is the distance between the infimum of $c>0$ and the infimum of $c\in[m(\epsilon), M(\epsilon)]$, it can be computed from \eqref{error-3}. Observe that the previous formula holds since, as observed before, depending on the sign of the normalized $p$-Laplacian, one of the maximums always vanish when taking the infimum for $c>0$. On the other hand, using Lemma \ref{exp-gradient} and Theorem \ref{exp-plapN} we obtain
		\begin{align*}
			\frac{\epsilon^2}{2}  |\gr \ph(x)|^\alpha (\plapN \ph(x))^{1-\alpha} 
			=&\frac{\epsilon^2}{2} E_3(\epsilon, \ph,x)+ \frac{1}{2} \max \Big[ 
			\inf_{c \in [m,M]} \Big\{ 
			\alpha \big( \!\!\sup_{B_{\epsilon^2 c^{1-\alpha}}(x)} \!\! \ph - \ph(x) \big) 
			- \alpha \epsilon^2 c^{1-\alpha} E_1(\epsilon^2 c^{1-\alpha},\ph,x) \\
			&  \qquad \quad + (1-\alpha) \big( M_{\epsilon c^{-\alpha/2}}[\ph](x) - \ph(x) \big) 
			- (1-\alpha) \epsilon^2 c^{-\alpha} E_2(\epsilon c^{-\alpha/2},\ph,x) 
			\Big\}, 0 \Big] \\
			&  \qquad  + \frac{1}{2} \min \Big[ 
			\sup_{c \in [m,M]} \Big\{ 
			\alpha \big( \inf_{B_{\epsilon^2 c^{1-\alpha}}(x)} \ph - \ph(x) \big) + \alpha \epsilon^2 c^{1-\alpha} E_1(\epsilon^2 c^{1-\alpha},\ph,x) \\
			&  \qquad \quad + (1-\alpha) \big( M_{\epsilon c^{-\alpha/2}}[\ph](x) - \ph(x) \big) 
			- (1-\alpha) \epsilon^2 c^{-\alpha} E_2(\epsilon c^{-\alpha/2},\ph,x) 
			\Big\}, 0 \Big] \\
			=& \A_\epsilon[\ph](x) - \ph(x) + \epsilon^2 E(\epsilon, \ph,x).
		\end{align*}
		Thus, formula \eqref{exp-c3-step2} will be proven once we  obtain an estimate for the error.
		
		\textit{Step 2: Estimate the error $E(\epsilon,\ph,x)$.} Recal that $\alpha=\frac{p-2}{p-1}$, $m(\epsilon)=\epsilon^{\frac{2(p-1)}{3p-4}}$ and  $M(\epsilon)=\epsilon^{-2+\frac{2}{p}}$.
		From Step 1, we know that
		\begin{equation}
			\left|E(\epsilon,\ph,x)\right|\leq\sup_{c\in[m,M]}\Big\{\alpha c^{1-\alpha} \left|E_1(\epsilon^2 c^{1-\alpha},\ph,x)\right|+(1-\alpha)c^{-\alpha}\left|E_2(\epsilon c^{-\alpha/2},\ph,x)\right|\Big\}+\left|E_3(\epsilon,\ph,x)\right|.
		\end{equation}
		In order to conclude the proof, we bound each error separately. We start with $E_3$ using Lemma \ref{infimum} to get
		\begin{align*}
			\left|E_3(\epsilon,\ph,x)\right|&\leq \alpha \|\gr\ph\|_{\LL^\infty(B_{\beps}(x))} m(\epsilon)^{1-\alpha}+(1-\alpha)\|D^2\ph\|_{\LL^\infty(B_{\beps}(x))} M(\epsilon)^{-\alpha}.
		\end{align*} 
		Now,  we are going to find an estimate for the supremum term for which we need Lemma \ref{exp-gradient} and Theorem \ref{exp-plapN},
		\begin{align*}
			\sup_{c\in[m,M]} & \Big\{ \alpha c^{1-\alpha} \big| E_1(\epsilon^2 c^{1-\alpha}, \ph,x) \big| 
			+ (1-\alpha) c^{-\alpha} \big| E_2(\epsilon c^{-\alpha/2}, \ph,x) \big| \Big\} \\
			&\leq \sup_{c\in[m,M]} \Big\{ \alpha k_1 \epsilon^2 c^{2(1-\alpha)} 
			\| D^2 \ph \|_{\LL^\infty(B_{\beps}(x))} 
			+ (1-\alpha) k_2 \epsilon c^{-3\alpha/2} \Big( \| D^3 \ph \|_{\LL^\infty(B_{\beps}(x))} 
			+ \frac{| D^2 \ph(x) |^2}{|\gr \ph(x)|} \Big) \Big\} \\
			&\leq \alpha k_1 \epsilon^2 M(\epsilon)^{2(1-\alpha)} \| D^2 \ph \|_{\LL^\infty(B_{\beps}(x))} 
			+ (1-\alpha) k_2 \epsilon m(\epsilon)^{-3\alpha/2} 
			\Big( \| D^3 \ph \|_{\LL^\infty(B_{\beps}(x))} + \frac{| D^2 \ph(x) |^2}{|\gr \ph(x)|} \Big)\\
			&\leq k_1 \epsilon^{2-\frac{4}{p}} \| D^2 \ph \|_{\LL^\infty(B_{\beps}(x))} 
			+ k_2 \epsilon^{\frac{2}{3p-4}} 
			\Big( \| D^3 \ph \|_{\LL^\infty(B_{\beps}(x))} + \frac{| D^2 \ph(x) |^2}{|\gr \ph(x)|} \Big).
		\end{align*}
		By combining the bounds for the two components of $E(\epsilon,\ph,x)$, the desired estimate follows. 
		
		\textit{Step 3: Parabolic case.} Again from a Taylor expansion we obtain 
		$$ \frac{\ph(x,t+\epsilon^2/2)-\ph(x,t)}{\epsilon^2/2}=\partial_t \ph(x,t)+R(\epsilon,\ph,x,t),$$
		where $ \left|R(\epsilon,\ph,x,t)\right|\leq \frac{\epsilon^2}{4}\|\partial_{tt}\ph\|_{\LL^\infty([t-\tau,t+\tau])}$. 
		Then, using the elliptic case for any $t>0$ fixed we finish the proof as in Theorem \ref{exp-c2}.
	\end{proof}
	
	The importance of this result lies in the fact that it allows us to estimate the error uniformly with respect to the point $(x,t)$ in a fixed neighbourhood. 
	The difficulty that now arises is that, if the gradient vanishes at the point, one would require an estimate that does not blow up.
	\begin{thm}\label{exp-c3-grad}
		Assume the same hypotheses of Theorem \ref{exp-c3}.
		\begin{enumerate}[label=\alph*), leftmargin=*]
			\item If $\ph\in C^3(B_R(x))$ and its gradient can vanish or not. Then, there exists $0<\beps<R$ such that for all $\epsilon\in(0,\beps)$ 
			\begin{equation}\label{elliptic-formulac3-grad}
				\A_\epsilon[\ph](x)=\ph(x)+\frac{\epsilon^2}{2} \left(\plap \ph(x)\right)^{\frac{1}{p-1}}+\epsilon^2 E(\epsilon,\ph,x),
			\end{equation}
			with
			\begin{equation}\label{error-grad}
				\left|E(\epsilon,\ph,x)\right|\leq k_3 \left( \|D^2\ph\|_{\LL^\infty(B_{\beps}(x))}^{1-\alpha}\left|\gr \ph(x)\right|^\alpha+ \epsilon^{2-\frac{4}{p}} \|D^2\ph\|_{\LL^\infty(B_{\beps}(x))}+\epsilon^{\frac{2}{3p-4}}\left|\gr\ph(x)\right|\right)
			\end{equation}
			where $k_3$ is a constant that only depends on $p$ and $d$.
			
			\item Moreover, assume that $\ph\in C^3(B_R(x)\times [t-\tau,t+\tau])$, then, there exists $0<\beps<\min[R,\tau]$ such that for all $\epsilon\in(0,\beps)$,
			\begin{equation}\label{formula-c3.88}
				\A_\epsilon[\ph](x,t)=\ph(x,t+\epsilon^2/2)+\frac{\epsilon^2}{2}\left(\left(\plap \ph(x,t)\right)^{\frac{1}{p-1}}-\partial_t \ph(x,t)\right)+\epsilon^2 \widehat{E}(\epsilon,\ph,x,t),
			\end{equation}
			with an error that can be estimated as follows
			\begin{equation}\label{error-para-grad}
				\begin{split}
					\Big|\widehat{E}(\epsilon,\ph,x,t)\Big|\leq & \epsilon^2 \|\partial_{tt} \ph \|_{\LL^\infty([t-\beps^2/2,t+\beps^2/2])}\\
					& + k_3 \left( \|D^2\ph\|_{\LL^\infty(B_{\beps}(x))}^{1-\alpha} \left|\gr \ph(x)\right|^\alpha+ \epsilon^{2-\frac{4}{p}} \|D^2\ph\|_{\LL^\infty(B_{\beps}(x))}+\epsilon^{\frac{2}{3p-4}}\left|\gr\ph(x)\right|\right).
				\end{split}
			\end{equation}
		\end{enumerate}
	\end{thm}
	\begin{proof}
		Since the underlying ideas are the same as in the previous results, we present only the key details. Essentially, we need to bound from above and below the quotient $Q_\epsilon[\ph](x)$ as follows,
		\begin{align*}
			Q_\epsilon&[\ph](x)
			\leq
			\max\Bigg[
			\inf_{c\in[m,M]} \Bigg\{
			\alpha c^{1-\alpha} \frac{\sup_{B_{\epsilon^2 c^{1-\alpha}}(x)} \ph
				- \ph(x)}{\epsilon^2 c^{1-\alpha}}
			+ (1-\alpha)c^{-\alpha}
			\|D^2\ph\|_{\LL^\infty(B_{\beps}(x))}
			\Bigg\},\,0
			\Bigg]\\
			&\leq \inf_{c \in [m,M]}\left\{\alpha c^{1-\alpha}\left(| \gr\ph(x)|+k_1\alpha \epsilon^2 c^{2(1-\alpha)}\|D^2\ph\|_{\LL^\infty(B_{\beps}(x))}\right)+(1-\alpha)c^{-\alpha}\|D^2\ph\|_{\LL^\infty(B_{\beps}(x))}\right\}\\
			&\leq |\gr\ph(x)|^\alpha \|D^2\ph\|_{\LL^\infty(B_{\beps}(x))}^{1-\alpha}+\alpha \epsilon^{2-\frac{4}{p}}|\gr\ph(x)|+(1-\alpha)\epsilon^{2-\frac{4}{p}}\|D^2\ph\|_{\LL^\infty(B_{\beps}(x))}+k_1 \alpha \epsilon^{2-\frac{4}{p}} \|D^2\ph\|_{\LL^\infty(B_{\beps}(x))},
		\end{align*}
		where we used estimate \eqref{error-3} in the last line.
		
		The lower bound is analogous and therefore,
		the estimate follows since we have $$\Big|\plap\ph(x)^{\frac{1}{p-1}}\Big|\leq \left|\gr\ph(x)\right|^\alpha \|D^2\ph\|_{\LL^\infty(B_{\beps}(x))}^{1-\alpha}.$$
	\end{proof}
	
	\begin{rem}\rm\label{rem-error}
		This new error estimate is, a priori, weaker than the previous one, since it is not of order $o(\epsilon^2)$. Nevertheless, it is useful because it remains valid at points where the gradient vanishes, yielding the following error bound.
		$$\left|E(\epsilon,\ph,x,t)\right|\leq k_3 \epsilon^{2-\frac{4}{p}}\| D^2\ph\|_{\LL^\infty(B_{\beps}(x))}.$$
		Moreover, in a neighbourhood of a point at which the gradient vanishes, namely in a ball where $\left|\gr\ph(x,t)\right|\leq \epsilon^\gamma$ for some $\gamma$, we recover a sharp error estimate that is uniform and independent of the particular point within that neighbourhood, as Theorem \ref{exp-c3} 
		can be used for neighbourhoods of points where the gradient does not vanish.
	\end{rem}
	
	As a consequence of the previous asymptotic expansions, we establish a characterization of smooth subsolutions and supersolutions of equation \eqref{equation.intro}.
	\begin{cor}\label{characterizacion-1}
		Let $d\geq1$, $p>2$, $x\in\Rd$, $t>0$ and $R,\tau>0$ small enough. Let $\ph\in C^3(B_R(x)\times [t-\tau,t+\tau])$ and assume that \eqref{m-M} holds. Then, the following equivalences are satisfied:
		\begin{equation}
			\partial_t \ph(x,t)\geq \left(\plap\ph(x,t)\right)^{\frac{1}{p-1}} \Longleftrightarrow \A_\epsilon[\ph](x,t)\leq \ph(x,t+\epsilon^2/2)+o(\epsilon^2)\hspace{3mm}\mbox{as }\epsilon\to 0^+.
		\end{equation}
		\begin{equation}
			\partial_t \ph(x,t)\leq \left(\plap\ph(x,t)\right)^{\frac{1}{p-1}} \Longleftrightarrow \A_\epsilon[\ph](x,t)\geq \ph(x,t+\epsilon^2/2)+o(\epsilon^2)\hspace{3mm}\mbox{as }\epsilon\to 0^+.
		\end{equation}
	\end{cor}
	\begin{proof}
		The proof is straightforward and relies on the asymptotic expansions established in Theorems \ref{exp-c3} and \ref{exp-c3-grad}. Indeed, the sign of $\left(\plap\ph(x,t)\right)^{\frac{1}{p-1}}-\partial_t\ph(x,t)$ determines whether the operator $\A_\epsilon[\ph](x,t)$ is asymptotically smaller or larger than $\ph(x,t+\epsilon^2/2)$.
	\end{proof}
	
	To conclude this section, we provide asymptotic mean value characterizations of viscosity solutions, see Definitions \ref{viscosity-equation} and \ref{viscosity-operator}, to equation \eqref{equation.intro}. Note that no a priori regularity assumptions on $u$ are made (beyond continuity).

	\begin{cor}\label{characterizacion-2}
		Let $d\geq 1$, $p>2$, and let $\Omega\subset \Rd$ be an open set. Assume that \eqref{m-M} holds. Then a lower semicontinuous and upper semicontinuous function $u$ is a viscosity solution to
		\begin{equation*}
			(\partial_t u(x,t))^{p-1}=\plap u(x,t)\hspace{5mm}\mbox{for all }(x,t)\in \Omega\times \RR^+,
		\end{equation*}
		if and only if $u$ is a viscosity solution to the AMVF,
		\begin{equation}
			u(x,t+\epsilon^2/2)=\A_\epsilon[u](x,t)+o(\epsilon^2)\hspace{3mm}\mbox{as }\epsilon\to 0^+.
		\end{equation}
	\end{cor}
	\begin{proof}
		Let $\ph\in C^2(B_R(x)\times [t-\tau,t+\tau])$ be a test function touching $u$ from below at $(x,t)$. Then, we have that $\partial_t\ph(x,t)\geq \left(\plap \ph(x,t)\right)^{\frac{1}{p-1}}$ and by Corollary \ref{characterizacion-1} we obtain 
		$$\ph(x,t+\epsilon^2/2)\geq \A_\epsilon[\ph](x,t)+o(\epsilon^2),$$
		which means that $u$ is a supersolution of the AMVF. 
		If the test function $\ph$ touches $u$ from above, the conclusion follows analogously.
		
		The converse implication is proved in the same way. Let $\ph$ be a $C^2$ test function touching $u$ from below at $(x,t)$. Then $ \A_\epsilon[\ph](x,t)\leq \ph(x,t+\epsilon^2/2)+o(\epsilon^2)$ and Corollary \ref{characterizacion-1} yields $$\partial_t\ph(x,t)\geq \left(\plap\ph(x,t)\right)^{\frac{1}{p-1}}.$$
		The case where $\ph$ touches $u$ from above is analogous.
	\end{proof}


	\section{Dynamic Programming Principle in bounded domains}\label{sec4}

	Given $\epsilon>0$, the mean value operator $\A_\epsilon$ defined in \eqref{operator-A} induces the following Dynamic Programming Principle (DPP) associated with \eqref{problem-intro-1},
	\begin{equation}\label{DPP}
		\left\{\begin{array}{cll} \ue(x,t+\epsilon^2/2)=\A_\epsilon[\ue](x,t), &\mbox{in }\Omega\times (-\epsilon^2/2,\infty),
		\\[4pt] \ue(x,0)=u_0(x), &\mbox{in } \Omega\times (-\epsilon^2/2,0],
		\\[4pt] \ue(x,t)=g(x), &\mbox{in }\Rd\setminus \Omega\times (-\epsilon^2/2,\infty).\end{array}\right.
	\end{equation}
	where $\Omega\subset \Rd$ is a bounded smooth domain, and the initial and exterior data belong to the space of continuous and bounded functions, namely $g\in C_b(\Rd\setminus\Omega)$ and $u_0\in C_b(\overline{\Omega})$. Note that the initial datum in time is extended to negative times. An equivalent way to state 
	this problem is the following:
	\begin{equation*}
		\left\{\begin{array}{cll} \ue(x,t)=\A_\epsilon[\ue](x,t-\epsilon^2/2), &\mbox{in }\Omega\times (0,\infty),
		\\[4pt] \ue(x,t)=u_0(x), &\mbox{in } \Omega\times (-\epsilon^2/2,0],
		\\[4pt] \ue(x,t)=g(x), &\mbox{in }\Rd\setminus \Omega\times (-\epsilon^2/2,\infty).\end{array}\right.
	\end{equation*}
	Depending on the computations to be carried out, we will choose the formulation of the DPP that is most convenient.
	
	The existence and uniqueness of solutions to this parabolic DPP are established via an explicit Euler time-discretization scheme, which, roughly speaking, proceeds as follows. For any initial time $t_0\in (-\epsilon^2/2,0]$, we define for all $j\in \NN$ $$t_j=t_0+j\frac{\epsilon^2}{2} \hspace{4mm}\mbox{and}\hspace{4mm} u_{\epsilon}^j=\ue(\cdot,t_j).$$
	The solution at the next time step is then obtained from the previous one by 
	$$u_{\epsilon}^{j+1}=\A_\epsilon[u_{\epsilon}^j] \quad \textup{in} \quad \Omega.$$
	
	Recall that this method allow us to define the solution at any time $\overline{t}>0$ such that there exist $t_0\in(-\epsilon^2/2,0]$ and $j\in\NN$ such that $\overline{t}=t_0+j\frac{\epsilon^2}{2}$. From here to the end, unless otherwise indicated and for simplicity in notation, we choose the mesh that results from choosing $t_0=0$.
	
	\begin{rem} {\rm
			The only delicate point in proving the existence of solutions to \eqref{DPP} is to show that the operator $\A_\epsilon$ maps the space of bounded measurable functions into itself. This can be established by an argument similar to the one used in \cite{delTeso_Rossi_2024}. An alternative approach would be to use viscosity solutions of the DPP, recently introduced in \cite{delTeso_Rossi_ruizC_2024}. However, this is not necessary here, since solutions solve the problem pointwise.}
	\end{rem}

	
	\subsection{Properties of solutions to the DPP}\label{sec4.1}
	We next present several properties of the solutions to \eqref{DPP} arising from the explicit Euler time-discretization scheme
	that we used to show existence and uniqueness.
	
	\begin{thm}[Comparison Principle]\label{comparison}
		Let $d\geq 1$, $p>2$, $\Omega\subset \Rd$ bounded and $\epsilon > 0$. Let $g_1,g_2\in C_b(\Rd\setminus\Omega)$ and $u_{0,1},u_{0,2}\in C_b(\overline{\Omega})$. Take $\ua,\ub:\Rd\times \RR\longrightarrow \RR$ two bounded measurable functions such that
		$$\left\{\begin{array}{cll} \ua(x,t+\epsilon^2/2)\geq\A_\epsilon[\ua](x,t), &\mbox{in }\Omega\times (-\epsilon^2/2,\infty),\\[4pt] \ua(x,t)\geq u_{0,1}(x), &\mbox{in } \Omega\times (-\epsilon^2/2,0],\\[4pt] \ua(x,t)\geq g_1(x), &\mbox{in }\Rd\setminus \Omega\times (-\epsilon^2/2,\infty).\end{array}\right.$$
		$$\left\{\begin{array}{cll} \ub(x,t+\epsilon^2/2)\leq\A_\epsilon[\ub](x,t), &\mbox{in }\Omega\times (-\epsilon^2/2,\infty),\\[4pt] \ub(x,t)\leq u_{0,2}(x), &\mbox{in } \Omega\times (-\epsilon^2/2,0],\\[4pt] \ub(x,t)\leq g_2(x), &\mbox{in }\Rd\setminus \Omega\times (-\epsilon^2/2,\infty).\end{array}\right.$$
		
		If $u_{0,1}\geq u_{0,2}$ in $\Omega$ and $g_1\geq g_2$ in $\Rd\setminus\Omega$, then:
		\begin{enumerate}[label=\alph*)]
			\item If $\ub(x,t_j)\leq \ua(x,t_j)$ for all $x\in\Rd$, then, functions are still ordered in the next step, that is $$\ub(x,t_{j+1})\leq \ua(x,t_{j+1}).$$
			\item 	$\ub(x,t)\leq \ua(x,t)$ for all $(x,t)\in \Rd\times (-\epsilon^2/2,\infty)$.
		\end{enumerate}
	\end{thm}
	
	\begin{rem}\label{remark-comparison} {\rm
			Notice that, the previous result can be stated on time intervals $(-\epsilon^2/2,T)$ for any $0<T\leq \infty$ instead of $(-\epsilon^2/2,\infty)$ 
			in order to use the comparison principle in bounded time inetervals.
			
			Moreover, to apply this comparison principle, it is enough to compare the exterior data in a neighborhood of the domain. More precisely, it suffices to impose the condition $g_1\geq g_2$ in the set 
			\[
			\partial\Omega_\epsilon := \Omega_\epsilon \setminus \Omega\qquad \mbox{with } \Omega_\epsilon:=(\Omega + B_{\epsilon^\gamma}(0)),
			\]
			where $\epsilon^\gamma=\max\left\{\epsilon^2 M(\epsilon)^{1-\alpha}, \, \epsilon \, m(\epsilon)^{-\alpha/2}\right\}$.}
	\end{rem}
	
	Before starting the proof, we need the following two lemmas. 
	\begin{lem}\label{claim-1}
		The following inequality always holds for all $A_1,A_2,B_1,B_2\in \RR$:
		$$\max\left[A_1,A_2\right]-\max\left[B_1,B_2\right]\leq \max\left[A_1-B_1, A_2-B_2\right].$$
	\end{lem}
	
	\begin{proof}
		Denote $A_i=\max\left[A_1,A_2\right]$ and $B_j=\max\left[B_1,B_2\right]$. Then, for any $j\in\left\{1,2\right\}$
		$$A_i-B_j\leq A_i-B_i\leq \max\left[A_1-B_1,A_2-B_2\right].$$
	\end{proof}

	\begin{lem}\label{claim-2}
		For all $\epsilon>0$, the operator $\A_\epsilon$ is a mean value operator and preserves order, that is, if we have two ordered (bounded and measurable) functions  $u\leq v$ in $\RR^d$, then $\A_\epsilon[u]\leq \A_\epsilon[v]$ in $\RR^d$.
	\end{lem}

	\begin{proof}
		Recall that, using $\sup\left\{a\right\}=-\inf\left\{-a\right\}$, we can rewrite $A_\epsilon$ as follows
		\begin{align*}
			\A_\epsilon[u](x)=& \frac{1}{2}\max\Big[\inf_{c\in[m,M]}\Big\{\alpha \sup_{B_{\epsilon^2 c^{1-\alpha}}(x)} u + (1-\alpha)M_{\epsilon c^{-\alpha/2}}[u](x,t)\Big\},u(x)\Big]\\
			&- \frac{1}{2}\max\Big[\inf_{c\in[m,M]}\Big\{\alpha\sup_{B_{\epsilon^2 c^{1-\alpha}}(x)}(-u)+(1-\alpha)M_{\epsilon c^{-\alpha/2}}[-u](x)\Big\},-u(x)\Big].
		\end{align*}
		Therefore, using Lemma \ref{claim-1} several times we can conclude that
		\begin{align*}
			2 & \bigl(\A_\epsilon[u](x)-\A_\epsilon[v](x)\bigr) \\
			& \le
			\max\Big[
			\sup_{c\in[m,M]} \Big\{
			\alpha\Big(
			\sup_{B_{\epsilon^2 c^{1-\alpha}}(x)} u
			- \sup_{B_{\epsilon^2 c^{1-\alpha}}(x)} v
			\Big)
			+ (1-\alpha)\Big(
			M_{\epsilon c^{-\alpha/2}}[u](x)
			- M_{\epsilon c^{-\alpha/2}}[v](x)
			\Big)
			\Big\},\,u(x)-v(x)
			\Big]\\
			&\quad
			+ \max\Big[
			\sup_{c\in[m,M]} \Big\{
			\alpha\Big( 
			\inf_{B_{\epsilon^2 c^{1-\alpha}}(x)} \!\! u
			- \!\! \inf_{B_{\epsilon^2 c^{1-\alpha}}(x)} \!\! v
			\Big)
			+ (1-\alpha)\Big(
			M_{\epsilon c^{-\alpha/2}}[u](x)
			- M_{\epsilon c^{-\alpha/2}}[v](x)
			\Big)
			\Big\},\,u(x)-v(x)
			\Big].
		\end{align*}
		Hence, $\A_\epsilon[u]-\A_\epsilon[v]\leq 0$ since $u\leq v$.
	\end{proof}

	\begin{proof}[Proof Theorem \ref{comparison}]
		\textit{ a).} On the one hand, if $x\in \Rd\setminus\Omega$ there is nothing to prove since 
		$$\ub(x,t)\leq g_2(x)\leq g_1(x)\leq \ua(x,t).$$ 
		
		On the other hand, let $x\in\Omega$ and assume that $\ub(x,t_j)\leq \ua(x,t_j)$ for some $j\in\NN$. Then, by Lemma \ref{claim-2} the operator $\A_\epsilon$ is monotone and therefore, $$\ub(x,t_{j+1})-\ua(x,t_{j+1})\leq \A_\epsilon[\ub](x,t_j)-\A_\epsilon[\ua](x,t_j)\leq 0.$$
		
		\textit{ b).}
		If $(x,t)\in \Rd\setminus\Omega\times \RR^+$ or $(x,t)\in \Omega\times \RR^-\cup \left\{0\right\}$ then 
		$$\ub(x,t)\leq g_2(x)\leq g_1(x)\leq \ua(x,t)\hspace{3mm}\mbox{and}\hspace{3mm} \ub(x,t)\leq u_{0,2}(x)\leq u_{0,1}(x)\leq \ua(x,t),$$
		respectively. Finally, if $(x,t)\in \Omega\times \RR^+$, then there exists $t_0\in(-\epsilon^2/2,0]$ and $j\in\NN$ such that $t=t_0+j\frac{\epsilon^2}{2}$. Applying part $a)$ iteratively $j$ times, we obtain
		$$\ub(x,t_0)\leq \ua(x,t_0)\Rightarrow \ub(x,t_0+\frac{\epsilon^2}{2})\leq \ua(x,t_0+\frac{\epsilon^2}{2})\Rightarrow...\Rightarrow \ub(x,t_0+j\frac{\epsilon^2}{2})\leq \ua(x,t_0+j\frac{\epsilon^2}{2}),$$
		and the proof is completed. 
	\end{proof}
	
	\begin{lem}[Contraction in $\LL^\infty$]\label{contraction}
		Let $d\geq 1$, $p>2$, $\epsilon >0$, and $\Omega\subset\Rd$ bounded. Let $g_1,g_2\in C_b(\Rd\setminus\Omega)$ and $u_0,v_0\in C_b(\overline{\Omega})$. Take $u,v:\Rd\times\RR\longrightarrow \RR$ two bounded measurable functions satisfying
		\begin{equation*}
			\left\{\begin{array}{cll} u(x,t+\epsilon^2/2)=\A_\epsilon[u](x,t), &\mbox{in }\Omega\times (-\epsilon^2/2,\infty),\\[4pt] u(x,t)=u_0(x), &\mbox{in } \Omega\times (-\epsilon^2/2,0],\\[4pt] u(x,t)=g_1(x), &\mbox{in }\Rd\setminus \Omega\times (-\epsilon^2/2,\infty).\end{array}\right.
		\end{equation*}
		\begin{equation*}
			\left\{\begin{array}{cll} v(x,t+\epsilon^2/2)=\A_\epsilon[v](x,t), &\mbox{in }\Omega\times (-\epsilon^2/2,\infty),\\[4pt] v(x,t)=v_0(x), &\mbox{in } \Omega\times (-\epsilon^2/2,0],\\[4pt] v(x,t)=g_2(x), &\mbox{in }\Rd\setminus \Omega\times (-\epsilon^2/2,\infty).\end{array}\right.
		\end{equation*}
		For fixed $t_0\in (-\epsilon^2/2,0]$ define $u^j(x)=u(x,t_0+j\frac{\epsilon^2}{2})$ and $v^j(x)=v(x,t_0+j\frac{\epsilon^2}{2})$. Then, 
		$$\|u^{j+1}-v^{j+1}\|_{\LL^\infty(\RR^d)}\leq \| u^j-v^j\|_{\LL^\infty(\RR^d)}.$$
		Moreover, for every $\overline{t}>0$, 
		$$\|u(\cdot,\overline{t})-v(\cdot,\overline{t})\|_{\LL^\infty(\RR^d)}\leq \max\Big\{\|u_0-v_0\|_{\LL^\infty(\Omega)},\|g_1-g_2\|_{\LL^\infty(\RR^d\setminus\Omega)}
		\Big\} .$$
	\end{lem}

	\begin{proof}
		Let $x\in\Omega$, then using $u^{j+1}(x)=\A_\epsilon[u^j](x)$ and the definition of the operator \eqref{operator-A}, we get
		\begin{align*}
			\big|&u^{j+1}(x)-v^{j+1}(x)\big|
			=
			\left|\A_\epsilon[u](x,t_j)-\A_\epsilon[v](x,t_j)\right|\\
			&\le
			\frac{1}{2}
			\Big|
			\max\Big[
			\sup_{c\in[m,M]} \Big\{
			\alpha\Big(
			\sup_{B_{\epsilon^2 c^{1-\alpha}}(x)} u^j
			- \sup_{B_{\epsilon^2 c^{1-\alpha}}(x)} v^j
			\Big)
			+ (1-\alpha)\Big(
			M_{\epsilon c^{-\alpha/2}}[u^j]
			- M_{\epsilon c^{-\alpha/2}}[v^j]
			\Big)
			\Big\},\,u^j-v^j
			\Big]
			\Big|\\
			&\quad
			+ \frac{1}{2}
			\Big|
			\max\Big[
			\sup_{c\in[m,M]} \Big\{
			\alpha\Big(
			\sup_{B_{\epsilon^2 c^{1-\alpha}}(x)} (-u^j)
			- \sup_{B_{\epsilon^2 c^{1-\alpha}}(x)} (-v^j)
			\Big)
			\\
			&\qquad \qquad + (1-\alpha)\Big(
			M_{\epsilon c^{-\alpha/2}}[-u^j]
			- M_{\epsilon c^{-\alpha/2}}[-v^j]
			\Big)
			\Big\},\,u^j-v^j
			\Big]
			\Big|.
		\end{align*}
		Now, we bound each term in the above expression separately. We have 
		\begin{equation*}
			\alpha \left|\sup_{B_{\epsilon^2 c^{1-\alpha}}(x)} u^j- \sup_{B_{\epsilon^2 c^{1-\alpha}}(x)} v^j\right|\leq \alpha \sup_{B_{\epsilon^2 c^{1-\alpha}}(x)}\left|u^j-v^j\right|\leq \alpha \|u^j-v^j\|_{\LL^\infty(\RR^d)},
		\end{equation*}
		and
		\begin{align*}
			(1-\alpha)\left|M_{\epsilon c^{-\alpha/2}}[u^j](x)-M_{\epsilon c^{-\alpha/2}}[v^j](x)\right|&\leq (1-\alpha)\beta \sup_{B_{\gamma \epsilon c^{-\alpha/2}}(x)}\left|u^j-v^j\right|\\
			&\qquad+(1-\alpha)(1-\beta)\fint_{B_{\gamma \epsilon c^{-\alpha/2}}}|u^j(y)-v^j(y)|\dy\\
			&\leq (1-\alpha)\|u^j-v^j\|_{\LL^\infty(\RR^d)}.
		\end{align*}
		Proceeding analogously with the second term, we obtain $$\left|u^{j+1}(x)-v^{j+1}(x)\right|\leq \|u^j-v^j\|_{\LL^\infty(\Rd)}.$$ Since the point $x$ is arbitrary, we may take the supremum over $x\in\Omega$ to ensure the $L^\infty$ bound in $\Omega$. If $x\in \Rd\setminus\Omega$, the bound follows from the definition of the solutions outside $\Omega$.
		
		For the second part, let $\overline{t}=t_0+j\frac{\epsilon^2}{2}$ for some $t_0\in(-\epsilon^2/2,0]$ and $j\in\NN$. Applying the the first part iteratively $j_n$ times, we obtain
		$$\|u^{j}-v^{j}\|_{\LL^\infty(\RR^d)}\leq \|u^{j-1}-v^{j-1}\|_{\LL^\infty(\RR^d)}\leq...\leq \max\Big\{\|u_0-v_0\|_{\LL^\infty(\Omega)},\|g_1-g_2\|_{\LL^\infty(\RR^d\setminus\Omega)}\Big\} .$$
	\end{proof}
	
	
	\subsection{Convergence as $\epsilon\to 0^+$.}\label{sec4.2}
	Throughout this part of the section, we show that the solutions to the DPP \eqref{DPP} converge, as $\epsilon\to 0^+$, to the viscosity solution of the parabolic problem \eqref{problem-intro-1}.

	Since existence and uniqueness for the DPP have been already established, our next task is to show that the solutions of the DPP are uniformly bounded in $\Rd\times \RR^+$ independently of $\epsilon$.
	\begin{lem}\label{bounded}
		Let $d\geq 1$, $p>2$, $\epsilon >0 $, $\Omega\subset\Rd$ bounded. Assume $u_0\in C_b(\Omega)$, $g\in C_b(\Rd\setminus\Omega)$. 
		Then, if $\ue$ is the solution of the DPP \eqref{DPP}, there exists $\epsilon_0>0$ such that for all $\epsilon\in(0,\epsilon_0)$
		\begin{equation}\label{bounded-ec1}
			\|\ue \|_{\LL^{\infty}(\Rd\times \RR^+)}\leq C.
		\end{equation}
		Where $C>0$ is a constant depending only on $d$, $p$, $\Omega$, $g$ and $u_0$ but not on $\epsilon$. 
	\end{lem}
	\begin{proof}
		The key observation is that constants act as smooth barriers for solutions to the DPP.
		
		\textit{Lower bound.}  
		Define the constant function
		\[
		\underline{v}(x,t)=-2\bigl(\|g\|_{\LL^\infty(\Rd\setminus\Omega)}+\|u_0\|_{\LL^\infty(\Omega)}\bigr).
		\]
		Clearly, \( \underline{v}(x,0)\leq u_0(x) \) for \( x\in\Omega \), and \( \underline{v}(x,t)\leq g(x) \) for \( x\in\Rd\setminus\Omega \). Moreover, since the coefficients of the operator \( \A_\epsilon \) sum up to one, we have $\A_\epsilon[\underline{v}](x,t)=\underline{v}(x,t)=\underline{v}(x,t+\epsilon^2/2)$.
		Thus, \( \underline{v} \) is a subsolution of \eqref{DPP}, and by the comparison principle,
		\[
		\underline{v}(x,t)\leq \ue(x,t)\qquad\text{for all }(x,t)\in \Rd\times\RR^+.
		\]
		
		\textit{Upper bound.}  
		Consider $\overline{v}(x,t)=2\bigl(\|g\|_{\LL^\infty(\Rd\setminus\Omega)}+\|u_0\|_{\LL^\infty(\Omega)}\bigr)$, and repeat the same argument as above.
	\end{proof}
	
	The final ingredients in the convergence analysis of the solutions to the DPP are the following estimates for the $\liminf$ and $\limsup$ of $\ue$ with respect to the initial and boundary data. We first address the case of the spatial boundary data.
	
	\begin{rem}\label{plap-radial-fun}\rm
		Before proceeding, let us recall the meaning of the uniform exterior ball condition for a bounded domain $\Omega$: there exists a radius $R>0$ such that for all $x_0\in\partial\Omega$ there exists a point $z_0\in\Rd\setminus\Omega$ such that $\overline{B_R(z_0)}\bigcap \overline{\Omega}=\left\{x_0\right\}$.
		
		Futhermore, we will need the following computation:
		Let $f_1(x)=|x|^{-\alpha}$ and $f_2(x)=|x|^\beta$ be two radial functions in $\Rd$ with $\alpha,\beta>0$. If $p>2$, then the $p-$Laplacian of these functions 
		is given by the following explicit expressions
		\begin{align*}
			\plap |x|^{-\alpha}&=\alpha^{p-1} (\alpha(p-1)+p-d)|x|^{-\alpha(p-1)-p},\\
			\plap |x|^\beta&=\beta^{p-1}(\beta(p-1)+d-p)|x|^{\beta(p-1)-p}.
		\end{align*}
		In particular, if $\alpha=\frac{p+d-2}{p-1}$, then all the coefficients in the first expression are positive. Moreover, if we choose $\beta=\frac{3p-2}{p-1}>\frac{p}{p-1}$, the $p-$Laplacian of $f_2$ can be extended continuously so that it takes the value $0$ at the origin.
	\end{rem}

	\begin{lem}\label{bound-g}
		Let $d\geq 1$, $p>2$, $\epsilon\in(0,1)$, and let $\Omega\subset\Rd$ be a bounded domain which satisfies the uniform exterior ball condition. Assume $u_0\in C_b(\Omega)$, $g\in C_b(\Rd\setminus\Omega)$ and \eqref{HT}. If $\ue$ is the solution of the DPP, then for every $T>0$ and $x_0\in \partial\Omega$
		\begin{align*}
			\liminf_{\epsilon\to 0^+, y\to x_0, s\to T}\ue(y,s)\geq g(x_0),\\
			\limsup_{\epsilon\to 0^+, y\to x_0, s\to T}\ue(y,s)\leq g(x_0).
		\end{align*}
	\end{lem}
	\begin{proof}
		Fix $T>0$ and $x_0\in\partial\Omega$. By the uniform exterior ball condition, there exist $R>0$ and $z_0\in\Rd\setminus\Omega$ such that $\overline{B_R(z_0)}\cap \overline{\Omega}=\{x_0\}$. We define the lower barrier
		\[
		\wb(x,t)=k \left(|x-z_0|^{-\alpha}-R^{-\alpha}\right)+g(x_0)-\eta+\left(-\frac{t^2}{T^2}+\frac{2t}{T}-1\right)\left(\|u_0\|_{\LL^\infty(\Omega)}+\|g\|_{\LL^\infty(\Rd\setminus\Omega)}\right),
		\]
		where $\alpha=\frac{p+d-2}{p-1}$ and $k>0$ is a constant to be chosen. We will show that for every $\eta>0$ there exists $\beps\in(0,1)$ such that $\wb$ is a lower barrier for $\ue$ for every $0<\epsilon<\beps$. To this end, we prove that $\wb$ is a subsolution of the DPP and then apply Theorem~\ref{comparison}. Observe that the function $\wb(x,t)$ is smooth outside a neighborhood of the point $z_0$; hence, in particular, it is smooth in $\Omega_\epsilon$ (see Remark~\ref{remark-comparison}).
		
		\begin{itemize}[leftmargin=*]
			
			\item First, we show that $\wb(x,t+\epsilon^2/2)\leq \A_\epsilon[\wb](x,t)$ in $\Omega\times (-\epsilon^2/2,T)$.
			
			Since $\wb\in C^\infty(\Omega_\epsilon\times \RR)$ and $\gr\wb (x) =k\alpha |x-z_0|^{-\alpha-2} (x-z_0)\neq 0$ 
			for $x\in \Omega_\epsilon$, we may apply Theorem~\ref{exp-c2}. Moreover, using Remark~\ref{plap-radial-fun} to compute the $p$-Laplacian of $\wb$, we obtain
			\begin{align*}
				\frac{\A_\epsilon[\wb](x,t)-\wb(x,t+\epsilon^2/2)}{\epsilon^2/2}
				&= \bigl(\plap \wb(x,t)\bigr)^{\frac{1}{p-1}}-\partial_t \wb(x,t)+o_\epsilon(1)\\
				&= k\alpha\bigl(\alpha(p-1)+p-d\bigr)^{\frac{1}{p-1}}|x-z_0|^{-\alpha-\frac{p}{p-1}}
				+\frac{2t}{T^2}-\frac{2}{T}+o_\epsilon(1)\\
				&\geq k\alpha\bigl(\alpha(p-1)+p-d\bigr)^{\frac{1}{p-1}}
				(\diam(\Omega)+R)^{-\alpha-\frac{p}{p-1}}\\
				&\quad -\frac{2}{T}\bigl(\|u_0\|_{\LL^\infty(\Omega)}+\|g\|_{\LL^\infty(\Rd\setminus\Omega)}\bigr)+o_\epsilon(1)\\
				&\geq 1,
			\end{align*}
			where in the last step $k$ has been chosen sufficiently large, depending only on $T$, $u_0$, $g$, $p$, $d$, and $\Omega$.
			
			\item Next, we show that $\wb(x,t)\leq u_0(x)$ in $\Omega\times (-\epsilon^2/2,0]$.
			
			Since the first term of $\wb$ is negative in $\Omega$ and the parabola in $t$ that appears in the definition of $\wb$ attains its maximum in this interval at $t=0$, we clearly have
			\[
			\wb(x,t)\leq g(x_0)-\|u_0\|_{\LL^\infty(\Omega)}-\|g\|_{\LL^\infty(\Rd\setminus\Omega)}
			\leq u_0(x).
			\]
			
			\item Finally, we show that $\wb(x,t)\leq g(x)$ in $\partial\Omega_\epsilon\times (-\epsilon^2/2,T)$.
			
			If $x\in\partial\Omega\subset\Rd\setminus B_R(z_0)$, then by construction $\wb(x,t)\leq g(x_0)-\eta$. Hence, there exists $r>0$, depending on $g$ and $\eta$, such that
			\[
			\wb(x,t)\leq g(x)-\eta/2
			\qquad\text{for all }x\in \partial\Omega\cap B_r(x_0).
			\]
			
			Since $g$ is bounded on $\partial\Omega$, by choosing $k$ sufficiently large if necessary, we also have
			\[
			\wb(x,t)\leq g(x)-\eta/2
			\qquad\text{for all }x\in \partial\Omega\setminus B_r(x_0).
			\]
			
			Therefore, we conclude that
			$$\wb(x,t)\leq g(x)-\eta/2 \qquad \mbox{for all } x\in \partial\Omega\times (-\epsilon^2/2,T).$$
			
			Now we extend this bound to an $\epsilon$-neighborhood of the boundary, namely $\partial\Omega_\epsilon$ (see Remark~\ref{remark-comparison}). Since $\wb$ and $g$ are continuous and bounded, given $x\in \partial\Omega_\epsilon$ we choose $\widehat{x}\in\partial \Omega$ such that $\dist(x,\partial\Omega)=|x-\widehat{x}|<\epsilon^\gamma$. Then
			\begin{equation*}
				\wb(x,t)\leq \wb(\widehat{x},t)+\|\gr \wb\|_{\LL^\infty(\Omega_\epsilon)}\epsilon^\gamma
				\leq g(\widehat{x})-\eta/2 +\|\gr \wb\|_{\LL^\infty(\Omega_\epsilon)}\epsilon^\gamma
				\leq g(x) -\eta/8,
			\end{equation*}
			where in the last step we used that $g$ is continuous and that $\|\gr\wb\|_{\LL^\infty(\Omega_\epsilon)}\epsilon^\gamma\leq \eta/8$ for any $0<\epsilon<\beps$, with $\beps$ depending on $\eta$ small enough.
		\end{itemize}
		
		Altogether, $\wb$ is a subsolution of the DPP~\eqref{DPP}. By the comparison principle,
		$\wb(x,s)\leq \ue(x,s)$ for all $x\in\overline{\Omega}$ and $s\in(-\epsilon^2/2,T]$.
		Therefore,
		\[
		\liminf_{\epsilon\to 0^+,\, y\to x_0,\, s\to T}\ue(y,s) \geq \liminf_{\epsilon\to 0^+,\, y\to x_0,\, s\to T} \wb(y,s)= g(x_0)-\eta.
		\]
		Since $\eta>0$ is arbitrary, the result follows.

		The $\limsup$ estimate is proved analogously by considering the upper barrier
		\[
		\wa(x,t)=k\bigl(R^{-\alpha}-|x-z_0|^{-\alpha}\bigr)+g(x_0)+\eta
		+\left(\frac{t^2}{T^2}-\frac{2t}{T}+1\right)
		\left(\|u_0\|_{\LL^\infty(\Omega)}+\|g\|_{\LL^\infty(\Rd\setminus\Omega)}\right).
		\]
	\end{proof}
	
	Next, we deal with the initial data.

	\begin{lem}\label{bound-u0}
		Let $d\geq 1$, $p>2$, $\epsilon\in(0,1)$, and let $\Omega\subset \Rd$ be a bounded domain satisfying the uniform exterior ball condition. Assume $u_0\in C_b(\overline{\Omega})$, $g\in C_b(\Rd\setminus\Omega)$ and \eqref{HT}. Let  $\ue$ be the solution of the DPP, then for all $x_0\in\Omega$ we have 
		\begin{align*}
			\liminf_{\epsilon\to 0^+, y\to x_0, s\to 0^+} \ue(y,s)\geq u_0(x_0).\\
			\limsup_{\epsilon\to 0^+, y\to x_0, s\to 0^+} \ue(x,s)\leq u_0(x_0).
		\end{align*} 
	\end{lem}
	\begin{proof}
		Let us fix $x_0\in\Omega$. We define the lower barrier
		$$
		\wb(x,t)=k_1(1-e^{t+1})-k_2\left(\|u_0\|_{\LL^\infty(\Omega)}+\|g\|_{\LL^\infty(\Rd\setminus \Omega)}\right)|x-x_0|^\beta+u_0(x_0)-\eta.
		$$
		Here $k_1,k_2>0$ are constants to be determined and $\beta=\frac{3p-2}{p-1}=\frac{p}{p-1}+2$.
		
		Note that, by the choice of $\beta$ sufficiently large, $\wb\in C^2(\Rd\times\RR^+)$. Moreover, the $p$-Laplacian of $\wb$ at the point $x_0$ can be extended continuously so as to take the value $0$.
		
		We will show that, for every $\eta>0$, there exists an $\beps\in(0,1)$ such that the function $\wb$ is a lower barrier for $\ue$ for every $\epsilon\in(0,\beps)$. To this end, we want to show that $\wb$ is subsolution to the DPP and after that, apply the Comparison Principle. We split the proof into three steps.
		
		\medskip
		\textit{Step 1: Choice of $k_2$ to ensure the boundary and initial inequalities.}
		\begin{itemize}[leftmargin=*]
			\item Let us show that $\wb(x,t)\leq u_0(x)$ in $\Omega\times (-\epsilon^2/2,0]$.
			
			Note that $\wb(x,t)\leq u_0(x_0)-\eta$ since the first term in $\wb$ is nonpositive and the second term is nonnegative. Hence, by the continuity of $u_0$, there exists a radius $r>0$ depending only on $\eta$ and $u_0$, such that
			$$
			\wb(x,t)\leq u_0(x)-\frac{\eta}{2}
			\qquad \text{for all } x\in B_r(x_0)\subset\Omega.
			$$
			If instead $x\in\Omega\setminus B_r(x_0)$, we have $-|x-x_0|^\beta\leq -r^\beta$, and therefore, if $k_2\geq 2r^{-\beta}$,
			$$
			\wb(x,t)\leq -2\left(\|u_0\|_{\LL^\infty(\Omega)}+\|g\|_{\LL^\infty(\Rd\setminus\Omega)}\right)+u_0(x_0)-\eta\leq u_0(x)-\eta \qquad \text{for all } x\in\Omega\setminus B_r(x_0).
			$$
			\item Let show now that $\wb(x,t)\leq g(x)$ in $\partial\Omega_\epsilon\times (-\epsilon^2/2,0]$, remember Remark \ref{remark-comparison}.
			
			As in the proof of Lemma \ref{bound-g}, we first see that the previous inequality holds for any $x\in \partial\Omega$. Since $x_0\in\Omega$, we have $|x-x_0|^\beta\geq \left(\frac{\dist(x_0,\partial\Omega)}{2}\right)^\beta$. Thus, if $k_2\geq \left(\frac{2}{\dist(x_0,\partial\Omega)}\right)^\beta$, it follows that
			$$
			\wb(x,t)\leq -\left(\|u_0\|_{\LL^\infty(\Omega)}+\|g\|_{\LL^\infty(\Rd\setminus\Omega)}\right)+u_0(x_0)-\eta\leq g(x)-\eta\qquad \mbox{for all }x\in\partial\Omega.
			$$
			By the continuity of $g$ and $\wb$ we could obtain, as in the proof of the last lemma, the inequality in a neiboorhood of the boundary. Let $x\in \partial\Omega_\epsilon$, we choose $\widehat{x}\in\partial\Omega$ such that $\dist(x,\partial\Omega)=|x-\widehat{x}|<\epsilon^\gamma$, then
			$$\wb(x,t)\leq \wb(\widehat{x},t)+\|\gr \wb\|_{\LL^\infty(\Omega_\epsilon)}\epsilon^\gamma
			\leq g(\widehat{x})-\eta/2 +\|\gr \wb\|_{\LL^\infty(\Omega_\epsilon)}\epsilon^\gamma
			\leq g(x) -\eta/8.$$
		\end{itemize}
		
		Collecting the above conditions on constant $k_2$ to be large enough, we fix it as follows
		$$
		k_2:=2\max\left\{
		2r^{-\beta},
		\left(\frac{2}{\dist(x_0,\partial\Omega)}\right)^\beta
		\right\}.
		$$
		
		\medskip
		
		\textit{Step 2: Choice of $k_1$ to ensure the inequality for the equation.}
		
		Let show that $\A_\epsilon[\wb](x,t)\geq \wb(x,t+\epsilon^2/2)$ in $\Omega\times(-\epsilon^2/2,\infty)$. Let $x\in\Omega$ and $t\in(-\epsilon^2/2,\infty)$. Since $\wb\in C^2(\Rd\times\RR)$, we use Theorem \ref{exp-c2} if $\gr\wb\not=0$, or Theorem \ref{exp-c2-grad} if $\gr\wb=0$ to obtain  
		\begin{align*}
			\frac{\A_\epsilon[\wb](x,t)-\wb(x,t+\epsilon^2/2)}{\epsilon^2/2}
			&=\left(\plap \wb(x,t)\right)^{\frac{1}{p-1}}-\partial_t \wb(x,t)+o_\epsilon(1)\\
			&= -C_{p,d} k_2
			|x-x_0|^{\beta-\frac{p}{p-1}}\left(\|u_0\|_{\LL^\infty(\Omega)}+\|g\|_{\LL^\infty(\Rd\setminus\Omega)}\right)+k_1 e^{t+1}+o_\epsilon(1)\\
			&\geq -C_{p,d}
			\diam(\Omega)^2 k_2
			\left(\|u_0\|_{\LL^\infty(\Omega)}+\|g\|_{\LL^\infty(\Rd\setminus\Omega)}\right)
			+k_1+o_\epsilon(1).
		\end{align*}
		
		Here, we have used Remark \ref{plap-radial-fun} to compute the $p$-Laplacian and $C_{d,p}=\frac{3p-2}{p-1}(2p-2+d)^{\frac{1}{p-1}}$. Choosing constant $k_1$ large, $$
		k_1:=C_{d,p}
		\diam(\Omega)^2 k_2
		\left(\|u_0\|_{\LL^\infty(\Omega)}+\|g\|_{\LL^\infty(\Rd\setminus\Omega)}\right)+2,
		$$
		we ensure that the right-hand side is greater than or equal to $1$, obtaining $$\A_\epsilon[\wb](x,t)\geq \wb(x,t+\epsilon^2/2).$$
		
		\medskip
		\textit{Step 3: Conclusion.}
		
		From the previous steps, $\wb$ is a subsolution of the DPP. By the comparison principle,
		\[
		\wb(x,t)\leq \ue(x,t)\quad \text{for all } x\in\overline{\Omega},\ t\in(-\epsilon^2/2,\infty).
		\]
		Therefore,
		$$
		\liminf_{\epsilon\to 0^+,\, y\to x_0,\, s\to 0^+}\ue(y,s)
		\geq \liminf_{\epsilon\to 0^+,\, y\to x_0,\, s\to 0^+}\wb(y,s)
		=u_0(x_0)-\eta.
		$$
		Since $\eta>0$ is arbitrary, the claim follows. 
		
		The corresponding $\limsup$ inequality is obtained by repeating the above arguments with the upper barrier
		$$
		\wa(x,t)=k_1(e^{t+1}-1)
		+k_2\left(\|u_0\|_{\LL^\infty(\Omega)}+\|g\|_{\LL^\infty(\Rd\setminus\Omega)}\right)|x-x_0|^\beta
		+u_0(x_0)+\eta.
		$$
	\end{proof}
	
	Finally, recalling the definition of the operator $\A_\epsilon $ given in~\eqref{operator-A}, we can immediately state the following properties.
	
	\begin{rem}\label{properties-A}\rm
		Let $\ph, \ph_1, \ph_2:\Rd\times \RR^+\to \RR$ be measurable functions such that $\ph_1\geq\ph_2$, and let $\xi$ be a constant. Then the following identities hold:
		$$
		\A_\epsilon[\ph_1](x,t)\geq\A_\epsilon[\ph_2](x,t)
		\quad \text{and} \quad
		\A_\epsilon[\ph+\xi](x,t)=\A_\epsilon[\ph](x,t)+\xi.
		$$
	\end{rem}
	
	Now, we are ready to prove the uniform convergence, as $\epsilon\to 0^+$, of the solutions of the DPP to the viscosity solution of problem~\eqref{problem-intro-1}. 
	
	\begin{thm}\label{convergence}
		Let $d\geq 1$, $p>2$, $\epsilon >0$, and let $\Omega\subset\Rd$ be a bounded domain satisfying the uniform exterior ball condition. Let $u_0\in C_b(\Omega)$, $g\in C_b(\Rd\setminus\Omega)$ and assume that \eqref{m-M} holds. Let $\ue$ be the solution of the DPP~\eqref{DPP-main-1} (or equivalently \eqref{DPP}). For all $x\in\overline{\Omega}$ and $t\geq 0$ we consider
		\begin{align*}
			\ua(x,t)&:=\limsup_{\epsilon\to 0^+,\, y\to x,\, s\to t}\ue(y,s),\\
			\ub(x,t)&:=\liminf_{\epsilon\to 0^+,\, y\to x,\, s\to t}\ue(y,s).
		\end{align*}
		Then, it holds that 
		$$\ua(x,t)=\ub(x,t),$$ for all $(x,t)\in\overline{\Omega}\times\RR^+$, and consequently the limit as $\epsilon \to 0^+$ of $u_\epsilon$
		exists. 
		This limit coincides with $u$ the viscosity solution of~\eqref{problem-intro-1},
		\begin{equation}\label{convergence-1}
			\ue(x,t)\to u(x,t)=\ua(x,t)=\ub(x,t)
			\quad \mbox{uniformly in } \overline{\Omega}\times\RR^+
			\quad \mbox{as } \epsilon\to 0^+.
		\end{equation}
	\end{thm}

	\begin{proof}
		Define $\ua$ and $\ub$ as above for all $x\in\overline{\Omega}$ and $t\geq 0$. By construction, we have $\ub(x,t)\leq \ua(x,t)$. 
		Hence, uniform convergence will be established once we prove the reverse inequality. 
		This follows from the fact that $\ub(x,t)$ is a supersolution of \eqref{problem-intro-1}
		and $\ua(x,t)$ is a subsolution using the comparison principle for the parabolic problem (see Section~4 of~\cite{Lindgren_Hynd_2016}). 
		Therefore, if this holds, we conclude that the functional limit exists and coincides with $u$.
		
		Before proceeding, we list several facts that must be taken into account:
		
		On the one hand, if $x\in\partial\Omega$ and $t>0$, Lemma~\ref{bound-g} yields
		\[
		\ub(x,t)=\liminf_{\epsilon\to 0^+,\, y\to x,\, s\to t}\ue(y,s)
		\geq g(x)
		\geq \limsup_{\epsilon\to 0^+,\, y\to x,\, s\to t}\ue(y,s)
		=\ua(x,t).
		\]
		On the other hand, if $x\in\Omega$ and $t=0$, Lemma~\ref{bound-u0} implies
		\[
		\ub(x,0)=\liminf_{\epsilon\to 0^+,\, y\to x,\, s\to 0^+}\ue(y,s)
		\geq u_0(x)
		\geq \limsup_{\epsilon\to 0^+,\, y\to x,\, s\to 0^+}\ue(y,s)
		=\ua(x,0).
		\]
		
		Therefore, we only need to compare $\ua$ and $\ub$ at the points $x\in\Omega$ and $t>0$. First, let us show that 
		$\ua$ is a viscosity subsolution of~\eqref{problem-intro-1}.
		Let $(x_0,t_0)\in\Omega\times\RR^+$ and let $R,\tau>0$ be sufficiently small.  
		The upper semicontinuous function $\ua$ is a viscosity subsolution if, see Definition \ref{viscosity-dirichlet}, for every test function
		$\ph\in C^\infty(B_R(x_0)\times[t_0-\tau,t_0+\tau])$ such that $\ph(x_0,t_0)=\ua(x_0,t_0)$ and $\ph(x,t)>\ua(x,t)$ for all $(x,t)\neq(x_0,t_0)$, the following inequality holds:
		\[
		|\partial_t\ph(x_0,t_0)|^{p-2}\partial_t\ph(x_0,t_0)\leq \plap\ph(x_0,t_0).
		\]
		
		The condition above implies that $(x_0,t_0)$ is a strict maximum of $\ua-\ph$.  
		By~\cite{Barles_Souganidis_1991}, there exists a sequence
		$(\epsilon_n,y_n,t_n)\to(0^+,x_0,t_0)$ such that $(y_n,t_n)$ are "quasi-maximum" points of $u_{\epsilon_n}-\ph$, namely,
		\begin{equation}\label{quasi-max}
			u_{\epsilon_n}(x,t)-\ph(x,t)
			\leq u_{\epsilon_n}(y_n,t_n)-\ph(y_n,t_n)+e^{-1/\epsilon_n},
		\end{equation}
		for all $(x,t)\in B_R(x_0)\times[t_0-\tau,t_0+\tau]$.
		
		Notice that, for $n$ large enough, $(y_n,t_n)\in\Omega\times\RR^+$. Set
		$\xi_n:=u_{\epsilon_n}(y_n,t_n)-\ph(y_n,t_n)$. Using~\eqref{quasi-max}, we obtain
		\begin{align*}
			0
			&=\frac{u_{\epsilon_n}(y_n,t_n)-\A_{\epsilon_n}[u_{\epsilon_n}](y_n,t_n-\epsilon_n^2/2)}{\epsilon_n^2/2} \\
			&=\frac{\ph(y_n,t_n)+\xi_n-\A_{\epsilon_n}[u_{\epsilon_n}](y_n,t_n-\epsilon_n^2/2)}{\epsilon_n^2/2} \\
			&\geq \frac{\ph(y_n,t_n)+\xi_n-\A_{\epsilon_n}[\ph+\xi_n+e^{-1/\epsilon_n}](y_n,t_n-\epsilon_n^2/2)}{\epsilon_n^2/2} \\
			&=\frac{\ph(y_n,t_n)-e^{-1/\epsilon_n}-\A_{\epsilon_n}[\ph](y_n,t_n-\epsilon_n^2/2)}{\epsilon_n^2/2} \\
			&=\partial_t\ph(y_n,t_n-\epsilon_n^2/2)
			-\bigl(\plap\ph(y_n,t_n-\epsilon_n^2/2)\bigr)^{\frac{1}{p-1}}
			+E(\epsilon,\ph, y_n, t_n)
			-2\frac{e^{-1/\epsilon_n}}{\epsilon_n^2}.
		\end{align*}
		
		In the last line we have used Theorems~\ref{exp-c2}, \ref{exp-c2-grad}, \ref{exp-c3} and~\ref{exp-c3-grad}. The uniformity of the asymptotic error for $C^3$ test functions is crucial here, since it allows us to control the error uniformly in $y_n$. That is, if $\nabla \phi(x_0,t_0)\neq 0$, then there exists a neighbourhood of this point in which $\|\nabla \phi\|^{-1}$ is uniformly bounded. Moreover, for $\epsilon$ small enough, the points $(y_\epsilon,t_\epsilon)$ all lie within this neighbourhood, and estimate \eqref{error-para} is valid for all of them. 
		
		If, on the contrary, $\nabla \phi(x_0,t_0)=0$, then there exists a neighbourhood in which $\|\nabla \phi\|$ is small and, once the points $(y_\epsilon,t_\epsilon)$ all fall within this neighbourhood. Hence, estimate \eqref{error-para-grad} is valid for all of them.
		
		Taking limits as $n\to\infty$, by the regularity of $\ph$ we obtain
		\[
		0\geq \partial_t\ph(x_0,t_0)
		-\bigl(\plap\ph(x_0,t_0)\bigr)^{\frac{1}{p-1}}\Longrightarrow |\partial_t\ph(x_0,t_0)|^{p-2}\partial_t\ph(x_0,t_0)\leq \plap\ph(x_0,t_0).
		\]
		
		The proof of the fact that $\ub$ is a viscosity supersolution of~\eqref{problem-intro-1} follows in a similar way.
	\end{proof}
	
	
		\section{The associated game}
	As we said in Section \ref{Sec2.5}, the DPP we studied in Section \ref{sec4} is related to a game. Here, our goal is to describe the game and show that it has a value that coincides with the solution to the DPP. 
	We are dealing with a two-player (that we call Player $I$ and Player $II$) zero-sum game. There is a parameter $\epsilon>0$ and an initial position $(x_0,t_0)\in\Omega\times (0,\infty)$. Then, the players update the position according to the following rules. 
	\begin{enumerate}[label=(\roman*)]
		
		\item[(I)] At each round, a fair coin is tossed with probability $1/2$ for heads and $1/2$ for tails.
		
		\item[(II)(a)] If the outcome of (I) is heads, Player $I$ may either keep the spatial position $x_{k+1}=x_k$ and update the time $t_{k+1}=t_k-\epsilon^2/2$, or initiate the following procedure. Player $II$ selects a parameter $c\in[m(\epsilon),M(\epsilon)]$, after which a biased coin is tossed with probabilities $\alpha=(p-2)/(p-1)$ for heads and $1-\alpha$ for tails:
		\begin{itemize}[leftmargin=1em]
			\item If heads occurs, Player $I$ chooses $x_{k+1}\in B_{\epsilon^2 c^{1-\alpha}}(x_k)$ and $t_{k+1}=t_k-\epsilon^2/2$.
			\item If tails occurs, a round of tug-of-war with noise is played in $B_{\gamma\epsilon c^{-\alpha/2}}(x_k)$ with probabilities $\beta$ and $1-\beta$ (this game was described in Section \ref{Sec2.5}), which determines $x_{k+1}$, and $t_{k+1}=t_k-\epsilon^2/2$.
		\end{itemize}
		
		\item[(II)(b)] If the outcome of (I) is tails, the same procedure as in (II)(a) is followed, but with the roles of Player $I$ and Player $II$ interchanged.
		
		\item[(III)] The procedure in (I) and (II) is repeated until the game ends, this happens the first time  
		the spacial position lies outside $\Omega$ of when time is exhausted. When the game ends, the final payoff os the game is given by: if $t_{k+1}\le 0$ with $x_{k+1}\in\Omega$, Player $II$ pays $u_0(x_{k+1})$ to Player $I$; if $x_{k+1}\notin\Omega$ with $t_{k+1}>0$, Player $II$ pays $g(x_{k+1})$ to Player $I$.
	\end{enumerate}
	
	The game described above generates a sequence of states
	\[
	P=\left\{(x_0,t_0), (x_1,t_1), \ldots , (x_\tau,t_\tau)\right\},
	\]
	where $(x_0,t_0)\in \Omega \times (0,\infty)$ is the initial state of the game and $(x_\tau,t_\tau)$ is the final state (when either $x_\tau\in \Rd\setminus \Omega$ or $t_\tau\leq 0$). Observe that the game always finishes in a finite number of plays, since for any $t_0>0$ and $\epsilon>0$ there exists $k\in \NN$ such that $t_k=t_0-k\epsilon^2/2\leq 0$. Now, we write the final payoff function as 
	$$\Phi(x_\tau,t_\tau)=\left\{\begin{array}{ll} u_0(x_\tau) & \mbox{if }(x_\tau,t_\tau)\in\Omega \times (-\infty,0],\\[4pt] g(x_\tau) &\mbox{if } (x_\tau,t_\tau)\in \Rd\setminus \Omega\times \RR.\end{array}\right.$$
	
	A strategy for Player~$I$ is a function $S_I$ defined on the partial histories of the game that specifies the choices that Player~$I$ will make. On the one hand, if the first fair coin lands heads, it determines whether Player~$I$ remains at the same position $x_1 = x_0$ or proceeds to play the game described in~(II). That is, for any choice of $c \in [m(\epsilon), M(\epsilon)]$, $S_I$ determines the next position $x_1 \in B_{\epsilon^2 c^{1-\alpha}}(x_0)$ if the second biased coin lands heads with probability $\alpha$, as well as the next position when Player~$I$ plays tug-of-war with probabilities $\beta$ and $1-\beta$.
	
	On the other hand, if the first coin lands tails, $S_I$ specifies how Player~$I$ chooses $c \in [m(\epsilon), M(\epsilon)]$ when Player~$II$ decides not to remain at the same position and also chooses the next preferred position according to the second biased coin toss (when they play tug-of-war with noise).
	
	Similarly, a strategy for Player~$II$ is a function $S_{II}$, also defined on the partial histories, that specifies how this player acts at every round of the game.
	
	Once the two players have fixed their strategies, $S_I$ and $S_{II}$, everything depends on the initial position and the coin tosses, which induce a probability measure on the space of sequences of states. With this probability, we can compute the expected outcome when playing with strategies $S_I$ and $S_{II}$ starting from $(x_0,t_0)$,
	\[
	\mathbb{E}_{S_I, S_{II}}^{(x_0,t_0)}\left(\Phi(x_\tau,t_\tau)\right).
	\]
	
	Since Player~$I$ tries to maximize the expected outcome and Player~$II$ aims to minimize it, the value for Player~$I$ and Player~$II$ are
	defined as
	\[
	u^\epsilon_I(x_0,t_0)
	=
	\sup_{S_{I}}
	\left(
	\inf_{S_{II}}
	\mathbb{E}_{S_I, S_{II}}^{(x_0,t_0)}
	(\Phi(x_\tau,t_\tau))
	\right) \quad \mbox{and}\quad
	u^\epsilon_{II}(x_0,t_0)
	=
	\inf_{S_{II}}
	\left(
	\sup_{S_{I}}
	\mathbb{E}_{S_I, S_{II}}^{(x_0,t_0)}
	(\Phi(x_\tau,t_\tau))
	\right),
	\]
	respectively. Notice that it triavially holds $u^\epsilon_I\leq u^\epsilon_{II}$. We say that the game has a value when $u^\epsilon_I=u^\epsilon_{II}$.
	
	Before providing a rigorous proof that $\ue$, the solution to the DPP \eqref{DPP}, is the value function for this game, we present a heuristic interpretation. Assume for a moment that the first fair coin lands heads. Then Player~$I$, who seeks to maximize the payoff, decides whether to remain at the same position (thereby reducing the remaining time) or to move according to the game described in~(II)(a) and~(II)(b).
	If Player~$I$ decides to move, once Player~$II$ chooses $c\in[m(\epsilon),M(\epsilon)]$, if the first player selects the next position the expected value is given by 
	\[
	\sup_{B_{\epsilon^2 c^{1-\alpha}}(x_0)}\ue(\cdot,t_1),
	\]
	while when tug-of-war with noise is played, the expected outcome is given by
	\begin{align*}
		M_{\epsilon c^{-\alpha/2}}[\ue](x_0,t_1)
		&=\frac{\beta}{2}\Big( \sup_{B_{\gamma \epsilon c^{-\alpha/2}}(x_0)}\ue(\cdot,t_1)
		+ \inf_{B_{\gamma \epsilon c^{-\alpha/2}}(x_0)}\ue(\cdot,t_1) \Big) + (1-\beta)\intn_{B_{\gamma \epsilon c^{-\alpha/2}}(x_0)} \ue(y,t_1)\dy.
	\end{align*}
	
	Thus, taking into account that Player~$II$ wants to minimize the payoff, if the first coin toss is heads, 
	we obtain as expected outcome of one round of the game
	\[
	\max\left[\inf_{c\in[m,M]}\left\{\alpha\sup_{B_{\epsilon^2 c^{1-\alpha}(x_0)}}\ue(\cdot,t_1) + (1-\alpha) M_{\epsilon c^{-\alpha/2}}[\ue](x_0,t_1)\right\}, \, \ue(x_0,t_1)\right],
	\]
	that is, the first term of the operator $\A_\epsilon$. A similar argument yields the second term.

	\begin{thm}\label{game-thm}
		Let the assumptions of Theorem \ref{convergence} hold. Then, the game defined in {\rm (I), (II)(a), (II)(b)} and {\rm (III)} has a 
		value which is characterized by the unique solution $\ue$ to the DPP \eqref{DPP}, that is
		$$\ue(x,t)=u_I^\epsilon(x,t)=u_{II}^\epsilon(x,t)\qquad\mbox{for }(x,t)\in \Rd\times \RR.$$
		
		In particular, the value of the game converges uniformly in $\overline{\Omega}\times\RR^+$ as $\epsilon\to 0^+$ to the unique solution to the parabolic problem \eqref{problem-intro-1}.
	\end{thm}

	\begin{proof}
		\textit{Step 1:} First, we prove that $\ue \leq u^\epsilon_I$. 
		Take $\eta>0$ arbitrary. Given any strategy for the second player $S_{II}$, we choose a strategy $S_I^*$ for the first player starting from the point $(x_k,t_k)$ following the solution of the DPP~\eqref{DPP} as follows.
		
		On the one hand, if the first fair coin lands heads, Player~$I$ decides to remain at the same position if
		\[
		\ue(x_k,t_{k+1})>
		\inf_{c\in[m,M]}\left\{
		\alpha \sup_{B_{\epsilon^2 c^{1-\alpha}}(x_k)} \ue(\cdot,t_{k+1})
		+
		(1-\alpha) M_{\epsilon c^{-\alpha/2}}[\ue](x_k,t_{k+1})
		\right\}.
		\]
		Otherwise, Player~$II$ chooses any constant $c\in[m(\epsilon),M(\epsilon)]$ and Player~$I$ chooses the next position as follows. If the second biased coin lands heads, Player~$I$ chooses
		\begin{equation*}
			x_{k+1}^*\in B_{\epsilon^2 c^{1-\alpha}}(x_k)
			\quad \mbox{such that }\;
			\ue(x_{k+1}^*,t_{k+1})
			\geq
			\sup_{B_{\epsilon^2 c^{1-\alpha}}(x_k)}
			\ue(\cdot,t_{k+1})
			-\frac{\eta}{2^{k+1}}.
		\end{equation*}
		
		If the second biased coin lands tails, a tug-of-war with noise is played and, in this case Player~$I$ chooses
		\begin{equation} \label{ggg}
			\widehat{x}_{k+1}\in B_{\gamma \epsilon c^{-\alpha/2}}(x_k)
			\quad \mbox{such that }\;
			\ue(\widehat{x}_{k+1},t_{k+1})
			\geq
			\sup_{B_{\gamma \epsilon c^{-\alpha/2}}(x_k)}
			\ue(\cdot,t_{k+1})
			-
			\frac{\eta}{2^{k+1}}.
		\end{equation}
		
		On the other hand, if the first fair coin lands tails, and the second player chooses to play, Player~$I$ selects $c^*\in[m(\epsilon),M(\epsilon)]$ such that
		\begin{align*}
			\alpha\inf_{B_{\epsilon^2 (c^*)^{1-\alpha}}(x_k)}&\ue(\cdot,t_{k+1})
			+(1-\alpha)M_{\epsilon (c^*)^{-\alpha/2}}[\ue](x_k,t_{k+1})\\
			&\geq
			\sup_{c \in [m,M]}\left\{
			\alpha \inf_{B_{\epsilon^2 c^{1-\alpha}}(x_k)}\ue(\cdot,t_{k+1})
			+(1-\alpha)M_{\epsilon c^{-\alpha/2}}[\ue](x_k,t_{k+1})
			\right\}
			-\frac{\eta}{2^{k+1}},
		\end{align*}
		and next, in case it becomes necessary as a point for the tug-of-war with noise, a point such that \eqref{ggg} holds. 
		
		Fix the strategies $S_I^*$ and any $S_{II}$ and consider the sequence of random variables
		\[
		M_k=\ue(x_k,t_k)-\frac{\eta}{2^{k}}.
		\]
		We have that $(M_k)_{k\geq 0}$ is a submartingale, in fact,
		\begin{align*}
			&\mathbb{E}^{(x_0,t_0)}_{S_I^*, S_{II}}[M_{k+1}\mid (x_k,t_k)]
			=\frac{1}{2}\max\Bigg[
			\ue(x_k,t_{k+1})-\frac{\eta}{2^{k+1}}, \,
			\alpha \ue(x_{k+1}^*,t_{k+1})
			\\
			&\qquad \qquad \qquad \quad + (1-\alpha)\left(
			\frac{\beta}{2}\ue(\widehat{x}_{k+1},t_{k+1})
			+ \frac{\beta}{2}\ue(x_{k+1}^{II},t_{k+1})
			\right)
			+ (1-\alpha)(1-\beta)
			\fint_{B_{\gamma \epsilon c^{-\alpha/2}}(x_k)}
			\ue(y,t_{k+1})\,dy
			\Bigg]
			\\
			&\qquad\qquad \qquad \qquad  + \frac{1}{2}\min\Bigg[
			\ue(x_k,t_{k+1})-\frac{\eta}{2^{k+1}}, \,
			\alpha\inf_{B_{\epsilon^2 (c^*)^{1-\alpha}}(x_k)}\ue(\cdot,t_{k+1})
			+(1-\alpha)M_{\epsilon (c^*)^{-\alpha/2}}[\ue](x_k,t_{k+1})
			\Bigg]
			\\
			&\geq
			\A_\epsilon[\ue](x_k,t_k-\epsilon^2/2)
			-2\frac{\eta}{2^{k+2}}
			-\frac{\eta}{2^{k+2}}\left(\alpha+(1-\alpha)\beta\right)
			\\
			&\geq
			\ue(x_k,t_k)-\frac{\eta}{2^k}
			= M_k.
		\end{align*}
		
		Once we have that $(M_k)_{k\geq 0}$ is a submartingale, we apply the Optional Stopping Theorem 
		(cf.\cite{Williams}) in order to obtain
		\[
		\mathbb{E}^{(x_0,t_0)}_{S^*_I,S_{II}}
		[M_{\tau}]
		\geq
		M_0
		=
		\ue(x_0,t_0)-\eta.
		\]
		
		Since $S_{II}$ is any possible strategy, we have
		\[
		\inf_{S_{II}} \mathbb{E}^{(x_0,t_0)}_{S^*_I,S_{II}}
		[\Phi(x_\tau,t_\tau)] = \inf_{S_{II}} \mathbb{E}^{(x_0,t_0)}_{S^*_I,S_{II}}
		[\ue (x_\tau,t_\tau)]\geq \inf_{S_{II}} \mathbb{E}^{(x_0,t_0)}_{S^*_I,S_{II}}
		[M_{\tau}]
		\geq
		\ue(x_0,t_0)-\eta,
		\]
		and then, we conclude that
		\begin{align*}
			\ue(x_0,t_0)-\eta
			\leq
			\inf_{S_{II}}
			\mathbb{E}^{(x_0,t_0)}_{S_I^*, S_{II}}[\Phi(x_\tau,t_\tau)]
			\leq
			\sup_{S_I}
			\left(
			\inf_{S_{II}}
			\mathbb{E}^{(x_0,t_0)}_{S_I, S_{II}}[\Phi(x_\tau,t_\tau)]
			\right)
			=
			u_I^\epsilon(x_0,t_0).
		\end{align*}
		The first step is completed by the arbitrariness of $\eta$.
		
		\medskip
		\textit{Step 2:} We show that $u_{II}^\epsilon \leq \ue$.  
		The proof is analogous to Step~1. Let $\eta>0$ be arbitrary and, given any possible strategy $S_I$ for Player~$I$, we choose a strategy $S^*_{II}$ for the second player as follows.
		
		If the first fair coin lands heads and Player~$I$ decides to play, Player~$II$ chooses $c^*\in[m(\epsilon),M(\epsilon)]$ such that
		\begin{align*}
			\alpha\sup_{B_{\epsilon^2 (c^*)^{1-\alpha}}(x_k)}&\ue(\cdot,t_{k+1})
			+(1-\alpha)M_{\epsilon (c^*)^{-\alpha/2}}[\ue](x_k,t_{k+1})\\
			&\leq
			\inf_{c \in [m,M]}\left\{
			\alpha \sup_{B_{\epsilon^2 c^{1-\alpha}}(x_k)}\ue(\cdot,t_{k+1})
			+(1-\alpha)M_{\epsilon c^{-\alpha/2}}[\ue](x_k,t_{k+1})
			\right\}
			+\frac{\eta}{2^{k+1}}.
		\end{align*}
		
		On the other hand, if the first fair coin lands tails, Player~$II$ decides to remain at the same position if
		\[
		\ue(x_k,t_{k+1})<
		\sup_{c\in[m,M]}\left\{
		\alpha \inf_{B_{\epsilon^2 c^{1-\alpha}}(x_k)} \ue(\cdot,t_{k+1})
		+
		(1-\alpha) M_{\epsilon c^{-\alpha/2}}[\ue](x_k,t_{k+1})
		\right\}.
		\]
		
		Otherwise, Player~$I$ chooses any constant $c\in[m(\epsilon),M(\epsilon)]$ and Player~$II$ chooses the next position as follows. If the second biased coin lands heads,
		\begin{equation*}
			x_{k+1}^*\in B_{\epsilon^2 c^{1-\alpha}}(x_k)
			\quad \mbox{such that }\;
			\ue(x_{k+1}^*,t_{k+1})
			\leq
			\inf_{B_{\epsilon^2 c^{1-\alpha}}(x_k)}
			\ue(\cdot,t_{k+1})
			+\frac{\eta}{2^{k+1}}.
		\end{equation*}
		If the second biased coin lands tails, a tug-of-war with noise is played and Player~$II$ aims for
		\begin{equation*}
			\widehat{x}_{k+1}\in B_{\gamma \epsilon c^{-\alpha/2}}(x_k)
			\quad \mbox{such that }\;
			\ue(\widehat{x}_{k+1},t_{k+1})
			\leq
			\inf_{B_{\gamma \epsilon c^{-\alpha/2}}(x_k)}
			\ue(\cdot,t_{k+1})
			+
			\frac{\eta}{2^{k+1}}.
		\end{equation*}
		
		Fix the strategies $S_I$ and $S^*_{II}$ and define the random variables
		\[
		M_k=\ue(x_k,t_k)+\frac{\eta}{2^{k}}.
		\]
		As in Step~1, we show that $(M_k)_{k\geq 0}$ is a supermartingale, and applying again the Optimal Stopping Theorem, we obtain 
%
		\[
		\mathbb{E}^{(x_0,t_0)}_{S_I,S_{II}^*}[M_{\tau}]\leq
		M_0= \ue(x_0,t_0)+\eta,
		\]
		and hence
		\begin{align*}
			\ue(x_0,t_0)+\eta
			\geq
			\sup_{S_{I}}
			\mathbb{E}^{(x_0,t_0)}_{S_I, S_{II}^*}[\Phi(x_\tau,t_\tau)]\geq
			\inf_{S_{II}}
			\left(
			\sup_{S_{I}}
			\mathbb{E}^{(x_0,t_0)}_{S_I, S_{II}}[\Phi(x_\tau,t_\tau)]
			\right)
			=
			u_{II}^\epsilon(x_0,t_0).
		\end{align*}
		Since $\eta$ is arbitrary we have obtained 
		$
		\ue(x_0,t_0)
			\geq u_{II}^\epsilon(x_0,t_0)$. 
			
			Now, we observe that we trivially have $
		u_{I}^\epsilon (x_0,t_0)
			\leq u_{II}^\epsilon(x_0,t_0)$ and hence we conclude that
			$$
			\ue(x_0,t_0)
			=  u_{I}^\epsilon(x_0,t_0) = u_{II}^\epsilon(x_0,t_0),
			$$
			proving that the game has a value that coincides with the unique solution to the DPP \eqref{DPP-main-1}.
		
		The proof is finished using that the solution to the DPP 
		\eqref{DPP-main-1} converges as $\epsilon \to 0$. In fact, from our results 
		in Section \ref{sec4} it holds that $\ue \to u$ uniformly as $\epsilon\to 0^+$ (Theorem \ref{convergence}) where $u$ is the unique viscosity solution of the parabolic problem \eqref{problem-intro-1}.
	\end{proof}

	
	\section{The problem in the whole space}\label{sec5}
	
	We now present an analogous result for the convergence of the Dynamic
	Programming Principle in the whole space. Let $\epsilon \in (0,1)$ and
	consider the following problem:
	\begin{equation}\label{DPPII}
		\left\{
		\begin{array}{cl}
			\ue(x,t+\epsilon^2/2)=\A_\epsilon[\ue](x,t), & \mbox{in } \Rd \times (-\epsilon^2/2,\infty),\\[4pt]
			\ue(x,0)=u_0(x), & \mbox{in } \Rd\times (-\epsilon^2/2,0].
		\end{array}
		\right.
	\end{equation}
	Here $\A_\epsilon$ is given by~\eqref{operator-A}, and the initial datum
	$u_0 \in C_b^\lambda(\Rd)$ for some $\lambda>0$. Equivalently, we may write the problem as follows
	\begin{equation*}
		\left\{
		\begin{array}{cl}
			\ue(x,t)=\A_\epsilon[\ue](x,t-\epsilon^2/2), & \mbox{in } \Rd \times (0,\infty),\\[4pt]
			\ue(x,0)=u_0(x), & \mbox{in } \Rd\times (-\epsilon^2/2,0].
		\end{array}
		\right.
	\end{equation*}
	
	The main goal of this section is to show that the solutions to~\eqref{DPPII}
	converge, as $\epsilon \to 0^+$, to a viscosity solution of
	\begin{equation}\label{problemII}
		\left\{
		\begin{array}{cl}
			(\partial_t u  (x,t) )^{p-1}=\plap u(x,t),
			& \mbox{in } \Rd \times (0,\infty),\\[4pt]
			u(x,0)=u_0(x), & \mbox{in } \Rd.
		\end{array}
		\right.
	\end{equation}
	Note that, since uniqueness for solutions to \eqref{problemII} is not known as far as we are aware, the arguments 
	from the previous section cannot be repeated, and we must proceed in a different way. The strategy to prove 
	convergence is as follows: if the family $\{\ue\}_{\epsilon>0}$ is
	uniformly bounded and equicontinuous, then by the Arzelà--Ascoli theorem it
	admits a subsequence that converges to a function that will be a viscosity solution of
	\eqref{problemII}. This compactness argument proves convergence along subsequences $\epsilon_j \to 0+$.
	To obtain convergence to the whole family $\{\ue\}_{\epsilon>0}$ we need to assume an exponential decay
	for the initial datum. 
	
	\begin{rem}{\rm \label{rem-wholespace} \
			\begin{enumerate}[label=\roman*)] 
				\item All asymptotic expansions involving the operator $\A_\epsilon$ for
				$C^2$ and $C^3$ functions remain valid, since their proofs rely on local
				arguments.
				\item For each fixed $\epsilon>0$, existence and uniqueness of solutions to
				\eqref{DPPII} follow from the explicit Euler time-discretization scheme,
				as discussed in Section~\ref{sec4.1}. Moreover, the proofs of
				Theorem~\ref{comparison} and Lemma~\ref{contraction} can be repeated with
				minor modifications.
			\end{enumerate}
		}
	\end{rem}
	
	Using calculations analogous to those in Section~\ref{sec:heuristic}, we show how
	problem~\eqref{problemII} can be discretized. Let $\epsilon>0$, $\tau=\epsilon^2/2$
	and $p>2$. Then
	\[
	\frac{u(x,t+\tau)-u(x,t)}{\tau}=\Delta_{p,\epsilon}u(x,t):=\frac{\A_\epsilon[u](x,t)-u(x,t)}{\epsilon^2/2}.
	\]
	
	We now establish uniform boundedness and equicontinuity in space for solutions at a fixed time.

	\begin{thm}
		Let $d\geq 1$, $p>2$, and $\epsilon\in(0,1)$.
		Let $\ue$ be the solution of the DPP~\eqref{DPPII} with initial datum $u_0\in B(\Rd)$. Then for any $t>0$ and $y\in\Rd$
		\begin{equation}\label{eq5-1}
			\|\ue(\cdot+y,t)-\ue(\cdot,t)\|_{\LL^\infty(\Rd)}
			\leq \|u_0(\cdot+y)-u_0(\cdot)\|_{\LL^{\infty}(\Rd)}.
		\end{equation}
		In particular, if $u_0\in UC_b(\Rd)$, then $\ue(\cdot,t)\in UC_b(\Rd)$.
		Moreover, 
		\begin{equation}\label{eq5-2}
			\| \ue(\cdot,t)\|_{\LL^\infty(\Rd)}\leq \|u_0\|_{\LL^\infty(\Rd)},
		\end{equation}
		for every $t>0$. 
	\end{thm}

	\begin{proof}
		We begin proving \eqref{eq5-1}. Fix $y\in\Rd$ and $t>0$. Then there exist $t_0\in(-\epsilon^2/2,0]$
		and $j\in \NN$ such that $t=t_0+j\epsilon^2/2$.
		Applying Lemma~\ref{contraction} with $u(x,t)=\ue(x,t)$ and
		$v(x,t)=\ue(x+y,t)$ yields~\eqref{eq5-1}.
		
		Since the supremum of the initial datum is bounded in $\Rd$, the same computation implies inequality \eqref{eq5-2} for all $t>0$.
	\end{proof}
	
	The next goal is to prove a time-regularity estimate. First, we establish this inequality for discretized times of the mesh.
	\begin{thm}
		Let $d\geq 1$, $p>2$, $\epsilon\in(0,1)$, and $\lambda\in(0,2)$.
		Let $\ue$ be the solution of the DPP~\eqref{DPPII} with initial datum $u_0\in C_b^\lambda(\Rd)$. Write $t_k=t_0+ \,k\tau$ with $t_0\in(\epsilon^2/2,0]$ and $\tau=\epsilon^2/2$, then for every
		$k,j\in\NN$,
		\[
		\|\ue(\cdot,t_{k+j})-\ue(\cdot,t_k)\|_{\LL^\infty(\Rd)}
		\leq C \|u_0\|_{C^\lambda(\Rd)}(t_{k+j}-t_k)^{\lambda/2}.
		\]
		Here $C>0$ is a constant depending only on $p$, $d$, and $\lambda$,
		but independent of $t$ and $\epsilon$.
	\end{thm}
	\begin{proof}
		Assume without loss of generality that $t_j=j\tau$ with $\tau=\epsilon^2/2$, and define $u^j_\epsilon(x)=\ue(x,t_j)$ for all $j\in \NN$. We split the proof into several steps.
		
		\medskip
		\textit{Step 1. Assume temporarily that the initial datum belongs to
			$C_b^2(\Rd)$. Then, we will see that there exists a constant $C>0$, depending only on $p$
			and $d$, such that} 
		$$\|u^1_\epsilon-u_0\|_{\LL^\infty(\Rd)}\leq C \tau \|u_0\|_{C^2(\Rd)}.$$
		
		Using the discretization of the problem, for any $x\in\Rd$ we have
		\begin{align*}
			u^1_\epsilon(x)-u_0(x)
			&= \frac{\tau}{\epsilon^2}
			\max\Bigg[\inf_{c\in[m,M]}
			\Big\{\alpha \big(\sup_{B_{\epsilon^2 c^{1-\alpha}}(x)} u_0-u_0(x)\big)
			+ (1-\alpha)\big(M_{\epsilon c^{-\alpha/2}}[u_0](x)-u_0(x)\big)
			\Big\},\,0\Bigg]\\
			&\quad
			-\frac{\tau}{\epsilon^2}
			\max\Bigg[\inf_{c\in[m,M]}
			\Big\{\alpha \big(u_0(x)-\inf_{B_{\epsilon^2 c^{1-\alpha}}(x)} u_0\big)
			-(1-\alpha)\big(M_{\epsilon c^{-\alpha/2}}[u_0](x)-u_0(x)\big)
			\Big\},\,0\Bigg].
		\end{align*}
		Hence it suffices to estimate each term via Taylor expansions.
		Since $u_0$ is continuous, it attains its maximum and minimum on every closed ball.
		Then,
		\[
		\alpha \left|\sup_{B_{\epsilon^2 c^{1-\alpha}}(x)} u_0-u_0(x)\right|
		\leq \alpha \|\nabla u_0\|_{\LL^\infty(\Rd)}\epsilon^2 c^{1-\alpha},
		\]
				and
				\[\alpha \left|u_0(x)-\inf_{B_{\epsilon^2 c^{1-\alpha}}(x)}u_0\right|
				\leq \alpha \|\nabla u_0\|_{\LL^\infty(\Rd)}\epsilon^2 c^{1-\alpha}.\]
		Moreover, we estimate the remaining terms as follows
		\[
		(1-\alpha)\frac{\beta}{2}
		\left|\sup_{B_{\gamma \epsilon c^{-\alpha/2}}(x)}u_0
		+\inf_{B_{\gamma \epsilon c^{-\alpha/2}}(x)}u_0-2u_0(x)\right|
		\leq (1-\alpha)\frac{\beta}{2}\|D^2u_0\|_{\LL^\infty(\Rd)}
		\gamma^2 \epsilon^2 c^{-\alpha},
		\]
		and
		\[
		(1-\alpha)(1-\beta)
		\left|\fint_{B_{\gamma \epsilon c^{-\alpha/2}}(x)}u_0(y)\dy-u_0(x)\right|
		\leq (1-\alpha)(1-\beta)\|D^2 u_0\|_{\LL^\infty(\Rd)}
		\gamma^2 \epsilon^2 c^{-\alpha}.
		\]
		
		Combining the previous estimates, we get
		\begin{align*}
			\left|u^1_\epsilon(x)-u_0(x)\right|&\leq \frac{\tau}{\epsilon^2} \inf_{c \in [m,M]}\Big\{\alpha \epsilon^2 c^{1-\alpha}\|\nabla u_0\|_{\LL^\infty(\Rd)}+(1-\alpha)2(p+d)\epsilon^2 c^{-\alpha}\|D^2 u_0\|_{\LL^\infty(\Rd)} \Big\}\\
			&\leq 2\tau (p+d)\inf_{c \in [m,M]}\Big\{\alpha c^{1-\alpha} \|\nabla u_0\|_{\LL^\infty(\Rd)}+(1-\alpha)c^{-\alpha}\|D^2 u_0\|_{\LL^\infty(\Rd)}\Big\}\\
			&\leq 2\tau (p+d)\Bigg(\|\nabla u_0\|_{\LL^\infty(\Rd)}^{\frac{p-2}{p-1}}\|D^2 u_0\|_{\LL^\infty(\Rd)}^{\frac{1}{p-1}}\\
			&\qquad \qquad\qquad+ \frac{p-2}{p-1}\|\nabla u_0\|_{\LL^\infty(\Rd)}\epsilon^{\frac{2}{3p-4}}+\frac{1}{p-1}\|D^2 u_0\|_{\LL^\infty(\Rd)} \epsilon^{\frac{2(p-2)}{p}}\Bigg)\\
			&\leq 4 (p+d)\tau \|u_0\|_{C^2(\Rd)}.
		\end{align*}
		Taking the supremum over $x\in\Rd$ completes Step~1.
		
		\medskip
		\textit{Step 2. For every $k\in\NN$, we will show that} $$\|u^{k+1}_\epsilon-u^k_\epsilon\|_{\LL^\infty(\Rd)}
		\leq \|u^1_\epsilon-u_0\|_{\LL^\infty(\Rd)}.$$
		
		Fix $x\in\Rd$, then
		\begin{align*}
			\left|u^{k+1}_\epsilon(x)-u^{k}_\epsilon(x)\right|
			&=
			\left|\A_\epsilon[u^k_\epsilon](x)-\A_\epsilon[\ue^{k-1}](x)\right|\\
			&\le
			\frac{1}{2}
			\Bigg|
			\max\Bigg[
			\sup_{c\in[m,M]} \Bigg\{
			\alpha\Bigl(
			\sup_{B_{\epsilon^2 c^{1-\alpha}}(x)} \ue^k
			- \sup_{B_{\epsilon^2 c^{1-\alpha}}(x)} \ue^{k-1}
			\Bigr)\\
			&\qquad\qquad
			+ (1-\alpha)\Bigl(
			M_{\epsilon c^{-\alpha/2}}[\ue^k]
			- M_{\epsilon c^{-\alpha/2}}[\ue^{k-1}]
			\Bigr)
			\Bigg\},\,\ue^k-\ue^{k-1}
			\Bigg]
			\Bigg|\\
			&\quad
			+ \frac{1}{2}
			\Bigg|
			\max\Bigg[
			\sup_{c\in[m,M]} \Bigg\{
			\alpha\Bigl(
			\sup_{B_{\epsilon^2 c^{1-\alpha}}(x)} (-\ue^k)
			- \sup_{B_{\epsilon^2 c^{1-\alpha}}(x)} (-\ue^{k-1})
			\Bigr)\\
			&\qquad\qquad
			+ (1-\alpha)\Bigl(
			M_{\epsilon c^{-\alpha/2}}[-\ue^k]
			- M_{\epsilon c^{-\alpha/2}}[-\ue^{k-1}]
			\Bigr)
			\Bigg\},\,\ue^k-\ue^{k-1}
			\Bigg]
			\Bigg|,
		\end{align*}
		and we bound from above each term with $\|\ue^k-\ue^{k-1}\|_{\LL^\infty(\Rd)}$ as in Lemma \ref{contraction}. Taking the supremum over $x$ and iterating the obtained inequality yields the claim.
		
		\medskip
		\textit{Step 3. For every $k,j\in\NN$, the following inequality holds}
		\[
		\|\ue^{k+j}-\ue^k\|_{\LL^\infty(\Rd)}
		\leq j \|u^1_\epsilon-u_0\|_{\LL^\infty(\Rd)}.
		\]
		
		Indeed, $$\|\ue^{k+j}-\ue^k\|_{\LL^\infty(\Rd)}
		\leq \sum_{l=1}^{j}
		\|\ue^{k+l}-\ue^{k+l-1}\|_{\LL^\infty(\Rd)}
		\leq j \|\ue^1-u_0\|_{\LL^\infty(\Rd)}.$$
		
		\medskip
		\textit{Step 4. Assume $u_0\in C_b^\lambda(\Rd)$. Then we claim that there exists a
			constant $C>0$, depending only on $p$ and $d$, such that for all
			$k,j\in\NN$,}
		\[
		\|\ue^{k+j}-\ue^k\|_{\LL^\infty(\Rd)}
		\leq C \|u_0\|_{C^\lambda(\Rd)}(t_{k+j}-t_k)^{\lambda/2}.
		\]
		
		Let $\eta\in C_c^\infty(\Rd)$ be a standard mollifier with
		$\supp (\eta)\subset B_1(0)$ and $\int_{\Rd}\eta(x)\dx=1$.
		For $\delta>0$ define
		\[
		\eta_\delta(x)=\delta^{-d}\eta(x/\delta),
		\qquad\mbox{and}\qquad
		u_{0,\delta}=u_0*\eta_\delta\in C_b^\infty(\Rd).
		\]
		Then, a standard calculation shows that
		\[
		\|u_{0,\delta}-u_0\|_{\LL^\infty(\Rd)}
		\leq \delta^\lambda \|u_0\|_{C^\lambda(\Rd)},
		\qquad\mbox{and}\qquad
		\|D^k u_{0,\delta}\|_{\LL^\infty(\Rd)}
		\leq C_d \delta^{\lambda-k}\|u_0\|_{C^\lambda(\Rd)}.
		\]
		
		Let $u_{\epsilon,\delta}^j$ denote the solution of~\eqref{DPPII} with initial
		datum $u_{0,\delta}$ at time $t_j=j\tau$ as before. Then, using the estimates of the previous steps,
		\begin{align*}
			\|\ue^{k+j}-\ue^k\|_{\LL^\infty(\Rd)}&\leq \|\ue^{k+j}-u_{\epsilon,\delta}^{k+j}\|_{\LL^\infty(\Rd)}+ \|u_{\epsilon,\delta}^k-\ue^k\|_{\LL^\infty(\Rd)}+\|u_{\epsilon,\delta}^{k+j}-u_{\epsilon,\delta}^k\|_{\LL^\infty(\Rd)}\\
			&\leq 2 \|u_{0,\delta}-u_0\|_{\LL^\infty(\Rd)} + \tau j \sup_{\epsilon\in(0,1)} \| \Delta_{p,\epsilon} u_{0,\delta}\|_{\LL^\infty(\Rd)}\\
			&\leq 2 \delta^\lambda \|u_0\|_{C^\lambda(\Rd)} + 8 (p+d) t_j \|u_{0,\delta}\|_{C^2(\Rd)}\\
			&\leq C_{d,p}  \|u_0\|_{C^\lambda}\left(\delta^\lambda + t_j \delta^{\lambda-2}\right).
		\end{align*}
		
		Choosing $\delta=t_j^{1/2}$ gives us $\|\ue^{k+j}-\ue^k\|_{\LL^\infty(\Rd)}\leq C_{d,p} t_j^{\lambda/2}\|u_0\|_{C^\lambda(\Rd)}.$
	\end{proof}

    Note that, here, when we write $C^\lambda(\Rd)$ with $\lambda\in(0,2)$, we mean $C^{0,\lambda}(\Rd)$ if $\lambda\in(0,1)$ and $C^{1,\lambda}(\Rd)$ if $\lambda\in[1,2)$. The previous result can be easily extended if the initial datum has a different modulus of continuity $\Lambda_{u_0}(|y|)$ instead of $|y|^{\lambda}$. Only Step~4 changes when we use the mollifier. In this case, choosing $\delta=t_j^{1/2}$, we obtain
				$$
				\|\ue^{k+j}-\ue^k\|_{\LL^\infty(\Rd)}\leq C_{d,p}\,\Lambda_{u_0}(t_j^{1/2}).
				$$
	
				
	
	In summary, we have shown that, given a time grid $t_j=t_0+j\tau$ with $t_0\in(-\tau,0]$, the family of DPP solutions is equicontinuous at the grid times. We now aim to extend this result to infer equicontinuity at arbitrary times. To this end, we introduce the following linear interpolation
	\begin{equation}\label{interpolation-lin}
		U_\epsilon(x,t)=\frac{t_{j+1}-t}{\tau}\ue^j(x)+\frac{t-t_j}{\tau}\ue^{j+1}(x)\qquad \mbox{for all }t\in [t_j,t_{j+1}).
	\end{equation}
	
	Now, we show that this family of functions are equicontinuous.
	
	\begin{thm}
		Let $d\geq 1$, $p>2$, $\epsilon\in(0,1)$, $\tau=\epsilon^2/2$ and $\lambda\in(0,2)$.
		Let $\ue$ be the solution of the DPP~\eqref{DPPII} with initial datum $u_0\in C_b^\lambda(\Rd)$. Assume that~\eqref{m-M} holds, and define $U_\epsilon(x,t)$ as in~\eqref{interpolation-lin}. Then, there exists a positive constant $C$, depending only on $p$ and $d$, such that for all $t,s>0$,
		\begin{equation}\label{equi-inter}
			\|U_\epsilon(\cdot,t)-U_\epsilon(\cdot,s)\|_{\LL^\infty(\Rd)}
			\leq C \|u_0\|_{C^\lambda(\Rd)} |t-s|^{\lambda/2}.
		\end{equation}
	\end{thm}
	
	\begin{proof}
		We distinguish two cases.
		
		\textit{Case 1: $t$ and $s$ belong to the same interval, that is, $t_j \le t < s < t_{j+1}$ for some $j\in\NN$.}
		Let $x\in\Rd$. By the definition of $U_\epsilon$,
		\[
		U_\epsilon(x,t)-U_\epsilon(x,s)
		=\frac{s-t}{\tau}\big(\ue^j(x)-\ue^{j+1}(x)\big).
		\]
		Hence,
		\begin{equation*}
			|U_\epsilon(x,t)-U_\epsilon(x,s)|
			\le \frac{s-t}{\tau}\, C \|u_0\|_{C^\lambda(\Rd)}\,\tau^{\lambda/2} = C \|u_0\|_{C^\lambda(\Rd)}
			|s-t|^{\lambda/2}\left(\frac{s-t}{\tau}\right)^{1-\lambda/2}.
		\end{equation*}
		Since $t,s\in[t_j,t_{j+1})$ implies $|t-s|<\tau$, we obtain \eqref{equi-inter} after taking supremum over $x\in\Rd$.
		
		\textit{Case 2: $t$ and $s$ belong to different intervals, namely $t\in[t_j,t_{j+1})$ and $s\in[t_k,t_{k+1})$ with $j<k$.}
		Let $x\in\Rd$. Then, adding and subtracting $U_\epsilon(x,t_{j+1})$ and $U_\epsilon(x,t_{k})$ we obtain
		\begin{align*}
			|U_\epsilon(x,t)-U_\epsilon(x,s)|
			&\le C \|u_0\|_{C^\lambda(\Rd)}
			\big(|t-t_{j+1}|^{\lambda/2}
			+ |t_{j+1}-t_k|^{\lambda/2}
			+ |t_k-s|^{\lambda/2}\big) \\
			&\le 3C \|u_0\|_{C^\lambda(\Rd)} |t-s|^{\lambda/2}.
		\end{align*}
		Taking supremum over $x\in\Rd$ concludes the proof.
	\end{proof}
	
	In conclusion, the family of functions $\left\{U_\epsilon\right\}_{\epsilon>0}$ is uniformly bounded and equicontinuous. Therefore, by the Arzelà--Ascoli theorem, there exists a subsequence that converges locally uniformly to a continuous limit $u$ as $\epsilon \to 0^+$. Before verifying that this limit function is a viscosity solution of problem \eqref{problemII}, we show that the left-constant interpolation
	\begin{equation}\label{interpolation-const}
		V_\epsilon(x,t)=\ue^j(x)\qquad \mbox{if }t\in [t_j,t_{j+1}),
	\end{equation}
	converges locally uniformly to the same function $u$ as $\epsilon \to 0^+$.
	
	We introduce this new interpolation because we do not know whether equation~\eqref{DPPII} is satisfied by $U_\epsilon$ at times outside the temporal grid $\{t_j\}_{j\in\NN}$. However, it is satisfied by $V_\epsilon$. The latter interpolation is precisely the one used to construct the solutions $\ue$ of~\eqref{DPPII} via the explicit Euler method on the grid $\{t_j\}$. Hence, we conclude that $\ue$, the unique solution to the DPP \eqref{DPPII}, also converges to the same limit $u$, that is, 
	$$\ue\to u \qquad \mbox{locally uniformly in }\Rd\times [0,\infty)\quad \mbox{as }\epsilon \to 0^+.$$
	
	\begin{lem}
		Let $d\geq 1$, $p>2$, $\epsilon\in(0,1)$, $\tau=\epsilon^2/2$, and $\lambda\in(0,2)$.
		Let $\ue$ be the solution of the DPP~\eqref{DPPII} with initial datum $u_0\in C_b^\lambda(\Rd)$. Assume that~\eqref{m-M} holds, and define $U_\epsilon(x,t)$ and $V_\epsilon(x,t)$ as in~\eqref{interpolation-lin} and~\eqref{interpolation-const}, respectively. Take a (locally uniform) convergent subsequence of $\{U_\epsilon\}_{\epsilon>0}$. Then, the same subsequence of $\{V_\epsilon\}_{\epsilon>0}$ converges to the same limit.
	\end{lem}
	
	\begin{proof}
		Let $\delta>0$ be arbitrary, and fix $x\in\Rd$ and $t\in[t_j,t_{j+1})$. Using that $U_\epsilon \to u$ locally uniformly, we obtain
		\begin{align*}
			|V_\epsilon(x,t)-u(x,t)|
			&\le |V_\epsilon(x,t)-U_\epsilon(x,t)|
			+ |U_\epsilon(x,t)-u(x,t)| \\
			&= \frac{t-t_j}{\tau}\,|\ue^j(x)-\ue^{j+1}(x)|
			+ |U_\epsilon(x,t)-u(x,t)| \\
			&\le C \|u_0\|_{C^\lambda(\Rd)}
			\left(\frac{\epsilon^2}{2}\right)^{\lambda/2}
			+ |U_\epsilon(x,t)-u(x,t)|.
		\end{align*}
		Choosing $\epsilon>0$ sufficiently small, both terms on the right-hand side are bounded by $\delta/2$, and therefore $$|V_\epsilon(x,t)-u(x,t)| \le \delta.$$ Taking the supremum in any compact subset of $\mathbb{R}^d$ yields $V_\epsilon \to u$ locally uniformly in $\Rd\times [0,\infty)$.
	\end{proof}
	
	Now, we verify that any possible uniform limit of the DPP solutions, $u$,  is a viscosity solution of \eqref{problemII}. 
	
	\begin{thm} \label{teo.sol.Rd}
		Let $d\geq 1$, $p>2$, $\lambda\in(0,2)$, and $\epsilon\in(0,1)$. Let $u_0\in C_b^\lambda(\Rd)$, and assume that \eqref{m-M} holds. Let $\ue$ denotes the solution of the DPP~\eqref{DPPII}. If $u$ is
		any possible limit along a subsequence of $\ue$, 
		$$\ue \to u\qquad \mbox{locally uniformly in }\Rd\times[0,\infty) \quad \mbox{as }\epsilon\to 0^+,$$
		then $u$ is a viscosity solution to~\eqref{problemII}.
	\end{thm}

	\begin{proof}
		Here we prove only that the limit $u$ is a viscosity subsolution since checking that it is a viscosity supersolution is analogous.
		
		On the one hand, it is immediate to check that $u(x,0)\le u_0(x)$ in $\Rd$ by the uniform convergence. Indeed,
		\[
		u(x,0)\le |u(x,0)-\ue(x,0)|+u_0(x),
		\]
		and the first term goes to $0$ as $\epsilon\to 0$.
		
		On the other hand, let $(x_0,t_0)\in\Rd\times(-\tau,\infty)$ and let $\ph\in C^\infty$ in a neighborhood of this point $(x_0,t_0)$ be such that $\ph(x_0,t_0)=u(x_0,t_0)$ and $\ph(x,t)>u(x,t)$ for all $(x,t)\ne (x_0,t_0)$. Then, we want to show that
		$$	(\partial_t\ph(x_0,t_0))^{p-1}\le \plap\ph(x_0,t_0).$$
		Since $(x_0,t_0)$ is a maximum point of $u-\ph$, it is standard (see for example \cite{Crandall_Ishii_Lions_1992}) that for every $\epsilon>0$, if $\ue$ converges uniformly to $u$ as $\epsilon\to 0^+$, there exists a sequence $(x_\epsilon,t_\epsilon)\to(x_0,t_0)$ as $\epsilon\to 0^+$ such that each $(x_\epsilon,t_\epsilon)$ is a maximum point of $\ue-\ph$. Hence, using that $\ue$ satisfies equation \eqref{DPPII} and the results concerning asymptotic expansions exposed in Section \ref{sec3}, we obtain
		\begin{align*}
			0
			&=\frac{\ue(x_\epsilon,t_\epsilon)-\A_\epsilon[\ue](x_\epsilon,t_\epsilon-\epsilon^2/2)}{\epsilon^2/2} \\
			&=\frac{\ph(x_\epsilon,t_\epsilon)+\xi_\epsilon-\A_\epsilon[\ue](x_\epsilon,t_\epsilon-\epsilon^2/2)}{\epsilon^2/2} \\
			&\ge \frac{\ph(x_\epsilon,t_\epsilon)+\xi_\epsilon-\A_\epsilon[\ph](x_\epsilon,t_\epsilon-\epsilon^2/2)-\xi_\epsilon}{\epsilon^2/2} \\
			&=\partial_t \ph(x_\epsilon,t_\epsilon-\epsilon^2/2)
			-\bigl(\plap\ph(x_\epsilon,t_\epsilon-\epsilon^2/2)\bigr)^{\frac{1}{p-1}}+E(\epsilon,\ph),
		\end{align*}
		where $\xi_\epsilon=\ue(x_\epsilon,t_\epsilon)-\ph(x_\epsilon,t_\epsilon)$. Moreover, the error can be controlled regardless of the point $(x_\epsilon,t_\epsilon)$ as we explained before in the proof of Theorem \ref{convergence}. Taking limits as $\epsilon\to 0^+$ and using the regularity of $\ph$ yields the desired inequality.
	\end{proof}

Finally, let us prove that we have convergence of the whole family $u_\epsilon$ as $\epsilon \to 0^+$ when 
the initial condition decays exponentially. To this end, we need the following one-dimensional lemma
whose proof is straitforward. 

\begin{lem} \label{lemma.supersol.expon}
Let $U : \mathbb{R} \times [0,\infty) \mapsto \mathbb{R}$ be given by $U (x,t) = C e^{L t} e^{x}.$ Then
\begin{equation} \label{solu.U}
U_t (x,t) - ((|U_x|^{p-2} U_x)_x (x,t))^{\frac{1}{p-1}} = (L-(p-1)^{\frac{1}{p-1}}) U(x,t).
\end{equation}
\end{lem}
\begin{proof}
The result follows from the next computation
$$
U_t (x,t) = C L e^{L t} e^{x} \quad \mbox{and} \quad
U_x (x,t) = C e^{L t} e^{x}.
$$
\end{proof}

With this function $U$ at hand we can obtain a uniform smallness for the whole family $u_\epsilon$
far from the origin. 

\begin{thm} \label{teo.unif.decay}
	Let $d\geq 1$, $p>2$, $\epsilon\in(0,1)$ and assume that \eqref{m-M} holds. If $\ue$ is the solution of the DPP \eqref{DPPII} with initial datum $u_0$ which satisfies
	\begin{equation}\label{exp.decay.u0.22}
		|u_0(x)| \leq C e^{-|x|}.
	\end{equation}
	Then, given $\eta >0$ and $T>0$, there exist $\epsilon_0>0$ and $K>0$ such that for all $|x| \geq K$, $t \in [0,T]$, and $0<\epsilon<\epsilon_0$,
	\begin{equation}\label{cota.unif}
		|u_\epsilon (x,t)| \leq \eta.
	\end{equation}
\end{thm}

\begin{proof}
Take any direction $z\in \mathbb{R}^d$ with $|z|=1$, and consider
$$
U^z (x,t) = C e^{L t} e^{\langle x, z \rangle},
$$
with $L=2(p-1)^{\frac{1}{p-1}}$.
Observe that
$$
|u_0(x)| \leq C e^{-|x|} \leq U^z (x,0) = C e^{\langle x, z \rangle}.
$$

The function $U^z$ belongs to $C^3$ and has a non-vanishing gradient. Hence, we can apply Theorem \ref{thm:AMVFparab-intro}, note that $\beps\in (0,\min[R,\tau])$ is fixed, to obtain
\begin{equation*}
	\A_\epsilon[ U^z ](x,t)= U^z \Big(x,t+\frac{\epsilon^2}{2}\Big)
	+\frac{\epsilon^2}{2}\Big(\left(\plap U^z (x,t)\right)^{\frac{1}{p-1}}-\partial_t U^z (x,t)\Big)
	+\epsilon^2 \widehat{E}(\epsilon, U^z, x,t),
\end{equation*}
where, if $k_1$ and $k_2$ are positive constants depending only on $p$ and $d$, the error term satisfies
\begin{align*}
	\left|\widehat{E}(\epsilon, U^z ,x,t)\right|
	&\leq \epsilon^2 \|\partial_{tt}U^z \|_{\LL^\infty([t-\beps^2/2,t+\beps^2/2])}
	+ k_1 \epsilon^{2-\frac{4}{p}} \|D^2 U^z\|_{\LL^\infty(B_{\beps}(x))}\\
	&\quad + k_2 \epsilon^{\frac{2}{3p-4}}
	\left(\|\gr U^z \|_{\LL^\infty(B_{\beps}(x))}
	+\|D^3 U^z \|_{\LL^\infty(B_{\beps}(x))}
	+\frac{|D^2 U^z (x,t)|^2}{|\gr U^z(x,t)|}\right).
\end{align*}

 Now, note that the derivatives that appear in the error estimation can be bounded by the same function $U^z$, for example $$
 \begin{array}{l}
\displaystyle  \| D^2 U^z\|_{\LL^\infty (B_{\beps}(x))}= \|C e^{L t} e^{\langle x+\cdot,z\rangle}z_i z_j\|_{\LL^\infty(B_1)} \\ [4pt]
\displaystyle \quad \leq C e^{Lt}e^{\langle x,z\rangle}\|e^{\langle\cdot,z\rangle}z_i z_j\|_{\LL^\infty(B_1)} = U^z(x,t) 
 \|e^{\langle\cdot,z\rangle}z_i z_j\|_{\LL^\infty(B_1)}.
\end{array}$$ 
 Therefore, there exists a constant $k_3>0$ (depending only on $d,p$, and $\beps$) such that
$$
\|\partial_{tt}U^z \|_{\LL^\infty([t-\beps^2/2,t+\beps^2/2])} \leq k_3 U^z (x,t),
$$
$$
\|D^2 U^z\|_{\LL^\infty(B_{\beps}(x))} \leq k_3 U^z (x,t),
$$
and
$$
\|\gr U^z \|_{\LL^\infty(B_{\beps}(x))}
+\|D^3 U^z \|_{\LL^\infty(B_{\beps}(x))}
+\frac{|D^2 U^z (x,t)|^2}{|\gr U^z(x,t)|}
\leq k_3 U^z (x,t).
$$

Therefore, using Lemma \ref{lemma.supersol.expon}, we obtain that for $\alpha = \min \left\{ 2-\frac{4}{p}, \frac{2}{3p-4} \right\} >0$
\begin{align*}
	U^z \Big(x,t+\frac{\epsilon^2}{2}\Big) - \A_\epsilon[ U^z ](x,t)
	&= \frac{\epsilon^2}{2}\Big(\left(\partial_t U^z (x,t)- \plap U^z (x,t)\right)^{\frac{1}{p-1}} \Big)
	+\epsilon^2 \widehat{E}(\epsilon, U^z, x,t)\\
	&\geq \frac{\epsilon^2}{2} \Big( (L-(p-1)^{\frac{1}{p-1}}) - 3 \epsilon^{\alpha} k_3 \Big) U^z (x,t),
\end{align*}

We conclude that, as $L-(p-1)^{\frac{1}{p-1}}>0$, there exists $\epsilon_0>0$ such that, for all $\epsilon < \epsilon_0$,
$$
U^z \Big(x,t+\frac{\epsilon^2}{2}\Big) - \A_\epsilon[ U^z ](x,t) \geq 0,
$$
for every $(x,t) \in \mathbb{R}^d \times [0,T]$.
Hence, for $\epsilon$ small enough, $U^z (x,t)$ is a supersolution to the DPP in $\mathbb{R}^d \times [0,T]$, with $U^z (x,0) \geq u_0 (x)$.
By the comparison principle for the DPP, see Remark \ref{rem-wholespace} and Theorem \ref{comparison}, we obtain
$$
u_\epsilon (x,t) \leq U^z(x,t) = C e^{L t} e^{\langle x, z \rangle}
\leq C e^{L T} e^{\langle x, z \rangle},
$$
for every $(x,t) \in \mathbb{R}^d \times [0,T]$.
Since $z$ is arbitrary, we deduce that
$$
u_\epsilon (x,t) \leq \min_{|z|=1}\left\{ C e^{L T} e^{\langle x, z \rangle}\right\}
=  C e^{L T} e^{-|x|}
\leq \eta,
$$
for every $|x| \geq K$ and every $t \in [0,T]$, provided $\epsilon < \epsilon_0$.

A similar argument applied to $-U^z (x,t) = - C e^{L t} e^{\langle x, z \rangle}$ yields
$$
u_\epsilon (x,t) \geq - C e^{L T} e^{-|x|}
\geq - \eta,
$$
for every $|x| \geq K$ and every $t \in [0,T]$, for $\epsilon < \epsilon_0$, which completes the proof.
\end{proof}

With this result at hand we can show that the whole family $u^\epsilon$ converges uniformly. 

\begin{thm} \label{teo.conv.unif.6}
	Let the hypotheses of Theorem \ref{teo.sol.Rd} holds. Let $u_\epsilon$ denote the solution to the DPP~\eqref{DPPII} with initial datum $u_0\in C_b^\lambda(\Rd)$, which satisfies
	\begin{equation}\label{exp.decay.u0.44}
		|u_0(x)| \leq C e^{-|x|}.
	\end{equation}
	Then, for every $T>0$,
	\[
	u_{\epsilon} \to u 
	\quad \text{uniformly in } \Rd\times[0,T]
	\quad \text{as } \epsilon \to 0^+,
	\]
	where $u \in UC(\Rd\times [0,\infty))$ is a viscosity solution to \eqref{problemII.intro}.
\end{thm}

\begin{proof}
	From Theorem \ref{teo.sol.Rd}, there exists a subsequence $\epsilon_j$ such that $u_{\epsilon_j}$
	converges locally uniformly to a limit $u$, which is a viscosity solution to \eqref{problemII.intro}. 
    
    Using the previous result, given $\eta>0$, we can choose $K>0$ large enough such that 
	\begin{equation}\label{cota.unif.66}
		|u_\epsilon (x,t)| \leq \eta
	\end{equation}
	for every $|x| \geq K$ and every $t \in [0,T]$, provided $\epsilon < \epsilon_0$.
	Hence, passing to the limit, we also have
	\begin{equation}\label{cota.unif.77}
		|u (x,t)| \leq \eta
	\end{equation}
	for every $|x| \geq K$ and every $t \in [0,T]$.
	
	Now, let $u_K$ denote the solution to
	\begin{equation}\label{problem-R}
		\left\{
		\begin{array}{rlll}
			\left( \partial_t u(x,t) \right)^{p-1} &=& \plap u(x,t), & \text{in } B_K\times (0,T),\\ [4pt]
			u(x,0)&=&u_0(x), & \text{in } B_K,\\ [4pt]
			u(x,t)&=& u_0 (x), & \text{on } \partial B_K \times (0,T),
		\end{array}
		\right.
	\end{equation}
	and let $u_{\epsilon,K}$ denote the solution to the corresponding DPP:
	\begin{equation}\label{DPP-R}
		\left\{
		\begin{array}{rlll}
			u_\epsilon(x,t)&=&\A_\epsilon[u_\epsilon](x,t-\epsilon^2/2), & \text{in } B_K\times (0,T),\\ [4pt]
			u_\epsilon(x,t)&=&u_0(x), & \text{in } B_K\times (-\epsilon^2/2,0],\\ [4pt]
			u_\epsilon(x,t)&=&u_0 (x), & \text{in } \Rd\setminus B_K \times (-\epsilon^2/2, T).
		\end{array}
		\right.
	\end{equation}
	
	Since $u$ satisfies \eqref{cota.unif.77} and solves \eqref{problemII.intro}, 
	the comparison principle (see \cite{Lindgren_Hynd_2016}) for \eqref{problem-R} yields
	\begin{equation}\label{ppp}
		u(x,t) - \eta \leq u_R (x,t) \leq u(x,t) + \eta,
	\end{equation}
	for every $|x| \leq K$ and every $t \in [0,T]$.
	Indeed, the function $w(x,t) = u(x,t) + \eta$ is a supersolution of 
	$\left( \partial_t w(x,t) \right)^{p-1}= \plap w(x,t)$ in $B_K \times (0,T)$,
	with $w(x,t) \geq u_0(x)$ on $\partial B_K \times (0,T)$ and $w(x,0) \geq u_0(x)$ in $B_K$.
	Thus, by the comparison principle, $u_K (x,t) \leq w(x,t) = u(x,t) + \eta$.
	The lower bound follows analogously by considering $w(x,t) = u(x,t) - \eta$.
	
	A similar argument, using the comparison principle for \eqref{DPP-R}, yields
	\begin{equation}\label{ppp.33}
		u_{\epsilon} (x,t) - \eta \leq u_{\epsilon,K} (x,t) \leq u_\epsilon (x,t) + \eta,
	\end{equation}
	for every $|x| \leq K$ and every $t \in [0,T]$.
	
	Combining \eqref{cota.unif.66}, \eqref{cota.unif.77}, \eqref{ppp}, and \eqref{ppp.33}, for $K$ large enough and $\epsilon < \epsilon_0$, we obtain
	\[
	\begin{array}{l}
		\displaystyle \| u - u_{\epsilon} \|_{L^\infty (\mathbb{R}^d \times [0,T]) } 
		\leq \| u - u_{\epsilon} \|_{L^\infty (B_K \times [0,T]) } + 2 \eta \\[8pt]
		\qquad \displaystyle \leq
		\| u_K - u_{\epsilon,K} \|_{L^\infty (B_K \times [0,T]) }
		+ \| u - u_{K} \|_{L^\infty (B_K \times [0,T]) } 
		+ \| u_\epsilon - u_{\epsilon,K} \|_{L^\infty (B_K \times [0,T]) } + 2 \eta \\[8pt]
		\qquad \leq \| u_K - u_{\epsilon,K} \|_{L^\infty (B_K \times [0,T]) }  + 4 \eta.
	\end{array}
	\]
	
	The convergence results for bounded domains established in Section \ref{sec4} imply that, for fixed $K>0$,
	\[
	\lim_{\epsilon \to 0^+} \| u_K - u_{\epsilon,K} \|_{L^\infty (B_K \times [0,T]) } = 0.
	\]
	Therefore,
	\[
	\limsup_{\epsilon \to 0^+}\| u - u_{\epsilon} \|_{L^\infty (\mathbb{R}^d \times [0,T]) } 
	\leq 4 \eta.
	\]
	Since $\eta>0$ is arbitrary, we conclude that
	\[
	\lim_{\epsilon \to 0^+}\| u - u_{\epsilon} \|_{L^\infty (\mathbb{R}^d \times [0,T]) } = 0,
	\]
	which completes the proof.
\end{proof}

	\paragraph{The game in $\mathbb{R}^d$.}

If we consider the problem in the whole space, the associated game is the same as the one described in this section, except that it ends only when the time becomes nonpositive, that is, when $t_\tau \le 0$, since there is no possibility of exiting a domain.

Repeating exactly the arguments of Theorem~\ref{game-thm}, one can prove that this game has a value and that it coincides with the solution of the DPP~\eqref{DPPII.intro}. Therefore, we conclude that the value of the game converges, up to a subsequence, as $\epsilon \to 0^+$, to a solution of the parabolic problem posed in the whole space~\eqref{problemII.intro}. In addition, if the initial condition decays exponentially fast, we have the convergence of the 
whole family $u_\epsilon$ as $\epsilon \to 0^+$. 
	
	\appendix
	
	\section{Viscosity solutions}\label{ap:viscosity}

	We devote this appendix to discuss the notion of viscosity solutions of the doubly-nonlinear parabolic equation and also for the
	asymptotic mean value formula used in this work.
	
	Let us start with the definition of a viscosity solution to the equation
	\begin{equation}\label{equation.A}
		\left( \partial_t u (x,t) \right)^{p-1}- \plap u (x,t) = 0.
	\end{equation}

	\begin{defn}[\textbf{Viscosity Solution I}]\label{viscosity-equation}
		A bounded upper semicontinuous function $u:\Omega \times (0,T)\to \RR$ is a viscosity subsolution of the parabolic equation \eqref{equation.A} in $\Omega\times (0,T)$ if, whenever $(x_0,t_0)\in \Omega\times (0,T)$ and $\ph\in C^2(B_R(x_0)\times [t_0-R,t_0+R])$ for some $R>0$ are such that $\ph(x_0,t_0)=u(x_0,t_0)$ and $\ph(x,t)>u(x,t)$ for all $(x,t)\in B_R(x_0)\times [t_0-R,t_0]$, then
		\[
		\left(\partial_t \ph(x_0,t_0)\right)^{p-1}\leq \plap \ph(x_0,t_0).
		\]
		Similarly, a bounded lower semicontinuous function $u$ is a viscosity supersolution of \eqref{equation.intro} in $\Omega\times (0,T)$ if, under the same assumptions but with $\ph(x_0,t_0)=u(x_0,t_0)$ and $\ph(x,t)<u(x,t)$ in $B_R(x_0)\times [t_0-R,t_0]$, we have
		\[
		\left(\partial_t \ph(x_0,t_0)\right)^{p-1}\geq \plap \ph(x_0,t_0).
		\]
		Moreover, a bounded continuous function $u$ is a viscosity solution of \eqref{equation.intro} if it is both a viscosity subsolution and a viscosity supersolution.
	\end{defn}
	
	This notion of solution is equivalent to the following one in which our mean value formula appears. 
	
	\begin{defn}[\textbf{Viscosity Solution II}]\label{viscosity-operator}
		A bounded upper semicontinuous (resp.\ lower semicontinuous) function $u$ is a viscosity subsolution (resp.\ supersolution) associated with the operator $\A_\epsilon$ if the following holds: for every $(x_0,t_0)\in \Omega \times [0,T]$ and every test function $\ph\in C^2(B_R(x_0)\times [t_0-R,t_0+R])$ for some $R>0$, such that $\ph(x_0,t_0)=u(x_0,t_0)$ and $\ph(x,t)>u(x,t)$ (resp.\ $\ph(x,t)<u(x,t)$) in $B_R(x_0)\times [t_0-R,t_0]$, we have
		\[
		\ph(x_0,t_0+\epsilon^2/2)\geq \A_\epsilon[\ph](x_0,t_0)+o(\epsilon^2),
		\quad \text{as }\epsilon\to 0^+, \qquad
		(\mbox{resp. } \ph(x,t+\epsilon^2/2)\leq \A_\epsilon[\ph](x,t)+o(\epsilon^2)).
		\]
	\end{defn}
	
	Here by $\ph(x_0,t_0+\epsilon^2/2)\geq \A_\epsilon[\ph](x_0,t_0)+o(\epsilon^2)$ (resp. $\ph(x,t+\epsilon^2/2)\leq \A_\epsilon[\ph](x,t)+o(\epsilon^2) $) we mean that
	$$
	0\geq \limsup_{\epsilon \to 0} \frac{ \A_\epsilon[\ph](x_0,t_0) - \ph(x_0,t_0+\epsilon^2/2)}{\epsilon^2},
	\qquad
	\left(\mbox{resp. } 0\leq \liminf_{\epsilon \to 0} \frac{ \A_\epsilon[\ph](x_0,t_0) - \ph(x_0,t_0+\epsilon^2/2)}{\epsilon^2}\right).
	$$
	The proof of the equivalence between both definitions follows from our asymptotic expansions (see Section \ref{sec3}) following the same
	arguments used in \cite{delTeso_Rossi_2024}.  
	
	Now, let us state the definition of being a viscosity solution to the boundary value problem
	\begin{equation}\label{problem-A-1}
		\left\{\begin{array}{rlll}
			\left( \partial_t u(x,t) \right)^{p-1} &=& \plap u(x,t), & \mbox{in } \Omega\times (0,\infty),\\[4pt]
			u(x,0)&=&u_0(x), & \mbox{in } \Omega,\\[4pt]
			u(x,t)&=&g(x), & \mbox{on } \partial \Omega\times (0,\infty).
		\end{array}\right.
	\end{equation}
	
	\begin{defn}[\textbf{Viscosity Solution III}]\label{viscosity-dirichlet}
		A bounded upper semicontinuous (resp.\ lower semicontinuous) function $u:\overline{\Omega}\times [0,\infty)\to \RR$ is a viscosity subsolution (resp.\ supersolution) to the Dirichlet problem \eqref{problem-A-1} if:
		\begin{enumerate}[label=\roman*)]
			\item $u$ is a viscosity subsolution (resp.\ supersolution) of \eqref{equation.intro} in the sense of Definition~\ref{viscosity-equation}.
			\item $u(x,0)\leq u_0(x)$ (resp.\ $u(x,0)\geq u_0(x)$) in $\Omega$, and $u(x,t)\leq g(x)$ (resp.\ $u(x,t)\geq g(x)$) in $\partial \Omega \times (0,\infty)$.
		\end{enumerate}
		
		Moreover, a bounded continuous function $u$ is a viscosity solution of \eqref{problem-A-1} if it is both a viscosity subsolution and a viscosity supersolution.
	\end{defn}
	
	For this boundary value problem, \eqref{problem-A-1}, a comparison principle for
	viscosity solutions is established in \cite{Hynd_Lindgren_2017}. From this comparison principle, in the same reference,
	the authors obtain 
	existence and uniqueness of a viscosity solution for continuous and compatible
	data $u_0$ and $g$. 
	
	The definition of viscosity solutions for the problem in the whole space is analogous. One only needs to replace $\Omega$ by $\Rd$ and remove the condition involving the function $g$. At this point we highlight again that uniqueness of viscosity solutions in the whole space is open.

\section*{Acknowledgments}
	
	This work was started during a visit of J. D. Rossi to Univ. Aut\'{o}noma de Madrid
	and 
	ICMAT (Spain)  and continued with a visit of 
	C. Fuertes Mor\'an to Univ. Torcuato di Tella (Argentina). The authors are grateful to both
	institutions for the friendly and stimulating working atmosphere. 
	
	F. Del Teso is funded by the Ram\'on y Cajal contract reference RYC2020-029589-I
and the research projects PID2021-127105NB-I00, CNS2024-154515 of the AEI, (Spain).

C. Fuertes was partially supported by the Projects PID2020-113596GB-I00, PID2023-150166NB-I00 and PID2021-127105NB-100 (Spanish Ministry of Science and Innovation) and acknowledge the financial support from the Spanish Ministry of Science and Innovation, through the ``Severo Ochoa Programme for Centres of Excellence in R\&D'' (CEX2019-000904-S and CEX2023-001347-S) and by the European Union’s Horizon 2020 research and innovation programme under the Marie Sk\l odowska-Curie grant agreement no. 777822.
	
	 J. D. Rossi was partially supported by 
			CONICET PIP GI No 11220150100036CO
(Argentina), PICT-03183 (Argentina) and UBACyT 20020160100155BA (Argentina).

	
	\noindent F\'elix del Teso and Carlos Fuertes Mor\'an. 
	
	\noindent Departamento de Matem\'{a}ticas, Universidad Aut\'{o}noma de Madrid,\\
	ICMAT - Instituto de Ciencias Matem\'{a}ticas, CSIC-UAM-UC3M-UCM, \\
	Campus de Cantoblanco, 28049 Madrid, Spain.
	
	\noindent E-mail:\texttt{~felix.delteso@uam.es }
	Web-page:  \texttt{https://sites.google.com/view/felixdelteso/home} 
	
	\noindent E-mail:\texttt{~carlos.fuertesm@uam.es}
	
	\bigskip
	
	\noindent Julio D. Rossi.
	
	\noindent Departamento de Matemáticas y Estadística, 
	Universidad Torcuato di Tella, \\
	Campus Alcorta - Av. Figueroa Alcorta 7350, Buenos Aires, Argentina.
	
	\noindent E-mail:\texttt{~julio.rossi@utdt.edu}		
	
\end{document}